\documentclass{amsart}
\usepackage{amssymb,amsmath,amsthm}
\usepackage{amsaddr}
\usepackage{amsfonts}
\usepackage{graphicx}
\usepackage{epstopdf}
\usepackage{caption}
\usepackage{subcaption}
\usepackage{hyperref}
\usepackage{color}
\usepackage{marginnote}
\usepackage{float}
\usepackage{graphicx}
\usepackage{comment}
\usepackage{tikz-cd}
\usetikzlibrary{arrows}
\usetikzlibrary{intersections,positioning}
\tikzset{>=latex}
\usepackage{lipsum}%
\usepackage{a4wide}
\allowdisplaybreaks[4]

\newcommand{\C}{\mathbb{C}}

\newcommand{\R}{\mathbb{R}}
\newcommand{\Z}{\mathbb{Z}}

\newcommand{\N}{\mathbb{N}}

\newcounter{newcounter}[section]
\numberwithin{equation}{section}
\numberwithin{newcounter}{section}
\numberwithin{figure}{section}
\numberwithin{footnote}{section}

\newcommand{\authorfootnotes}{\renewcommand\thefootnote{\@fnsymbol\c@footnote}}%

\newtheorem{thm}[newcounter]{Theorem}
\newtheorem{defi}[newcounter]{Definition}
\newtheorem*{defiun}{Definition}
\newtheorem{prop}[newcounter]{Proposition}
\newtheorem{lem}[newcounter]{Lemma}
\newtheorem{cor}[newcounter]{Corollary}
\newtheorem{rem}[newcounter]{Remark}
\newtheorem{exam}[newcounter]{Example}

\title{Obstructions to Symplectic Embeddings between Toric Domains}

\author{Marcelo Miranda}
\address{Instituto de Matem\'atica e Estat\'istica, Universidade de S\~ao Paulo}
\email{marcelo.miranda@ime.usp.br}

\author{Pedro A. S. Salom\~ao}
\address{Shenzhen International Center for Mathematics, SUSTech}
\email{psalomao@sustech.edu.cn}

\author{J. Trejos}
\address{Shenzhen International Center for Mathematics, SUSTech}
\email{yostrejos@sustech.edu.cn}

\begin{document}

\begin{abstract}
We study symplectic embeddings between four-dimensional toric domains using embedded contact homology. Extending Hutchings' criterion for embeddings between convex toric domains, we obtain obstructions for embeddings from convex toric domains into concave toric domains, from semi-weakly convex toric domains into convex toric domains, and from concave toric domains into concave toric domains. The obstructions are formulated in terms of factorizations of convex, concave, and semi-weakly convex generators together with combinatorial constraints relating their ECH indices and actions. As applications, we recover the sharp obstruction for symplectic embeddings of a polydisk into a union of cylinders, and obtain sharp results for embeddings of certain quadrilateral toric domains into balls and ellipsoids. We also show that these obstructions are strictly stronger than those coming only from ECH capacities.
\end{abstract}

\maketitle

\tableofcontents
 
\section{Introduction}

A basic problem in symplectic geometry is to determine when one
symplectic manifold admits a symplectic embedding into another. In
dimension four, this problem exhibits strong rigidity phenomena, and
symplectic embeddings are often constrained by invariants arising from
pseudoholomorphic curve theory.

For toric domains in $\C^2$, embedded contact homology (ECH) provides
some of the strongest known obstructions. The most familiar examples are
the ECH capacities introduced by Hutchings, which give sharp results for
several embedding problems involving ellipsoids and polydisks. In the
toric setting, these capacities admit combinatorial descriptions in terms
of lattice paths associated to the moment region; see, for instance,
\cite{ConcaveCapacities,hutchings2011quantitative,
miranda2025embedded,trejos2024symplecticembeddingstoricdomains}.

Beyond ECH capacities, Hutchings \cite{hutchings2016beyond} introduced
finer obstructions for symplectic embeddings between convex toric domains
using convex generators and their factorizations. These obstructions
detect phenomena which are invisible to ECH capacities alone and can
therefore produce strictly stronger embedding restrictions.

In this paper, we extend Hutchings' approach to embeddings involving
concave and semi-weakly convex toric domains. More precisely, we establish
obstruction criteria for symplectic embeddings from convex toric domains
into concave toric domains, from semi-weakly convex toric domains into
convex toric domains, and between concave toric domains. The resulting
obstructions are expressed in terms of factorizations of convex, concave,
and semi-weakly convex generators, together with combinatorial conditions
relating their ECH indices and actions.

The concave case introduces several new difficulties. The combinatorics
of the generators changes, the inequalities relating factorizations are
no longer symmetric, and additional restrictions arise when a concave
domain occurs as the target. In the semi-weakly convex case, adapted
generators arise naturally from the geometry near the boundary of the
moment region. We also obtain a description of the ECH capacities of a
family of semi-weakly convex toric domains in terms of these adapted
generators.

The proofs of Theorems~\ref{thm: criterion_convex_concave} and
\ref{thm:criterion_convex_convex} are independent of the unpublished
\cite{UserGuide}. By contrast,
Theorem~\ref{thm:criterion_concave_concave} uses one finite filtered
statement from \cite{UserGuide} concerning the sum of elliptic generators of the ECH
chain complex of a concave toric boundary. We state this input explicitly
immediately before the proof of the theorem and use it as a black box; no
part of the proofs of the first two obstruction theorems depends on
\cite{UserGuide}.

As applications, we recover the sharp obstruction for symplectic
embeddings of a polydisk into a union of cylinders previously proved by
Gutt and Hutchings \cite{Gutt_2018}. We also obtain sharp embedding
results for certain quadrilateral toric domains into balls and
ellipsoids. These examples show that the factorization obstructions
developed here are strictly stronger than the obstructions obtained from
ECH capacities alone.

The paper is organized as follows. In Section~2, we review toric domains,
generators, and ECH capacities. Section~3 contains the main obstruction
theorems and their applications. In Section~4, we recall the necessary
background on embedded contact homology and ECH cobordism maps. Section~5
constructs the Morse--Bott approximations and $L$-nice perturbations used
throughout the paper. In Section~6, we study the ECH differential for
these perturbations and establish the required behavior of the
hyperbolic labels. The proofs of the main obstruction theorems are given
in Section~7.

\section{Toric Domains}

Consider the space $\mathbb{C}^2$ with coordinates $(z_1,z_2)=(x_1+iy_1,x_2+iy_2)$ and standard symplectic form
$
\omega_{\mathrm{std}}:=dx_1\wedge dy_1+dx_2\wedge dy_2.
$
Given a domain $\Omega\subset\mathbb{R}_{\geq0}^2$, we define the associated toric domain by
$$
X_\Omega=\left\{(z_1,z_2)\in\mathbb{C}^2:(\pi|z_1|^2,\pi|z_2|^2)\in\Omega\right\}.
$$

Toric domains play a central role in four-dimensional symplectic embedding problems. In this section we review the combinatorial description of ECH capacities and generators associated to toric domains.

\subsection{Examples of toric domains}

Several important examples of toric domains have been extensively studied. Let $a,b,c,d>0$.

\begin{enumerate}

\item\label{exa:ellipsoid} The \emph{ellipsoid}
$$
E(a,b):=\left\{(z_1,z_2)\in\mathbb{C}^2:\frac{\pi|z_1|^2}{a}+\frac{\pi|z_2|^2}{b}\leq1\right\},
$$
corresponding to the moment region
$$
\Omega=\Omega_{a,b}=\left\{(x,y)\in\mathbb{R}_{\geq0}^2:\frac{x}{a}+\frac{y}{b}\leq1\right\}.
$$
The special case $B(a):=E(a,a)$ is the four-dimensional ball.

\item\label{exa:polydisk} The \emph{polydisk}
$$
P(a,b):=\left\{(z_1,z_2)\in\mathbb{C}^2:\pi|z_1|^2\leq a,\ \pi|z_2|^2\leq b\right\}.
$$

\item\label{exa:unionofcylinders} The \emph{non-disjoint union of cylinders}
$$
Z(a,b):=\left\{\pi|z_1|^2\leq a\right\}\cup\left\{\pi|z_2|^2\leq b\right\}.
$$

\item\label{exa:quadrilateral} \emph{Convex quadrilateral toric domains.} Denote by $\Omega_{a,b,c,d}$ the convex hull of the points
$
(0,0), (a,0), (c,d)$ and $(0,b)
$
in $\mathbb{R}_{\geq0}^2$, and assume
$
0<c\leq a
\qquad\text{and}\qquad
bc+ad>ab.
$
The second inequality means that $(c,d)$ lies strictly above the
segment joining $(a,0)$ to $(0,b)$, so that the convex hull has the
stated quadrilateral upper boundary.
The associated toric domain
$
Q_{a,b,c,d}:=X_{\Omega_{a,b,c,d}}
$
will be called a \emph{convex quadrilateral toric domain}.

\end{enumerate}

\subsection{Convex, concave and semi-weakly convex toric domains}

We now introduce the classes of toric domains considered in this paper.

\begin{defi}\label{defi_convex_concave_toric_domain}

Let
$
\Omega=\left\{(x,y)\in\mathbb{R}_{\geq0}^2:0\leq x\leq a,\ 0\leq y\leq f(x)\right\},
$
where
$
f:[0,a]\to[0,+\infty)
$
is continuous and satisfies $f(0)=b>0$. We say that the toric domain $X_\Omega$ is

\begin{itemize}

\item[(1)] {\bf convex} if $f$ is decreasing and concave;

\item[(2)] {\bf concave} if $f$ is strictly decreasing, convex and satisfies $f(a)=0$;

\item[(3)] {\bf semi-weakly convex} if $f$ is concave, not necessarily decreasing.

\end{itemize}

We denote by $\partial^+\Omega\subset\mathbb{R}_{\geq0}^2$ the union of the graph of $f$ with the segment
$$
\{a\}\times[0,f(a)].
$$

\end{defi}

The polydisk is convex, while the ellipsoid is both convex and concave. The quadrilateral toric domain $Q_{2,1,1,2}$ is semi-weakly convex.

\subsection{ECH capacities}

Embedded contact homology provides a sequence of symplectic capacities
$$
0=c_0(X,\omega)\leq c_1(X,\omega)\leq c_2(X,\omega)\leq\cdots\leq\infty
$$
associated to a symplectic four-manifold $(X,\omega)$.

These capacities satisfy the following properties:

\begin{enumerate}

\item {\bf Monotonicity:} if there exists a symplectic embedding
$
(X,\omega)\hookrightarrow(X',\omega'),
$
then
$
c_k(X,\omega)\leq c_k(X',\omega')
$
for every $k\geq0$;

\item {\bf Conformality:} for every $r>0$, we have 
$
c_k(X,r\omega)=r\,c_k(X,\omega);
$

\item {\bf Disjoint union:}
$$
c_k\left(\bigsqcup_{i=1}^n(X_i,\omega_i)\right)
=
\max_{k_1+\cdots+k_n=k,k_i\geq 0}\sum_{i=1}^n c_{k_i}(X_i,\omega_i).
$$

\end{enumerate}

For ellipsoids,
$
c_k(E(a,b))=N(a,b)_k,
$
where $N(a,b)$ denotes the sequence of nonnegative integer linear combinations of $a$ and $b$, arranged in nondecreasing order.

For polydisks,
\begin{equation}
\label{eq:cap_poli}
c_k(P(a,b))
=
\min\{am+bn:(m,n)\in\N^2,\ (m+1)(n+1)\geq k+1\}.
\end{equation}

ECH capacities of toric domains admit a combinatorial description in terms of generators, which we now recall.

\subsection{Integral paths and generators}

\begin{defi}[Integral paths]\label{def_integral_path_generators}
An {\bf integral path} $\Lambda\hookrightarrow\mathbb{R}^2$ is an injective piecewise linear path formed by line segments joining points in the lattice $\mathbb{Z}^2$, starting at $(x(\Lambda),0)$ and ending at $(0,y(\Lambda))$.

A vertex of $\Lambda$ is either an endpoint of $\Lambda$ or an interior point where $\Lambda$ changes direction. An edge of $\Lambda$ is a non-trivial line segment joining consecutive vertices.

The integral path consisting only of the point $(0,0)$ is called trivial. Otherwise it is called non-trivial.

A non-trivial integral path is oriented from $(x(\Lambda),0)$ to $(0,y(\Lambda))$. This orientation induces an ordering and an orientation of the edges
$
v_1,\dots,v_{k(\Lambda)}.
$
\end{defi}

We say that an integral path $\Lambda$ is

\begin{itemize}

\item[(1)] {\bf convex} if

\begin{itemize}

\item $x(\Lambda),y(\Lambda)\geq0$;

\item $\Lambda$ is contained in the rectangle
$
[0,x(\Lambda)]\times[0,y(\Lambda)];
$

\item
$
v_i\times v_{i+1}>0, i=1,\dots,k(\Lambda)-1;
$

\item the compact region enclosed by $\Lambda$ and the coordinate axes is convex.

\end{itemize}

\item[(2)] {\bf concave} if

\begin{itemize}

\item $x(\Lambda),y(\Lambda)\geq0$;

\item $\Lambda$ is contained in the rectangle
$
[0,x(\Lambda)]\times[0,y(\Lambda)];
$

\item no edge is parallel to a coordinate axis.

\item
$
v_i\times v_{i+1}<0, i=1,\dots,k(\Lambda)-1;
$

\item the unbounded closed region bounded by $\Lambda$ and the coordinate axes is convex;

\end{itemize}

\item[(3)] {\bf semi-weakly convex} if

\begin{itemize}

\item $x(\Lambda)\geq0$;

\item $\Lambda$ is contained in the strip
$
0\leq x\leq x(\Lambda);
$

\item
$
v_i\times v_{i+1}>0, i=1,\dots,k(\Lambda)-1;
$

\item unless $\Lambda$ consists of a single vertical edge, the last edge of $\Lambda$ is not parallel to the $y$-axis.

\end{itemize}

\item[(4)] $\Lambda$ is {\bf horizontally $n$-adapted} if
$
(-n,1)\times v\geq0
$
for every edge $v$ of $\Lambda$.

\item[(5)] $\Lambda$ is {\bf vertically $m$-adapted} if
$
v\times(-1,m)\geq0
$
for every edge $v$ of $\Lambda$.

\item[(6)] $\Lambda$ is {\bf $(n,m)$-adapted} if it is both horizontally $n$-adapted and vertically $m$-adapted.

\end{itemize}

\begin{defi}[Multiplicity of an edge]

Let $v$ be an edge of an integral path $\Lambda$. The {\bf multiplicity} of $v$ is the largest positive integer $m(v)$ such that
$
v=m(v)v'
$
for some primitive vector
$
v'\in\mathbb{Z}^2\setminus\{0\}.
$
\end{defi}

\begin{defi}[Generators]

A {\bf generator} is an integral path whose edges are labeled by either `e' or `h'. We say that a generator is

\begin{itemize}

\item[(1)] {\bf convex} if the underlying integral path is convex and every horizontal or vertical edge is labeled `e';

\item[(2)] {\bf concave} if the underlying integral path is concave;

\item[(3)] {\bf semi-weakly convex} if the underlying integral path is semi-weakly convex;

\item[(4)] {\bf $(n,m)$-adapted} if the underlying integral path is $(n,m)$-adapted and every edge parallel to $(-n,1)$ or $(-1,m)$ is labeled `e';

\item[(5)] {\bf elliptic} if every edge is labeled `e'.

\end{itemize}

\end{defi}

\begin{defi}

Let $\Lambda$ be a generator and let $v$ be an edge of $\Lambda$ with multiplicity $m(v)$.

\begin{enumerate}

\item If $v$ is labeled `e', define
$
e(v):=m(v)$ and $h(v):=0.
$

\item If $v$ is labeled `h', define
$
e(v):=m(v)-1$ and $h(v):=1.
$

\end{enumerate}

We define
$$
e(\Lambda):=\sum_{i=1}^{k(\Lambda)}e(v_i),\qquad h(\Lambda):=\sum_{i=1}^{k(\Lambda)} h(v_i).
$$
Finally, we define the multiplicity of $\Lambda$ by
$$
m(\Lambda) := \sum_{i=1}^{k(\Lambda)}m(v_i)=e(\Lambda) + h(\Lambda).
$$
\end{defi}

\begin{defi}
\label{def_integral_paths_semi_weakly_convex}

Let $\Lambda\hookrightarrow\mathbb{R}^2$ be an integral path.

\begin{itemize}

\item[(1)] If $\Lambda$ is convex, we define $B_\Lambda$ as the compact region enclosed by $\Lambda$ and the coordinate axes.

\item[(2)] If $\Lambda$ is concave, let $B$ be the compact region enclosed by $\Lambda$ and the coordinate axes, and define
$$
B_\Lambda:=B\setminus\Lambda.
$$

\item[(3)] Suppose that $\Lambda$ is semi-weakly convex and let
$$
x_0:=\min\{x\in[0,x(\Lambda)]:\exists\, y\geq 0 \mbox{ such that}\, (x,y)\in \Lambda)\}.
$$ 
We define $B_\Lambda^\pm$ as follows:
\begin{itemize}
\item If $x_0=0$, define $B_\Lambda^+$ as the compact region enclosed by $\Lambda$ and the coordinate axes, and set
$
B_\Lambda^-:=\emptyset.
$

\item If $0<x_0<x(\Lambda)$, write
$
\Lambda^+=\Lambda\cap\{x_0\leq x\leq x(\Lambda),\ y\geq0\},
$
and
$
\Lambda^-=\Lambda\cap\{0\leq x\leq x_0,\ y\leq0\}.
$

Define $B_\Lambda^+$ as the compact region enclosed by $\Lambda^+$ and the segment
$
[x_0,x(\Lambda)]\times\{0\},
$
and define
$
B_\Lambda^-
=
B\setminus\left(\Lambda^-\cup([0,x_0]\times\{0\})\right),
$
where $B$ is the compact region enclosed by $\Lambda^-$ and the coordinate axes.

\item If $0<x_0=x(\Lambda)$, write
$
\Lambda^-=\Lambda\cap\{0\leq x\leq x_0,\ y\leq0\}.
$
Define
$
B_\Lambda^+:=\{(x(\Lambda),0)\},
$
and
$
B_\Lambda^-
=
B\setminus\left(\Lambda^-\cup([0,x(\Lambda)]\times\{0\})\right),
$
where $B$ is the compact region enclosed by $\Lambda^-$ and the coordinate axes.
\end{itemize}
\end{itemize}

\end{defi}

\subsection{ECH index and actions}

\begin{defi}[Length]

Let $X_\Omega$ be a convex, concave, or semi-weakly convex toric domain, and let $\Lambda$ be an integral path of the corresponding type.
For each edge $v_i$ of $\Lambda$, choose a point
$
p_{\Omega,i}\in\partial^+\Omega
$
such that:

\begin{enumerate}

\item If $X_\Omega$ is convex or semi-weakly convex, then $\partial^+\Omega$ is contained in the closed half-plane to the left of the oriented line through $p_{\Omega,i}$ parallel to $v_i$;

\item If $X_\Omega$ is concave, then $\partial^+\Omega$ is contained in the closed half-plane to the right of the oriented line through $p_{\Omega,i}$ parallel to $v_i$.

\end{enumerate}

The $\Omega$-length (or $\Omega$-action) of $\Lambda$ is defined by
$$
\ell_\Omega(\Lambda)
=
\sum_{i=1}^{k(\Lambda)}
p_{\Omega,i}\times v_i.
$$

This quantity does not depend on the choice of the points $p_{\Omega,i}$.

\end{defi}

\begin{exam}
Let $X_\Omega=E(a,b)$ and let $\Lambda$ be a convex integral path. There is a vertex
$p_n$ of $\Lambda$ such that the edges before and after $p_n$ have slopes respectively greater and smaller than $-b/a$.
Choosing $p_{\Omega,i}=(a,0)$ before $p_n$ and $p_{\Omega,i}=(0,b)$ after $p_n$, we obtain
\[
\ell_\Omega(\Lambda)=a\,y(p_n)+b\,x(p_n).
\]
Thus $\ell_\Omega(\Lambda)=c$, where $L=\{bx+ay=c\}$ is the supporting line of $\Lambda$ at $p_n$, oriented by $(-a,b)$.
\end{exam}

\begin{exam}

Let $X_\Omega=P(a,b)$ be a polydisk, and let $\Lambda$ be a convex integral path. With a similar computation as in the case of the ellipsoid, we obtain that
$$
\ell_\Omega(\Lambda)
=
b\,x(\Lambda)+a\,y(\Lambda).
$$

\end{exam}

\begin{defi}[ECH index]

Let $B\subset\R^2$ be bounded and denote by
$$
\mathcal L(B):=|B\cap\Z^2|
$$
the number of lattice points in $B$.

\begin{itemize}

\item[(1)] If $\Lambda$ is a convex generator, define
$$
I(\Lambda)
:=
2(\mathcal L(B_\Lambda)-1)-h(\Lambda).
$$

\item[(2)] If $\Lambda$ is a concave generator, define
$$
I(\Lambda)
:=
2\mathcal L(B_\Lambda)+h(\Lambda).
$$

\item[(3)] If $\Lambda$ is a semi-weakly convex generator, define
$$
I(\Lambda)
=
2\left(\mathcal L(B_\Lambda^+)-\mathcal L(B_\Lambda^-)-1\right)-h(\Lambda).
$$

\end{itemize}

\end{defi}

\begin{defi}[Pick's formula] Let $\Lambda\subset \R^2_{\geq 0}$ be a convex or a concave integral path. Denote by $L(K_\Lambda)$ the number of lattice points in the compact region $K_\Lambda\subset \R_{\geq 0}^2$ enclosed by $\Lambda$ and the axes, including the lattice points in the boundary of $K_\Lambda$. Let $A(\Lambda)$ be the area of $K_\Lambda$. Then
$$
2A(\Lambda) = 2L(K_\Lambda) - m(\Lambda)-x(\Lambda) - y(\Lambda) -2.
$$
We conclude from Pick's formula that if $\Lambda$ is a convex generator, then
$$
I(\Lambda) = 2A(\Lambda) +x(\Lambda)+y(\Lambda) +e(\Lambda),
$$
and if $\Lambda$ is a concave generator, then
$$
I(\Lambda) = 2A(\Lambda) + x(\Lambda) + y(\Lambda) -e(\Lambda).
$$
In particular, if $\Lambda$ is an elliptic convex or concave generator and $d\in \Z_{>0}$, then
$$
I(\Lambda^d) = 2(d^2-d)A(\Lambda) + dI(\Lambda).
$$
\end{defi}

\begin{lem}[Signed Pick's formula]
\label{lem:signed-pick}
Let $\Lambda$ be an $(n,m)$-adapted semi-weakly convex integral path. Then
$$
2\left(
A(B_\Lambda^+)
-
A(B_\Lambda^-)
\right)
+
x(\Lambda)
+
y(\Lambda)
+
m(\Lambda)
=
2\left(
\mathcal L(B_\Lambda^+)
-
\mathcal L(B_\Lambda^-)
-
1
\right).
$$
\end{lem}

\begin{proof}
Write
$
x=x(\Lambda),
y=y(\Lambda),
m_\Lambda=m(\Lambda).
$

We first treat the degenerate case $x=0$. Then $\Lambda$ is a single
vertical segment. By the convention in
Definition~\ref{def_integral_paths_semi_weakly_convex},
$
B_\Lambda^-=\emptyset
$
and $B_\Lambda^+$ is the segment on the $y$-axis bounded by $\Lambda$
and the origin. Hence
$
A(B_\Lambda^+)=0,
$
$
\mathcal L(B_\Lambda^+)=y+1,
$
and
$
m_\Lambda=y.
$
Therefore
\[
2\bigl(A(B_\Lambda^+)-A(B_\Lambda^-)\bigr)
+x+y+m_\Lambda
=
2y
=
2\bigl(\mathcal L(B_\Lambda^+)-\mathcal L(B_\Lambda^-)-1\bigr),
\]
so the formula holds.

Henceforth assume $x>0$. Choose an integer $N>0$ sufficiently large so that the vertical translate
$
\Lambda_N
=
\Lambda+(0,N)
$
is contained in the closed upper half-plane and $N+y\geq0$.

Let $K_N$ be the compact lattice polygon bounded by $\Lambda_N$, the
vertical segment joining $(0,0)$ to $(0,N+y)$, the horizontal segment
joining $(0,0)$ to $(x,0)$, and the vertical segment joining $(x,0)$ to
$(x,N)$.

The signed area enclosed by $\Lambda$ and the coordinate axes is
$
A(B_\Lambda^+)
-
A(B_\Lambda^-).
$
Translating $\Lambda$ vertically by $N$ adds the rectangle
$[0,x]\times[0,N]$. Therefore,
$$
A(K_N)
=
Nx
+
A(B_\Lambda^+)
-
A(B_\Lambda^-).
$$

We next compute the number of lattice points on the boundary of $K_N$.
The contribution from the edges of $\Lambda_N$ is $m_\Lambda$. The two
vertical segments have lattice lengths $N$ and $N+y$, while the horizontal
segment has lattice length $x$. Hence,
$$
\#\left(\partial K_N\cap\mathbb Z^2\right)
=
m_\Lambda+N+(N+y)+x
=
2N+x+y+m_\Lambda.
$$

We now compare the lattice points of $K_N$ with those of
$B_\Lambda^+$ and $B_\Lambda^-$. For $j\in\{0,\ldots,x\}$, write
$V_j=\{j\}\times\mathbb Z$ and set $k_j=\#(K_N\cap V_j)$,
$k_j^+=\#(B_\Lambda^+\cap V_j)$, and
$k_j^-=\#(B_\Lambda^-\cap V_j)$.

By the definitions of $B_\Lambda^+$ and $B_\Lambda^-$, vertical translation
by $N$ gives
$
k_j
=
N+k_j^+-k_j^-.
$

The boundary conventions in Definition~\ref{def_integral_paths_semi_weakly_convex} ensure that this identity also
holds when the vertical line contains a lattice point on $\Lambda$ or on
the $x$-axis. Summing over $j=0,\ldots,x$, we obtain
$$
\mathcal L(K_N)
=
N(x+1)
+
\mathcal L(B_\Lambda^+)
-
\mathcal L(B_\Lambda^-).
$$

Applying Pick's formula to the lattice polygon $K_N$ gives
$$
2A(K_N)
=
2\mathcal L(K_N)
-
\#\left(\partial K_N\cap\mathbb Z^2\right)
-
2.
$$

Substituting the formulas above, we obtain
$$
\begin{aligned}
&
2Nx
+
2\left(
A(B_\Lambda^+)
-
A(B_\Lambda^-)
\right)
\\
&\qquad =
2N(x+1)
+
2\mathcal L(B_\Lambda^+)
-
2\mathcal L(B_\Lambda^-)
-
\left(
2N+x+y+m_\Lambda
\right)
-
2.
\end{aligned}
$$

Cancelling the term $2Nx$ from both sides gives
$$
2\left(
A(B_\Lambda^+)
-
A(B_\Lambda^-)
\right)
+
x+y+m_\Lambda
=
2\left(
\mathcal L(B_\Lambda^+)
-
\mathcal L(B_\Lambda^-)
-
1
\right).
$$

This proves the lemma. \end{proof}

With these definitions, we can describe the ECH capacities of toric domains combinatorially.

\begin{thm}[Choi, Cristofaro-Gardiner, Frenkel, Hutchings, Ramos {\cite[Theorem 1.21]{ConcaveCapacities}}]
\label{thm:concave_capacities}

Suppose that $X_\Omega$ is a concave toric domain. Then
$
c_k(X_\Omega)
=
\max\{l_\Omega(\Lambda):\mathcal L(B_\Lambda)=k\},
$
where $\Lambda$ runs over all concave integral paths $\Lambda$.

\end{thm}

The convex case is due to Hutchings.

\begin{thm}[Hutchings {\cite[Proposition 5.6]{hutchings2011quantitative}}]
\label{thm:convex_capacities}

Suppose that $X_\Omega$ is a convex toric domain. Then
$
c_k(X_\Omega)
=
\min\{l_\Omega(\Lambda):\mathcal L(B_\Lambda)=k+1\},
$
where $\Lambda$ runs over all convex integral paths $\Lambda$.

\end{thm}

\begin{defiun}[Minimal convex generator]
Following \cite[Definition~1.15]{hutchings2016beyond}, let $X_{\Omega'}$
be a convex toric domain and let $\Lambda'$ be a convex generator with
$I(\Lambda')=2k$. We say that $\Lambda'$ is \emph{minimal for
$X_{\Omega'}$} if the following three conditions hold:
\begin{enumerate}
\item[(a)] every edge of $\Lambda'$ is labeled $e$;
\item[(b)] $\ell_{\Omega'}(\Lambda')=c_k(X_{\Omega'})$;
\item[(c)] $\Lambda'$ uniquely minimizes $\ell_{\Omega'}$ among all
convex generators of ECH index $2k$.
\end{enumerate}
\end{defiun}

The next theorem gives an analogous formula for semi-weakly convex toric domains.

\begin{thm}
\label{thm:semiweak_capacities}

Let $X_\Omega$ be a semi-weakly convex toric domain with defining function
$
f:[0,a]\to[0,+\infty)
$
as in Definition \ref{defi_convex_concave_toric_domain}. The following statements hold:

\begin{itemize}

\item[(i)]
Suppose that $f$ is differentiable at $0$ and $a$, satisfies
$
f(a)=0,
$

$
f'(0)\leq -m,$ and $nf'(a)\geq -1
$
for some
$
m\in\Z,
n\in\Z_{\geq0}.
$ If $n>0$ we require $mn<1$.

Then
$$
c_k(X_\Omega)
=
\min\left\{
l_\Omega(\Lambda):
\mathcal L(B_\Lambda^+)=k+1
\right\},
\qquad \forall k\geq0,
$$
where $\Lambda$ runs over all $(n,m)$-adapted semi-weakly convex integral paths $\Lambda$ satisfying
$
y(\Lambda)\geq0.
$

\item[(ii)]
Suppose that $f$ is differentiable at $0$, satisfies
$
f(a)>0,$ and $
f'(0)\leq -m
$
for some
$
m\in\Z.
$
Then the same formula holds, where $\Lambda$ runs over all $(0,m)$-adapted semi-weakly convex integral paths satisfying
$
y(\Lambda)\geq0.
$

\end{itemize}

\end{thm}

The proof of Theorem \ref{thm:semiweak_capacities} is given at the end of Section \ref{sec:convexpert}.

\subsection{Products and factorizations of generators}

We now introduce the product operation for generators and the corresponding notion of factorization used in the obstruction theorems.

\begin{defi}[Product of generators]

Let $\Lambda'$ and $\Lambda''$ be convex, concave or semi-weakly convex generators of the same type. Assume that $\Lambda'$ and $\Lambda''$ do not contain parallel edges both labeled `h'.
Reordering the edges of $\Lambda'$ and $\Lambda''$, we obtain a convex, concave or semi-weakly convex generator, respectively, denoted by
$
\Lambda'\cdot\Lambda''.
$

If $v'$ and $v''$ are parallel edges of $\Lambda'$ and $\Lambda''$, respectively, then the corresponding edge
$
v=v'+v''
$
of $\Lambda'\cdot\Lambda''$ is labeled `e' if both $v'$ and $v''$ are labeled `e'. Otherwise, $v$ is labeled `h'.

Edges belonging to only one factor inherit their labels from the corresponding factor.

\end{defi}

\begin{rem}

The product operation is associative and commutative.

\end{rem}

\begin{defi}

Let $\Lambda$ be a convex, concave or semi-weakly convex generator.

\begin{itemize}

\item[(1)] If $\Lambda$ is the product of $k$ copies of a generator $\Lambda'$, we write
$
\Lambda=(\Lambda')^k.
$

\item[(2)] Suppose
$
\Lambda=\Lambda_1^{d_1}\cdots\Lambda_r^{d_r}.
$
We say that this is a {\bf disjoint factorization} if the following condition holds: If $v_i$ and $v_j$ are parallel edges of $\Lambda_i$ and $\Lambda_j$, respectively, satisfying
$
e(v_i)>0
$ and $
e(v_j)>0,
$
then
$
i=j.
$

\end{itemize}

\end{defi}

\begin{defi}[Notation for generators]

Let $\Lambda$ be a convex, concave or semi-weakly convex generator and let $v$ be an edge of $\Lambda$ parallel to an irreducible vector $(-v_1,v_2)$ with $v_1>0$.

\begin{enumerate}

\item We write
$
v=e_{-v_1,v_2}^m
$
if
$
e(v)=m,$ and $
h(v)=0.
$

\item We write
$
v=h_{-v_1,v_2}^m
$
if
$
e(v)=m-1$ and $
h(v)=1.
$
\end{enumerate}

\end{defi}

\section{Main results}
\label{sec:Obstructions for Embeddings}

We provide obstructions for three different kinds of symplectic embedding problems involving toric domains. The main results concern symplectic embeddings from a convex toric domain into a concave toric domain, from a semi-weakly convex toric domain into a convex toric domain, and from a concave toric domain into a concave toric domain.  We order the statements in this way because the first two proofs are independent of the unpublished User Guide, while the third uses one finite filtered input from it. 

\begin{thm}[Convex into concave] \label{thm: criterion_convex_concave}
Suppose that there exists a symplectic embedding $X_{\Omega} \rightarrow X_{\Omega'}$ from a convex toric domain $X_{\Omega}$ into a concave toric domain $X_{\Omega'}$. For every $k\in \Z_{>0}$, there exist an elliptic concave generator $\Lambda'$ and an elliptic convex generator $\Lambda$, with respective disjoint factorizations $\Lambda' = \Lambda'^{d_1}_{1} \cdots \Lambda'^{d_r}_{r}$ and $\Lambda = \Lambda^{d_1}_{1} \cdots \Lambda^{d_r}_{r}$, so that
\begin{enumerate}
\item[(i)] $I(\Lambda) =  I(\Lambda') = 2k$,

\item[(ii)] $l_\Omega(\Lambda_i) \leq l_{\Omega'}(\Lambda'_i)$ for every $i$,

\item[(iii)] For every $S\subset\{1,\ldots,r\}$ and every
$0\leq d_i'\leq d_i$ for $i\in S$, one has
\[
I\left(\prod_{i\in S}(\Lambda_i)^{d_i'}\right)
=
I\left(\prod_{i\in S}(\Lambda_i')^{d_i'}\right).
\]

\item[(iv)] $x(\Lambda'_i) + y(\Lambda'_i) - 1 \leq x(\Lambda_i) + y(\Lambda_i)$ for every $i$.
    \end{enumerate}
\end{thm}

The following criterion generalizes \cite[Theorem~1.20]{hutchings2016beyond}
to symplectic embeddings of a semi-weakly convex toric domain into a
convex toric domain.

\begin{thm}[Semi-weakly convex into convex]\label{thm:criterion_convex_convex}
Let $X_{\Omega}$ be a semi-weakly convex toric domain with defining function $f:[0,a]\rightarrow [0,+\infty)$ as in Definition \ref{defi_convex_concave_toric_domain}. Let $X_{\Omega'}$ be a convex toric domain and assume that there exists a symplectic embedding $X_{\Omega}\rightarrow X_{\Omega'}$. The following statements hold:
\begin{enumerate}
\item If $f$ is differentiable at $0$ and at $a$, $f(a)=0$, $f'(0)\leq -m,$ and $n f'(a)\geq -1$ for some $m\in \Z$ and $n\in \Z_{\geq 0}$ with $mn<1$, then for every $k\in\mathbb{Z}_{\geq 0}$ and every convex generator $\Lambda'$ with $I(\Lambda')=2k$ which is minimal for $X_{\Omega'}$, there exist an $(n,m)$-adapted semi-weakly convex generator $\Lambda$, a disjoint  factorization $\Lambda=\Lambda_1^{d_1}\cdots\Lambda_r^{d_r}$ and a factorization $\Lambda'=\Lambda'^{d_1}_1\cdots\Lambda'^{d_r}_r$, such that   
\begin{enumerate}
\item[(i)] $I(\Lambda)=I(\Lambda')=2k$,

\item[(ii)] $l_{\Omega}(\Lambda_i)\leq l_{\Omega'}(\Lambda_i')$ for all $i$,

\item[(iii)] For every $S\subset\{1,\ldots,r\}$ and every
$0\leq d_i'\leq d_i$ for $i\in S$, one has
\[
I\left(\prod_{i\in S}(\Lambda_i)^{d_i'}\right)
=
I\left(\prod_{i\in S}(\Lambda_i')^{d_i'}\right).
\]

\item[(iv)] $x(\Lambda_i)+y(\Lambda_i)-h(\Lambda_i)/2\geq x(\Lambda_i')+y(\Lambda_i')+m(\Lambda_i')-1$.
\end{enumerate}

\item If $f$ is differentiable at $0$, $f(a)>0$, and $f'(0)\leq -m$ for some $m\in \Z$, then the same conclusions in part (1) hold after replacing $(n,m)$-adapted semi-weakly convex generator $\Lambda$ with $(0,m)$-adapted semi-weakly convex generator $\Lambda$. 
\end{enumerate}
\end{thm}

For symplectic embeddings between concave toric domains, we have the following criterion.

\begin{thm}[Concave into concave] \label{thm:criterion_concave_concave}
Suppose that there exists a symplectic embedding $X_{\Omega} \rightarrow X_{\Omega'}$ from a concave toric domain $X_\Omega$ into a concave toric domain $X_{\Omega'}$. For every $k\in \Z_{>0}$ and every elliptic concave generator $\Lambda$ with $I(\Lambda) = 2k$, there exist a factorization $\Lambda = \Lambda^{d_1}_1\cdots \Lambda^{d_r}_r$ of $\Lambda$, and an elliptic concave generator $\Lambda'$, with disjoint factorization $\Lambda' = \Lambda'^{d_1} _1 \cdots \Lambda'^{d_r} _r$, such that:
\begin{enumerate}
\item[(i)] $I(\Lambda) = I(\Lambda')=2k$,

\item[(ii)] $l_\Omega(\Lambda_i) \leq l_{\Omega'}(\Lambda'_i)$ for every $i$,

\item[(iii)] For every $S\subset\{1,\ldots,r\}$ and every
$0\leq d_i'\leq d_i$ for $i\in S$, one has
\[
I\left(\prod_{i\in S}(\Lambda_i)^{d_i'}\right)
=
I\left(\prod_{i\in S}(\Lambda_i')^{d_i'}\right).
\]

\item[(iv)] $x(\Lambda'_i) + y(\Lambda'_i) - 1 \leq x(\Lambda_i) + y(\Lambda_i) - m(\Lambda_i) $ for every $i$.
\end{enumerate} 
\end{thm}

The proofs of these theorems are left to Section \ref{sec:proofs}.

\subsection{Applications.} 
We start considering symplectic embeddings of the polydisk into the non-disjoint unions of two cylinders, recovering the $4$-dimensional version of a theorem of Gutt-Hutchings \cite[Prop.~1.20]{Gutt_2018}. 
\begin{prop}
        \label{app}
            Suppose that there exists a symplectic embedding of the polydisk $P(a,a)$ into $Z(b,b)$. Then $a\leq b$. 
\end{prop}

\begin{rem} The ECH capacities are not enough to prove Proposition \ref{app}.  Indeed, it follows from Theorem \ref{thm:concave_capacities} and the identity \eqref{eq:cap_poli} that 
$$
\begin{aligned}
c_k(Z(b,b)) & =b(0,2,3,4,5,6,7,8,9,10,11\dots),\\
c_k(P(a,a)) & =a(0,1,2,2,3,3,3,4,4,4,4\dots).
\end{aligned}
$$
Therefore, by comparing the ECH capacities of $P(a,a)$ and $Z(b,b)$, we can only achieve the weaker inequality $2a\leq 3b$.
\end{rem}

\begin{proof}[Proof of Proposition \ref{app}.] Suppose by contradiction that $0<b<a$ and there exists a symplectic embedding of $P(a,a)$ into $Z(b,b)$. We may take $a$ slightly smaller, still satisfying $a>b$, so that $P(a,a)$ symplectically embeds in the interior of $Z(b,b)$.  
For every $n$ sufficiently large, denote by $X_n=X_{\Omega_n}\subset Z(b,b)$ the concave toric domain associated with the region $\Omega_n\subset \R^2_{\geq0}$ bounded by the two line segments $(0,n) \to (b,b) \to (n,0)$ and the axes. It follows that there exists a symplectic embedding $P(a,a) \to X_n$ for every $n\geq N_0$ sufficiently large. 

Let $k\in \Z_{> 0}$. Since the condition $I(\Lambda')\leq 2k$ has finitely many solutions in the space of concave generators, there exists $n_k \geq N_0$ such that if  $\Lambda'$ is a concave generator satisfying $I(\Lambda') \leq 2k$, then 
$$
l_{\Omega_n}(\Lambda')=b(x(\Lambda')+y(\Lambda')), \quad \forall n\geq n_k.
$$

  Applying Theorem \ref{thm: criterion_convex_concave} to the embedding
$P(a,a)\to X_{n_k}$ for each $k\in\mathbb Z_{>0}$, we obtain a sequence of elliptic concave generators $\Lambda'_k$ with $I(\Lambda'_k)=2k$ and a sequence of convex generators $\Lambda_k$ with $I(\Lambda_k)=2k$, and disjoint factorizations 
  \begin{equation}
  \label{eq:decompositions}  \Lambda'_k=(\Lambda'_{k,1})^{d_{k,1}}\cdots (\Lambda'_{k,r_k})^{d_{k,r_k}} 
  \quad\mbox{ and } \quad \Lambda_k=(\Lambda_{k,1})^{d_{k,1}}\cdots (\Lambda_{k,r_k})^{d_{k,r_k}}
  \end{equation}  
such that conditions (i)-(iv) in Theorem \ref{thm: criterion_convex_concave} hold for every $k$. Indeed, condition (ii) is independent of $n_k$ and reads as
$$
a(x(\Lambda_{k,j}) + y(\Lambda_{k,j})) \leq b(x(\Lambda'_{k,j}) + y(\Lambda'_{k,j})) \quad \forall j=1,\ldots r_k.
$$

Conditions (ii) and (iv) give
$$
a\left(1-\frac{1}{x(\Lambda'_{k,j})+y(\Lambda'_{k,j})}\right)\leq b,
$$
for every $k\in \Z_{> 0}$ and every $j=1,\ldots, r_k$. The inequality above and the assumption $a>b$ imply that 
$\{x(\Lambda'_{k,j})+y(\Lambda'_{k,j})\}_{k,j}$ is uniformly bounded in $k$ and $j$, which in turn implies that $\{I(\Lambda'_{k,j})\}$ and $\{I(\Lambda_{k,j})\}$ are also uniformly bounded. Therefore, there exists a finite collection $\mathcal F'$ of concave generators and a finite collection $\mathcal F$ of convex generators such that every disjoint factorization of $\Lambda'_k$ and $\Lambda_k$ in \eqref{eq:decompositions} only includes products of generators in $\mathcal F'$ and $\mathcal F$, respectively.

Since $I(\Lambda_k)=I(\Lambda'_k)=2k\to+\infty$ as
$k\to+\infty$, while all factors belong to the finite collections
$\mathcal F$ and $\mathcal F'$, for every sufficiently large $k$ there
exist non-trivial generators
$
\widetilde\Lambda'=\Lambda'_{k,j}\in\mathcal F',
\qquad
\widetilde\Lambda=\Lambda_{k,j}\in\mathcal F,
$
such that $d_{k,j}\ge2$ for some $j$. Since the total generators
$\Lambda'_k$ and $\Lambda_k$ are elliptic, all their factors are
elliptic. Hence $\widetilde\Lambda'$ and $\widetilde\Lambda$ are
non-trivial elliptic generators, and therefore
$
e(\widetilde\Lambda')+e(\widetilde\Lambda)\ge2.
$

    Notice that for every elliptic convex or concave generator $\Lambda$, we have   
\begin{equation} \label{multiplicities}
    I(\Lambda^d)=2(d^2-d)A(\Lambda)+d I(\Lambda), \quad \forall d\in \Z_{>0},  
\end{equation}
    where $A(\Lambda)$ is the area enclosed by $\Lambda$ and the axes. 
    
    It follows from condition (iii) that $I(\tilde \Lambda')=I(\tilde \Lambda)$ and $A(\tilde\Lambda')=A(\tilde\Lambda)$. From $I(\tilde \Lambda')=I(\tilde \Lambda)$ and Pick's formula we obtain
    $$
    2A(\tilde \Lambda') + x(\tilde \Lambda') + y(\tilde\Lambda') -e(\tilde\Lambda')=2A(\tilde\Lambda) + x(\tilde \Lambda) + y(\tilde \Lambda) +e(\tilde \Lambda),
    $$ 
    which implies from $A(\tilde \Lambda')=A(\tilde \Lambda)$ that
    $x(\tilde\Lambda')+y(\tilde\Lambda')-e(\tilde\Lambda')=x(\tilde\Lambda)+y(\tilde\Lambda)+e(\tilde\Lambda)$. Condition (iv) leads to 
    $0 \leq e(\tilde\Lambda)+e(\tilde\Lambda')\leq 1,$
    which is a contradiction. We conclude that $a\leq b$. 
    \end{proof}  

We now consider the problem of symplectically embedding a quadrilateral
toric domain $Q_{a,b,c,d}$ into either a ball or a specific family of
ellipsoids. The obvious inclusion
$
Q_{a,b,c,d}\hookrightarrow B(r)
$
exists whenever
$
r\geq\max\{a,b,c+d\}.
$
We now obtain lower bounds which apply to arbitrary symplectic
embeddings.
    
\begin{thm} \label{lem:firstbound}
Let $d\geq 1$ be a real positive number. Let $X_\Omega=Q_{1,b,1,d}$ where $b=d-\lfloor d\rfloor$ (if $d\not \in \Z$), and $b=1$ (if $d\in \Z$). If there exists a symplectic embedding $X_\Omega\rightarrow B(r)$, then 
$$
r\geq\sup_{k\in \Z_{>0}} \min\left\{d+1,\frac{3k+d-2}{k}, \frac{k+3}{2} \right\}.
$$

\end{thm}

\begin{proof} Let $m$ be equal to $-\lfloor d\rfloor$ if $d$ is not an integer, and equal to $1-d$ if $d$ is an integer. Under the standing assumptions, $X_\Omega=Q_{1,b,1,d}$ is a semi-weakly convex toric domain with defining function $f:[0,a] \to [0,+\infty)$ satisfying $f(t) = b -m t, t\in [0,1]$. 
We see the $4$-ball $B(r)=X_{\triangle_{r,r}}, r>0,$ as a convex toric domain, and assume that there exists a symplectic embedding $X_\Omega \to B(r)$ for some $r>0$. 

Since the function $g(x):=\frac{3x+d-2}{x},x\in [1,+\infty),$ is monotone,  $g(1) = d+1$ and $g(k) =\frac{3k+d-2}{k}$, it is enough to prove that for every $k\in \Z_{> 0},$ we have
\begin{equation}
\label{eq:inequalities}
r\geq \displaystyle\frac{k+3}{2}\quad\text{or}\quad r\geq \min_{1\leq l\leq k}\left\{\displaystyle\frac{3l+d-2}{l}\right\}.
\end{equation}

Fix $k\in \Z_{> 0}$ and consider the convex generator $\Lambda'=e^k_{-1,1}$. By Pick's formula, $I(\Lambda')=k(k+3)$. By \cite[Lemma~2.1(a)]{hutchings2016beyond}, $\Lambda'$ is a minimal convex generator for $B(r)$, and hence it satisfies the hypothesis on the target generator in Theorem \ref{thm:criterion_convex_convex}-(2). Therefore there exist a $(0,m)$-adapted semi-weakly convex generator $\Lambda$ with $I(\Lambda)=k(k+3)$, a disjoint factorization $\Lambda=\prod_{i=1}^s\Lambda_i^{d_i}$ and a factorization $\Lambda'=\prod_{i=1}^s (\Lambda'_i)^{d_i} = \prod_{i=1}^s\big(e_{-1,1}^{k'_i}\big)^{d_i}$, so that conditions (i)-(iv) in Theorem \ref{thm:criterion_convex_convex}-(2) hold. 

We analyse three different cases:

\textit{Case 1: $s=1$ and $\Lambda_1=e_{0,1}^{k_1}$.}
We claim that $d_1=1$.  Condition~(iii) with $d'_1=1$ gives
$
I(e_{0,1}^{k_1})
=
I(e_{-1,1}^{k'_1}),
$
hence
$
2k_1=k'_1(k'_1+3).
$
If $d_1>1$, then
\[
\begin{aligned}
I\big((e_{0,1}^{k_1})^{d_1}\big)
&=2d_1k_1
=d_1k'_1(k'_1+3)\\
&<
d_1k'_1(d_1k'_1+3)
=
I\big((e_{-1,1}^{k'_1})^{d_1}\big),
\end{aligned}
\]
contradicting condition~(iii) with $d'_1=d_1$.  Thus $d_1=1$.
Since the total target generator is $e_{-1,1}^{k}$, this gives
$k'_1=k$, and therefore
$
k_1=k(k+3)/2.
$

Now using condition (ii), we obtain
 \begin{align*}
    kr=l_{\triangle_{r,r}}(\Lambda') &\geq l_\Omega(e_{0,1}^{k(k+3)/2})=\frac{k(k+3)}{2}\Rightarrow r \geq \frac{k+3}{2}.
\end{align*}  

\textit{Case 2:} $s=1$ and $\Lambda_1\not=e_{0,1}^{k_1}\Rightarrow x(\Lambda)>0$. Condition (iv) gives $x(\Lambda_1)+y(\Lambda_1)\geq x(\Lambda'_1) + y(\Lambda'_1) + m(\Lambda'_1)-1 =3k'_1-1$.   Using condition (ii), we thus obtain 
$$
\begin{aligned}
            k_1'r=l_{\triangle_{r,r}}(\Lambda_1') & \geq l_\Omega(\Lambda_1)=y(\Lambda_1)+dx(\Lambda_1)
             = x(\Lambda_1)+y(\Lambda_1)+(d-1)x(\Lambda_1)\\
            &\geq 3k'_1-1+(d-1)x(\Lambda_1)\\
            &\geq 3k'_1+d-2, \quad \text{since } x(\Lambda_1)\geq 1,
\end{aligned}
$$
where the identity
$
l_\Omega(\Lambda_1)=y(\Lambda_1)+d\,x(\Lambda_1)
$
holds since $\Lambda_1$ is $(0,m)$-adapted. We conclude that $r\geq \frac{3k'_1+d-2}{k_1'}$. Since $k_1'\leq k,$ we obtain $r\geq \min_{1\leq l\leq k} \left\{\frac{3l+d-2}{l} \right\}$. 

\textit{Case 3:} $s>1$. In this case, since the factorization of $\Lambda$ is disjoint and $\Lambda$ is $(0,m)$-adapted, there exists $\Lambda_i$ such that $x(\Lambda_i)>0$. We can now argue as in Case 2  to conclude as before that $r\geq \min_{1\leq l \leq k}\left\{\frac{3l+d-2}{l} \right\}$. \end{proof}

    \begin{cor}\label{thm: quadri into ball}
        Let the quadrilateral toric domain $X_\Omega=Q_{1,b,1,d}$ satisfy $1< d \leq 2$ and $b=d-1$. Then the inclusion $X_\Omega \to B(d+1)$ is sharp.
        
    \end{cor}

    \begin{proof}
        Assume that there exists a symplectic embedding $X_\Omega \to B(r)$ for some $r>0$. Since $1< d\leq 2$, the function $g(x):=\frac{3x+d-2}{x},x\in [1,+\infty),$ is increasing. Since $g(1) = d+1$ and $g(k) =\frac{3k+d-2}{k}$, Theorem \ref{lem:firstbound} implies that  $r\geq \min\{d+1,(k+3)/2\}$ for every $k\in \Z_{>0}$. For $k\geq 3$, we obtain $r\geq d+1$.
    \end{proof}

The next theorem is about symplectic embeddings of quadrilateral toric domains into ellipsoids.

\begin{thm}\label{thm: quadrilateral ellipsoid1}
Let $d\in\mathbb R_{>0}$ and $b\in\mathbb Z_{>0}$. Then there exists a
symplectic embedding
$
Q_{1,1,1,2}\longrightarrow E(d,bd)
$
if and only if
$
\Omega_{1,1,1,2}\subset\triangle_{d,bd}.
$
\end{thm}

\begin{proof}
If
$
\Omega_{1,1,1,2}\subset\triangle_{d,bd},
$
then the inclusion of moment regions gives the desired symplectic
embedding. We prove the converse.

The inclusion
$
\Omega_{1,1,1,2}\subset\triangle_{d,bd}
$
is equivalent to
$
bd\geq b+2.
$
Assume by contradiction that
$
bd<b+2
$
and that there exists a symplectic embedding
$
Q_{1,1,1,2}\to E(d,bd).
$
If $b=1$, the conclusion follows from
Corollary~\ref{thm: quadri into ball}, so we assume $b\geq2$.

The defining function of $Q_{1,1,1,2}$ is
$
f(x)=1+x
$
on $[0,1]$. Thus Theorem~\ref{thm:criterion_convex_convex}-(2) is the relevant criterion, with $m=-1$; it applies to each minimal convex generator chosen at the target. For every $(0,-1)$-adapted semi-weakly convex generator
$\Gamma$ with $y(\Gamma)\geq0$,
\begin{equation}
\label{eq:Q112-action}
\ell_\Omega(\Gamma)=2x(\Gamma)+y(\Gamma).
\end{equation}
Indeed, an edge $(-p,q)$ of a $(0,-1)$-adapted path satisfies
$q\geq-p$, and hence the supporting point for that edge is $(1,2)$.

First consider $\Lambda'=e_{-1,b}$. By \cite[Lemma~2.1(a)]{hutchings2016beyond}, this is a minimal convex generator for $E(d,bd)$: it is the maximal convex integral path tangent to the line of slope $-b$ through $(1,0)$ and $(0,b)$. Since $(-1,b)$ is primitive and $\Lambda'$ consists of a single
elliptic edge of multiplicity one, $\Lambda'$ admits no nontrivial
factorization. Hence the factorization supplied by
Theorem~\ref{thm:criterion_convex_convex}-(2) has a single nontrivial
factor with multiplicity one. Consequently, the corresponding
factorization of the negative generator $\Lambda$ also has a single
factor with multiplicity one, so item~(iv) applies directly to
$\Lambda$ and $\Lambda'$.

Theorem~\ref{thm:criterion_convex_convex}-(2) therefore gives a
$(0,-1)$-adapted semi-weakly convex generator $\Lambda$ satisfying
$
I(\Lambda)=I(\Lambda')=2b+2,
$
together with the inequalities in Theorem~\ref{thm:criterion_convex_convex}.
In particular,
\[
x(\Lambda)+y(\Lambda)-\frac{h(\Lambda)}2
\geq
x(\Lambda')+y(\Lambda')+m(\Lambda')-1
=
b+1.
\]
Hence
$
x(\Lambda)+y(\Lambda)\geq b+1.
$
If $x(\Lambda)\geq1$, then by
\eqref{eq:Q112-action},
$$
bd
=
\ell_{\triangle_{d,bd}}(e_{-1,b})
\geq
\ell_\Omega(\Lambda)
=
2x(\Lambda)+y(\Lambda)
\geq
b+2,
$$
contradicting $bd<b+2$. Therefore $x(\Lambda)=0$. The only possible
nontrivial $(0,-1)$-adapted generator with $x=0$ is a vertical
elliptic generator, and the index equality gives
$
\Lambda=e_{0,1}^{b+1}.
$

The generator $(e_{-1,b})^2$ is likewise a minimal convex generator for $E(d,bd)$ by \cite[Lemma~2.1(a)]{hutchings2016beyond}, now using the parallel tangent line through $(2,0)$ and $(0,2b)$. Apply Theorem~\ref{thm:criterion_convex_convex}-(2) to $\Lambda'=(e_{-1,b})^2$.
Any factorization of this one-edge generator consists
of factors in the same direction. If there were two distinct factors,
then applying the preceding argument to each factor would force the
corresponding two negative factors to be
$
e_{0,1}^{b+1},
$
contradicting the disjointness of the negative factorization.
There is also no factorization with one factor repeated twice:
condition~(iii) would give
$
I\big((e_{0,1}^{b+1})^2\big)
=
I\big((e_{-1,b})^2\big),
$
whereas
$
I\big((e_{0,1}^{b+1})^2\big)=4b+4
$
and
$
I\big((e_{-1,b})^2\big)=6b+4.
$
Consequently the factorization has a single factor with multiplicity
one.

For the resulting negative generator $\Lambda$, item~(iv) gives
\[
x(\Lambda)+y(\Lambda)-\frac{h(\Lambda)}2
\geq
x((e_{-1,b})^2)+y((e_{-1,b})^2)
+m((e_{-1,b})^2)-1
=
2b+3.
\]
Thus $x(\Lambda)+y(\Lambda)\geq2b+3$. If $x(\Lambda)\geq1$, then
\[
2bd
=
\ell_{\triangle_{d,bd}}\big((e_{-1,b})^2\big)
\geq
\ell_\Omega(\Lambda)
=
2x(\Lambda)+y(\Lambda)
\geq
2b+4,
\]
again contradicting $bd<b+2$. Hence $x(\Lambda)=0$. The index equality
then forces
$
\Lambda=e_{0,1}^{3b+2}.
$
Therefore
$
2bd
\geq
\ell_\Omega(\Lambda)
=
3b+2.
$
On the other hand $bd<b+2$ gives
$
2bd<2b+4.
$
Since $b\geq2$,
$
3b+2\geq2b+4,
$
which is a contradiction. Hence
$
bd\geq b+2,
$
as required.
\end{proof}

\begin{cor}
Let $1\leq a\leq2$, $b\in\mathbb Z_{>0}$ and $d\in\mathbb R_{>0}$.
Then there exists a symplectic embedding
$
Q_{1,a,1,2}\longrightarrow E(d,bd)
$
if and only if
$
\Omega_{1,a,1,2}\subset\triangle_{d,bd}.
$
\end{cor}

\begin{proof}
If
$
\Omega_{1,a,1,2}\subset\triangle_{d,bd},
$
the inclusion gives a symplectic embedding. Conversely, if
$
Q_{1,a,1,2}\to E(d,bd)
$
symplectically, then
$
Q_{1,1,1,2}\subset Q_{1,a,1,2}
$
also symplectically embeds into $E(d,bd)$. By
Theorem~\ref{thm: quadrilateral ellipsoid1},
$
\Omega_{1,1,1,2}\subset\triangle_{d,bd}.
$
By Theorem~\ref{thm: quadrilateral ellipsoid1}, the inclusion
$
\Omega_{1,1,1,2}\subset\triangle_{d,bd}
$
is equivalent to
$
bd\geq b+2.
$
Since $b\geq1$, this implies $d>1$ and therefore
$
d>1, bd\geq b+2\geq3\geq a.
$
Hence the vertices $(1,0)$ and $(0,a)$ of
$\Omega_{1,a,1,2}$ belong to $\triangle_{d,bd}$. Moreover,
\[
\frac{1}{d}+\frac{2}{bd}
=
\frac{b+2}{bd}
\leq1,
\]
so the remaining nonzero vertex $(1,2)$ also belongs to
$\triangle_{d,bd}$. Since $\triangle_{d,bd}$ is convex, it follows
that
$
\Omega_{1,a,1,2}\subset\triangle_{d,bd}.
$
\end{proof}

Now we show that obstructions coming from ECH capacities are not enough
to obtain Theorem~\ref{thm: quadrilateral ellipsoid1}.

\begin{prop}
\label{prop:capacities-Q112}
The ECH capacities of $Q_{1,1,1,2}$ are given by
\[
c_k(Q_{1,1,1,2})
=
\min\left\{
2m+n:
m,n\in\mathbb Z_{\geq0},\
(m+1)(n+1)+\frac{m(m+1)}2\geq k+1
\right\}.
\]
\end{prop}

\begin{proof}
Let $A_k$ denote the right hand side.  The defining function of
$Q_{1,1,1,2}$ is $f(x)=1+x$, so, as in
\eqref{eq:Q112-action}, every $(0,-1)$-adapted path with endpoint
coordinates $(m,n)$ has
$
\ell_\Omega(\Lambda)=2m+n.
$

Fix $m,n\in\mathbb Z_{\geq0}$ and consider the elliptic
$(0,-1)$-adapted path
$
\Lambda_{m,n}
=
e_{0,1}^{m+n}e_{-1,-1}^{m}.
$
It starts at $(m,0)$, first follows the vertical direction to
$(m,m+n)$, and then follows the diagonal direction $(-1,-1)$ to
$(0,n)$.  A direct lattice-point count gives
\[
\mathcal L(B^+_{\Lambda_{m,n}})
=
(m+1)(n+1)+\frac{m(m+1)}2,
\]
and
$
\ell_\Omega(\Lambda_{m,n})=2m+n.
$

Suppose that
\[
(m+1)(n+1)+\frac{m(m+1)}2\geq k+1.
\]
If $m=0$, then the inequality says $n\geq k$.  We take
$
\widetilde\Lambda=e_{0,1}^{k}.
$
This is a $(0,-1)$-adapted semi-weakly convex path satisfying
$
\mathcal L(B^+_{\widetilde\Lambda})=k+1
$
and
$
\ell_\Omega(\widetilde\Lambda)=k\leq n=2m+n.
$
If $m>0$ and the inequality is strict, apply the corner-rounding
operation used in the proof of
Theorem~\ref{thm:semiweak_capacities}. Repeating it produces a
$(0,-1)$-adapted path $\widetilde\Lambda$ with
$
\mathcal L(B^+_{\widetilde\Lambda})=k+1
$
and
$
\ell_\Omega(\widetilde\Lambda)
\leq
\ell_\Omega(\Lambda_{m,n})
=
2m+n.
$
If $m>0$ and equality holds, simply take
$
\widetilde\Lambda=\Lambda_{m,n}.
$
In every case,
Theorem~\ref{thm:semiweak_capacities} gives
$
c_k(Q_{1,1,1,2})\leq2m+n.
$
Taking the minimum over all such pairs $(m,n)$ gives
$
c_k(Q_{1,1,1,2})\leq A_k.
$

Conversely, let $\Lambda$ realize the minimum in
Theorem~\ref{thm:semiweak_capacities}, so that
\[
\mathcal L(B^+_\Lambda)=k+1,
\qquad
\ell_\Omega(\Lambda)=c_k(Q_{1,1,1,2}).
\]
Set
$
m=x(\Lambda),
n=y(\Lambda).
$
Because $\Lambda$ is $(0,-1)$-adapted, every nonvertical edge
$(-p,q)$ satisfies $q\geq-p$. The path $\Lambda$ therefore lies on or
below the extremal path $\Lambda_{m,n}$ defined above. Hence
$
\mathcal L(B^+_{\Lambda_{m,n}})
\geq
\mathcal L(B^+_\Lambda)
=
k+1.
$
Thus the pair $(m,n)$ is admissible in the definition of $A_k$. Since
the action depends only on the endpoint coordinates in this case,
\[
A_k
\leq
2m+n
=
\ell_\Omega(\Lambda)
=
c_k(Q_{1,1,1,2}).
\]
Combining the two inequalities proves the proposition.
\end{proof}

Now suppose that $Q_{1,1,1,2}$ symplectically embeds into $B(r)$. The
first capacities are
\[
\{c_k(Q_{1,1,1,2})\}_{k\geq0}
=
\{0,1,2,3,3,4,4,5,5,6,6,6,7,7,7,\ldots\},
\]
whereas
\[
\{c_k(B(r))\}_{k\geq0}
=
\{0,r,r,2r,2r,2r,3r,3r,3r,3r,4r,\ldots\}.
\]
In fact ECH capacities give exactly the lower bound $r\geq2$ and no
stronger one. To see this, let
$
c_k(B(1))=d.
$
Then
\[
\frac{d(d+1)}2
\leq
k
\leq
\frac{(d+1)(d+2)}2-1.
\]
Taking $m=d$ and $n=0$ in
Proposition~\ref{prop:capacities-Q112} gives
\[
c_k(Q_{1,1,1,2})\leq2d=2c_k(B(1)).
\]
On the other hand,
$
c_2(Q_{1,1,1,2})=2
$
and
$
c_2(B(1))=1.
$
Thus
\[
\sup_{k\geq1}
\frac{c_k(Q_{1,1,1,2})}{c_k(B(1))}
=
2.
\]
Consequently ECH capacities alone imply only $r\geq2$, while
Corollary~\ref{thm: quadri into ball} gives the sharp obstruction
$r\geq3$.

\section{Foundations of Embedded Contact Homology} \label{ECHfoundations}

Let $(Y, \xi = \ker \lambda)$ be a  closed contact $3$-manifold with  contact form $\lambda$ and contact structure $\xi$. The Reeb vector field  of $\lambda$ is determined by $d \lambda(R_\lambda,\cdot)=0$ and $\lambda(R_\lambda)=1$. We denote the flow of $R_{\lambda}$ by $\phi_t, t\in \R,$ called the Reeb flow. A closed orbit of $\phi_t$ is called a Reeb orbit. A Reeb orbit $\gamma:\mathbb{R}/T\mathbb{Z}\rightarrow Y$ with period $T>0$ is nondegenerate if the linear map $P_\gamma:= d\phi_T|_\xi: \xi_{\gamma(0)}\rightarrow \xi_{\gamma(0)}$ has no eigenvalue $1$. The contact form $\lambda$ is nondegenerate if all Reeb orbits are nondegenerate. 
We say that $\gamma$ is elliptic if the eigenvalues of $P_\gamma$ lie on the unit circle. We say that a nondegenerate Reeb orbit $\gamma$ is positive (negative) hyperbolic if the eigenvalues of $P_\gamma$ are positive (negative).

An {\it orbit set} is a finite set $\alpha=\{(\alpha_i,m_i)\}$, where $\alpha_i$ are distinct embedded Reeb orbits on $Y$ and $m_i$ are positive integers. An admissible orbit set is an orbit set such that $m_i=1$ whenever $\alpha_i$ is hyperbolic. We denote the homology class of an orbit set $\alpha$ by 
$$[\alpha]=\displaystyle\sum m_i [\alpha_i]\in H_1(Y).$$
For a fixed $\Gamma\in H_1(Y)$, and a generic almost complex structure $J$ on $\mathbb{R}\times Y$ compatible with its symplectic structure, the chain complex $\text{ECC}_*(Y,\lambda,\Gamma,J)$ is the $\mathbb{Z}_2$-vector space generated by the admissible orbit sets in homology class $\Gamma$, and its differential counts certain $J$-holomorphic curves in $\mathbb{R}\times Y$, as explained below. This chain complex gives rise to the embedded contact homology $\text{ECH}_*(Y,\lambda,
\Gamma,J)$. Taubes proved in \cite{TaubesECHSW}  that $\text{ECH}_*(Y,\lambda,
\Gamma,J)$ is isomorphic to a version of Seiberg-Witten Floer cohomology $\widehat{HM}^{-*}(Y,\mathfrak{s}_\xi+PD(\Gamma))$. In particular, $\text{ECH}_*(Y,\lambda,\Gamma,J)$ does not depend on $\lambda$ or $J$, and so we just write $ECH_*(Y,\xi,\Gamma)$.

In the next sections we will need to consider relative homological classes that relate two different orbit sets $\alpha=\{(\alpha_i,m_i)\}$ and $\beta=\{(\beta_j,n_j)\}$ with the same homology $\Gamma = [\alpha]=[\beta]$. To be precise we denote by $H_2(Y,\alpha,\beta)$ the affine space over $H_2(Y)$ consisting of $2$-chains $\Sigma$ in $Y$ with 
$$\partial \Sigma=\displaystyle\sum_im_i\alpha_i-\displaystyle\sum_jn_j\beta_j$$
modulo boundaries of $3$-chains. We call $H_2(Y,\alpha,\beta)$ the \textit{relative second homology of $\alpha$ and $\beta$}.

Several of the definitions of ECH are a bit delicate. Because of that we dedicate some more  subsections to properly define the different parts that constitute 
 this homology.

\subsection{The \text{ECH} index}

Given orbit sets $\alpha = \{(\alpha_i, m_i)\}, \beta= \{(\beta_j, n_j)\}$ and $Z\in H_2(Y, \alpha,\beta)$, the ECH index $I(\alpha, \beta, Z)$ is an integer defined by
\begin{equation}
    \label{ECHindex}    I(\alpha,\beta,Z):=c_\tau(Z)+Q_\tau(Z)+\text{CZ}^I_\tau(\alpha)-\text{CZ}^I_\tau(\beta).
\end{equation}
Here, $\tau$ is a trivialization of $\xi$ over the Reeb orbits $\alpha_i$ and $\beta_j$, $c_\tau(Z)$ denotes the relative first Chern class of $\xi$ over $Z$ with respect to $\tau$, $Q_\tau(Z)$ denotes the relative self-intersection number of $Z$ with respect to $\tau$, and 
$$
\text{CZ}^I_\tau(\alpha):=\displaystyle\sum_i\displaystyle\sum_{k=1}^{m_i}\text{CZ}_\tau (\alpha_i^k)
$$
where $\text{CZ}_\tau$ denotes the Conley-Zehnder index with respect to $\tau$, and $\alpha^k_i$ denotes the $k$-fold cover of $\alpha_i$. For more details about the definition of the different summands of the ECH index see \cite[Sec.~3.4]{hutchings2014lecture}. The ECH index does not depend on the trivialization $\tau$. 

\subsection{Trivializations.} Since several of the quantities appearing in the ECH index depend individually on the choice of trivialization, while the ECH index itself does not, we briefly recall the relevant conventions regarding changes of trivialization.

Although the definition of the ECH index involves a choice of trivialization, the resulting integer is independent of this choice. Its individual components, however, do depend on the trivialization.

Let $\gamma$ be a Reeb orbit. Following \cite{hutchings2002index}, we denote by $\mathcal{T}(\gamma)$ the set
of homotopy classes of symplectic trivializations of $\xi|_\gamma$.
 If $\tau, \tau'\in\mathcal{T}(\gamma)$,  then $\tau' - \tau$ denotes the degree of $\tau \circ (\tau')^{-1}:S^1\rightarrow \text{Sp}(2,\mathbb{R})$ along $\gamma$.  Let $\mathcal{T}(\alpha,\beta)=
\prod_i\mathcal{T}(\alpha_i)\times\prod_j\mathcal{T}(\beta_j)$ where $\alpha=\{(\alpha_i,m_i)\}$ and $\beta=\{(\beta_i,m_i)\}$ be orbit sets.  If
$\tau=(\tau^+,\tau^-)\in\mathcal{T}(\alpha,\beta)$, we denote the corresponding elements of
$\mathcal{T}(\alpha_i)$ and $\mathcal{T}(\beta_j)$ by $\tau_i^+$ and $\tau_j^-$.

The next lemma records the transformation laws of the components of the ECH index under a change of trivialization.

\begin{lem}
\label{lem:changetrivialization}
Let $\{(\alpha_i,m_i)\}$ and $\{(\beta_j,n_j)\}$ be orbit sets, and let $\tau, \tau' \in \mathcal T(\alpha,\beta)$. Then 

\begin{equation}
\label{eqn:chernTriv}
c_\tau(Z)-c_{\tau'}(Z)=
\sum_im_i({\tau'_i}^+-\tau_i^+)-\sum_jn_j({\tau'_j}^--\tau_j^-).
\end{equation}

\begin{equation}
\label{eqn:QtauTriv}
Q_\tau(Z)-Q_{\tau'}(Z)=\sum_im_i^2({\tau'_i}^+-\tau_i^+) -
\sum_jn_j^2({\tau'_j}^--\tau_j^-).
\end{equation}

\begin{equation}
\label{eqn:CZTriv}
\text{CZ}^I_{\tau}(\alpha)-\text{CZ}^I_{\tau'}(\alpha)=\sum_i m_i(m_i+1)(\tau^+_i-\tau'^{+}_i).
\end{equation}
\end{lem}

\begin{proof}
Equation \eqref{eqn:chernTriv} is equation (6) in Section 2.2 of
\cite{hutchings2002index}, while the relative self-intersection number is Lemma 2.5(b). To prove
\eqref{eqn:CZTriv}, equation (7) in Section 2.3 of
\cite{hutchings2002index} gives
\[
\text{CZ}_{\tau}(\alpha_i^k)-\text{CZ}_{\tau'}(\alpha_i^k)
=2k(\tau_i^+-{\tau_i'}^+),
\]
for every positive integer $k$. Therefore,
\[
\begin{aligned}
\text{CZ}^I_{\tau}(\alpha)-\text{CZ}^I_{\tau'}(\alpha)
&=\sum_i\sum_{k=1}^{m_i}
\left(
\text{CZ}_{\tau}(\alpha_i^k)-\text{CZ}_{\tau'}(\alpha_i^k)
\right)\\
&=\sum_i\sum_{k=1}^{m_i}
2k(\tau_i^+-{\tau_i'}^+)\\
&=\sum_i m_i(m_i+1)(\tau_i^+-{\tau_i'}^+),
\end{aligned}
\]
which proves \eqref{eqn:CZTriv}.
\end{proof}

\subsection{Holomorphic Curves and the Index Inequality}

Let $s$ denote the $\mathbb{R}$-coordinate.  An almost complex structure $J$ on $\mathbb{R}\times Y$ is $\lambda$-compatible if $J(\partial_s)=R_{\lambda}$, $J$ preserves $\xi$, $d \lambda(v,Jv)> 0$ for $0\neq v\in \xi$, and $J$ is $\mathbb{R}$-invariant. Let us fix a $\lambda$-compatible almost complex structure $J$. A $J$-holomorphic curve from the orbit set  $\alpha=\{(\alpha_i, m_i)\}$ to $\beta=\{(\beta_j,n_j)\}$ is a $J$-holomorphic curve in $\mathbb{R}\times Y$, where the domain is a possibly disconnected punctured compact Riemann surface, with positive ends asymptotic to covers $\alpha_i^{q_{i,k}}$ with total multiplicity $\sum_k q_{i,k}=m_i$, and negative ends asymptotic to covers $\beta_j^{q_{j,l}}$ with total multiplicity $\sum_l q_{j,l}=n_j$. For a more formal definition, see \cite{hutchings2014lecture}. A holomorphic curve $u$ as above determines a homology class $[u]\in H_2(Y,\alpha,\beta)$.

The Fredholm index of $u$ is defined as 
\begin{align}
    \label{FredholmIndex}
    \text{ind}(u):=-\chi(u)+2 c_\tau(u)+\text{CZ}_\tau^\text{ind}(u),
\end{align}
where $\chi(u)$ denotes the Euler characteristic of the domain of $u$, $\tau$ is a trivialization of $\xi$ over the orbit sets $\alpha$ and $\beta$, $c_\tau(u)$ is shorthand for $c_\tau([u])$, and

$$\text{CZ}_\tau^\text{ind}(u):=\sum_i\displaystyle\sum_k \text{CZ}_\tau (\alpha_i^{q_{i,k}})-\displaystyle\sum_j\displaystyle\sum_l \text{CZ}_\tau (\beta_j^{q_{j,l}}).$$

If $J$ is a generic almost complex structure and $u$ has no multiply covered components, then the moduli space of $J$-holomorphic curves from $\alpha$ to $\beta$ is a manifold near $u$ of dimension $\text{ind}(u)$, see \cite[Proposition 3.1]{hutchings2014lecture}. Also, if $u$ has no multiply covered components, then without any genericity assumption on $J$, we have the index inequality
\begin{equation}
    \label{boundsFred}
    \text{ind}(u)\leq I(u).    
\end{equation}
Here, we write $I(u)$ as a shorthand for $I(\alpha, \beta, [u])$. A proof of this inequality can be found in \cite[Sec. 3.4]{hutchings2014lecture}.

\subsection{Holomorphic Currents}

Let $\alpha$ and $\beta$ be orbit sets. A \textit{$J$-holomorphic current} from $\alpha$ to $\beta$ is a formal sum of $J$-holomorphic curves $\mathcal{C}=\sum_k d_k C_k$ where the $C_k$ are distinct, irreducible, somewhere injective $J$-holomorphic curves, such that if $C_k$ is a holomorphic curve from the orbit set $\alpha_k$ to the orbit set $\beta_k$, then  $\alpha=\prod_k\alpha_k^{d_k}$ and $\beta=\prod_k\beta_k^{d_k}$. Here, the product of two orbit sets is obtained by summing up the multiplicities of Reeb orbits. The curves $C_k$ are the components of the holomorphic current $\mathcal{C}$, and the integers $d_k$ are the \textit{multiplicities} of the components. 

Let $\mathcal{M}^J(\alpha,\beta)$ denote the set of $J$-holomorphic currents from $\alpha$ to $\beta$. Notice that $\mathbb{R}$ acts on this set by translation of the $\mathbb{R}$-coordinate on $\mathbb{R}\times Y$. Also, each $\mathcal{C}\in \mathcal{M}^J(\alpha,\beta)$ determines a homology class $[\mathcal{C}]\in H_2(Y,\alpha,\beta)$. Define the \text{ECH} index $I(\mathcal{C}):=I(\alpha, \beta, [\mathcal{C}])$.

\begin{prop}[Hutchings {\cite[Proposition 3.1]{hutchings2014lecture}}]
\label{BasedDiff}
Assume that $J$ is a generic $\lambda$-compatible almost complex structure on $\R \times Y$. Let $\alpha$ and $\beta$ be orbit sets and let $\mathcal{C}\in \mathcal{M}^J(\alpha, \beta)$ be a $J$-holomorphic current in $\mathbb{R}\times Y$, not necessarily somewhere injective. Then
\begin{enumerate}
    \item[(i)] The ECH index is nonnegative and is zero if and only if the current $\mathcal{C}$ is a union of trivial cylinders with multiplicities.
    
    \item[(ii)] If $I(\mathcal{C})=1$, then we have the decomposition $\mathcal{C}=\mathcal{C}_0\sqcup C_1$, where the component $\mathcal{C}_0$ has ECH index equal to zero, while the component $\mathcal{C}_1$ has ECH index equal to one.
    
    \item[(iii)] If $I(\mathcal{C})=2$, and $\alpha$ and $\beta$ are chain complex generators, then we have the decomposition $\mathcal{C}=\mathcal{C}_0\sqcup C_2$, where the component $\mathcal{C}_0$ has ECH index equal to zero, while the component $\mathcal{C}_2$ has ECH index equal to two.
\end{enumerate}
\end{prop}
In the definition above, $\mathcal C_0$ is possibly the empty set and we follow the convention that $I(\emptyset)=0$.

\subsection{The Differential}
Let $\alpha$ and $\beta$ be admissible orbit sets with $[\alpha]=[\beta]=\Gamma\in H_1(Y)$, and let $Z \in H_2(Y, \alpha,\beta)$. For a generic $\lambda$-compatible $J$, the number of currents $\mathcal{C}\in \mathcal{M}^J(\alpha,\beta)$ modulo $\R$-translations, with $[\mathcal C]=Z$ and $I(\mathcal C)=1$, is finite. The space of such equivalence classes is denoted by $\widetilde{\mathcal{M}}^J(\alpha,\beta,Z)$. We define the differential $\partial$ by
$$
\partial \alpha :=\displaystyle\sum_\beta \displaystyle\sum_{I(\alpha,\beta,Z)=1}\#\widetilde{\mathcal{M}}^J(\alpha,\beta,Z)\beta. 
$$
In \cite[Section 3.5]{hutchings2014lecture}, it is proven that this differential is well-defined and in \cite[Section 7]{hutchings2007gluing}, it is proven that $\partial^2=0$. We denote the homology of the chain complex $\text{ECC}(Y,\lambda,\Gamma,J)$ by $\text{ECH}(Y,\lambda, \Gamma, J)$.

If $\Gamma=0$ and $H_2(Y)=0$ then the chain complex $\text{ECC}(Y,\lambda,0,J)$ has a canonical $\mathbb{Z}$-grading, in which the grading of an admissible orbit set $\alpha$ is defined by
\begin{equation}
    \label{ourindex}
    I(\alpha)=I(\alpha,\emptyset,Z),
\end{equation}
where $Z$ is the unique element in $H_2(Y,\alpha,\emptyset)$. In this case, we use similar notation for the summands of $I(\alpha)$, that is $c_\tau(\alpha):=c_\tau(Z)$ and $Q_\tau(\alpha):=Q_\tau(Z)$.

It follows from a theorem of Taubes \cite{TaubesECHSW}, that by identifying $\text{ECH}$  with a version of Seiberg-Witten Floer cohomology,   $\text{ECH}(Y,\lambda,\Gamma,J)$ depends only on $Y$, $\xi=\ker \lambda$ and $\Gamma$. This invariance of $\text{ECH}$ currently cannot be proved directly by counting holomorphic curves.

An important case is when $Y$ is diffeomorphic to a lens space, and $\xi = \ker \lambda$ is the standard contact structure. It is proved in \cite[Cor. 3.9]{trejos2024symplecticembeddingstoricdomains} that

\begin{equation}
    \label{ECHlensspace}
    \text{ECH}_*(Y,\lambda,0,J)=\begin{cases}
        \mathbb{Z}/2, & *=0,2,4,6,\dots\\
        0, & \text{otherwise.}
    \end{cases}
\end{equation}

\subsection{Topological Complexity.}
\label{sec:topcom}
Notice from \eqref{boundsFred} that the ECH index bounds the Fredholm index from above which indeed leads toward the fact that the ECH differential is well-defined. Besides the ECH index there is another quantity denoted by $J_0$ which in a way bounds the topological complexity of the $J$-holomorphic currents. 

Let $\alpha=\{(\alpha_i,m_i)\}$ and $\beta=\{(\beta_j,n_j)\}$ be admissible orbit sets with $[\alpha]=[\beta]\in H_1(Y)$, and let $Z\in H_2(Y,\alpha,\beta)$. We define

\begin{equation}
    \label{eq:topcomind}
    J_0(\alpha,\beta,Z):=-c_\tau(Z)+Q_\tau(Z)+\displaystyle\sum_i\displaystyle\sum_{k=1}^{m_i-1} \text{CZ}_\tau (\alpha_i^k)-\displaystyle\sum_j\displaystyle\sum_{k=1}^{n_j-1} \text{CZ}_\tau (\beta_j^k).
\end{equation}

Note that if we compute the difference between the ECH index and the $J_0$-index we obtain
$$
I(\alpha,\beta,Z)-J_0(\alpha,\beta,Z)=2c_{\tau}(Z)+\sum_i CZ_{\tau}(\alpha_i^{m_i})-\sum_j CZ_{\tau}(\beta_j^{n_j}).
$$

As in the case of the ECH index, if $\mathcal{C}\in \mathcal{M}^J(\alpha,\beta)$, we write $J_0(\mathcal{C})=J_0(\alpha,\beta,[\mathcal{C}])$.

Let $C$ be a somewhere injective and irreducible $J$-holomorphic curve from $\alpha$ to $\beta$. We denote by $n_i^+$ the number of positive ends of $C$ at covers of $\alpha_i$, and by $n_j^-$ the number of negative ends of $C$ at covers of $\beta_j$. We can now state precisely how  $J_0$ controls the topological complexity of $C$. 

\begin{prop}[Hutchings {\cite[Prop.~3.2]{hutchings2016beyond}}]
    \label{pro:topcomp}
    Let $\alpha$, $\beta$ and $C$ be as above. Then
    \begin{equation}
        \label{TopComB}
        2g(C)-2+\sum_i(2n_i^+ -1) + \sum_j(2n_j^- -1)\leq J_0(C)
    \end{equation}
    where $g(C)$ denotes the genus of $C$.
\end{prop}
\begin{rem}
    In the case of a contact manifold $Y$ diffeomorphic to $S^3$ the index $J_0$ only depends on the orbit sets $\alpha$ and $\beta$ since $H_2(Y)=0$. Furthermore, since $H_1(Y)=0$ we can  write $J_0(\alpha):=J_0(\alpha,\emptyset,Z)$ where $Z$ is the only element in $Z\in H_2(Y,\alpha,\emptyset)$.
\end{rem}
\subsection{Filtered ECH}
There exists a filtration on ECH which allows us to compute the embedded contact homology via succesive approximations, see \cite[Theorem 2.17]{nelson2022embedded}. The symplectic action (or the lenght) of an orbit set $\alpha=\{(\alpha_i,m_i)\}$ is defined by
$$
    \mathcal{A}(\alpha):=\sum_i m_i\int_{\alpha_i} \lambda.
$$

If $J$ is a $\lambda$-compatible almost complex structure and there exists a $J$-holomorphic current from $\alpha$ to $\beta$, then $\mathcal{A}(\alpha)\geq \mathcal{A}(\beta)$ by Stokes theorem. Indeed, $d \lambda$ is a non-negative form along $J$-holomorphic curves. Since the differential $\partial$ counts $J$-holomorphic currents, the symplectic action decreases under $\partial$, that is, 
\begin{align}
\label{Stokes}   \langle\partial\alpha,\beta\rangle\not=0 \Rightarrow \mathcal{A}(\alpha)\geq \mathcal{A}(\beta).
\end{align}

Given $L>0$, let $\text{ECC}_*^L(Y,\lambda,\Gamma,J)$ denote the subgroup of $\text{ECC}_*(Y,\lambda,\Gamma,J)$ generated by the orbit sets of symplectic action less that $L$. Because $\partial$ decreases action, $\text{ECC}_*^L(Y,\lambda,\Gamma,J)$ is a subcomplex of $\text{ECC}_*(Y,\lambda,\gamma,J)$. We denote the homology of $\text{ECC}^L(Y,\lambda,\Gamma,J)$ by 
$\text{ECH}^L_*(Y,\lambda,\Gamma,J)$, and we call it the filtered $\text{ECH}$ homology.
Given $L<L'$, there exists a homomorphism 
$$
\iota^{L,L'}:\text{ECH}_*^L(Y,\lambda,\Gamma)\rightarrow\text{ECH}_*^{L'}(Y,\lambda,\Gamma),   
$$
induced by the inclusion map $\text{ECC}_*^L(Y,\lambda,\Gamma)\hookrightarrow\text{ECC}_*^{L'}(Y,\lambda,\Gamma)$, which is independent of $J$.  Notice that 
\begin{align}
\label{directlimit}
\text{ECH}_*(Y,\lambda,\Gamma)=H_*\left(\varinjlim_{L\rightarrow \infty} \text{ECC}_*^L(Y,\lambda,\Gamma,J)\right)=\varinjlim_{L\rightarrow\infty}\text{ECH}^L_*(Y,\lambda,\Gamma),   
\end{align}
which allows us to define objects over $\text{ECH}_*(Y,\lambda,\Gamma)$ by defining them over $\text{ECH}_*^L(Y,\lambda,\Gamma)$.

\subsection{Symplectic Cobordisms.}
\label{symcob}

A central construction in the proof of Theorems \ref{thm: criterion_convex_concave}, \ref{thm:criterion_convex_convex} and \ref{thm:criterion_concave_concave} is the ECH cobordism map induced by a strong symplectic cobordism between two contact manifolds $(Y_+,\lambda_+)$ and $(Y_-,\lambda_-)$. 

\begin{defi} A strong symplectic cobordism from $(Y_+,\lambda_+)$ to $(Y_-,\lambda_-)$ is a compact symplectic manifold $(X,\omega)$ with contact-type boundary $\partial X=Y_+-Y_-$ such that $\omega|_{Y_\pm}=d \lambda_{\pm}$. 
\end{defi}

Let $(X,\omega)$ be a strong symplectic cobordism from $Y^+$ to
$Y^-$. From the definition, there are
collar neighborhoods of the two boundary components on which the
symplectic form takes the standard cylindrical form. More precisely,
\begin{itemize}
    \item For some $\epsilon>0$, a neighborhood $N_-$ of $Y_-$ can be identified
with $[0,\epsilon)\times Y_-$ in such a way that $\omega=d (e^s\lambda_-)$, where $s$ denotes the coordinate on $[0,\epsilon)$.

    \item At the positive
boundary, we similarly identify a neighborhood $N_+$ of $Y_+$ with
$(-\epsilon,0]\times Y_+$, with $\omega=d(e^s\lambda_+)$.
\end{itemize}

We can therefore attach cylindrical ends to both boundary components.
The resulting symplectic manifold, called the completion of $(X,\omega)$,
is
$$
\bar{X}
=
(((-\infty,\epsilon]\times Y_-)\cup X
\cup([-\epsilon,+\infty)\times Y_+))/\sim,
$$
where the equivalence relation $\sim$ is given by the identifications
of $N_-$ and $N_+$ with the corresponding collar neighborhoods in
$(-\infty,\epsilon]\times Y_-$ and $[-\epsilon,+\infty)\times Y_+$,
respectively.

Fix an almost complex structure $J$ on $\bar{X}$ satisfying the
following conditions. On $X$, the almost complex structure $J$ is
required to be $\omega$-compatible, while on the cylindrical ends it
is required to coincide with $\lambda_\pm$-compatible almost complex
structures $J_\pm$ on $(-\infty,0]\times Y_-$ and
$[0,\infty)\times Y_+$, respectively. We refer to such a choice of
$J$ as admissible.

For orbit sets $\alpha_\pm$ in $Y_\pm$, let
$\mathcal{M}^J(\alpha_+,\alpha_-)$ denote the moduli space of
$J$-holomorphic currents in $\bar{X}$ connecting $\alpha_+$ to
$\alpha_-$. The definition is the same as that given previously for
$J$-holomorphic currents in $\mathbb{R}\times Y$. The estimates
\eqref{boundsFred} and \eqref{TopComB} continue to apply in this
setting, provided that the $J$-holomorphic currents have no multiply
covered components; see \cite[Sec. 3.9]{hutchings2016beyond}.

The definitions of ${\rm ind}$, $I$, and $J_0$ are unchanged from the
case of $\mathbb{R}\times Y$, with one modification concerning the
relative first Chern class. Namely, the term $c_\tau$ is now computed
using the relative first Chern class of $T\bar{X}$.

The $J$-holomorphic currents that appears in the ECH cobordism map are more complex than the ones in the definition of the differential. Indeed, they deserve a new definition.

\begin{defi}
\label{J-current}
Let $J$ be an admissible almost complex structure on $\bar{X}$ which restricts to $\lambda_\pm$-compatible almost complex structures $J_\pm$ on the ends. Let $\alpha_\pm$ be orbit sets in $Y_\pm$. We define a broken $J$-holomorphic current from $\alpha_+$ to $\alpha_-$ as a tuple $B=(\mathcal{C}_{N_-},\mathcal{C}_{N_-+1}\dots,\mathcal{C}_{N_+})$, with $N_-\leq 0 \leq N_+$ so that
\begin{enumerate}
    \item $\alpha_{-,i}, i=N_-, \ldots 0,$ are orbit sets in $Y_-$, with $\alpha_{-,N_-}=\alpha_-$.
    
    \item $\alpha_{+,i},  i=0, \ldots, N_+,$ are orbit sets in $Y_+$, with $\alpha_{+,N_+} = \alpha_+$.

    \item $\mathcal{C}_i\in \mathcal{M}^{J_-}(\alpha_{-,i+1},\alpha_{-,i})/ \mathbb{R}$ for every $i= N_-,\dots, -1$.
    
    \item $\mathcal{C}_0\in \mathcal{M}^J(\alpha_{+,0},\alpha_{-,0})$.
    
    \item $\mathcal{C}_i\in \mathcal{M}^{J_+}(\alpha_{+,i},\alpha_{+,i-1})/ \mathbb{R}$ for $i= 1,\dots, N_+$.
    
    \item If $i\not=0$, then one of the components of $\mathcal{C}_i$ is not a trivial cylinder. 
\end{enumerate}
 
    The $J$-holomorphic currents $\mathcal{C}_i$ are called levels of $B$. The ECH index of $B$ is defined as the sum of the ECH indices of its levels
    $$ I(B):=\sum^{N_+}_{i=N_-}I(\mathcal{C}_i).
    $$
\end{defi}

\subsection{Cobordism Maps}
The particular cobordisms that we use are called \textit{weakly exact}. The relevant explanation about how this kind of cobordism induces an ECH map can be found in \cite[Theo.~1.9]{hutchings2013proof}.

\begin{defi}
    We call the strong symplectic cobordism $(X,\omega)$ a \textit{weakly exact} cobordism if there exists a $1$-form $\lambda$ on $X$ such that $d \lambda =\omega$. 
\end{defi}

\begin{thm}[Hutchings {\cite[Theorem 3.5]{hutchings2016beyond}}]
\label{CoborMap}
    Let $(Y_+,\lambda_+)$ and $(Y_-,\lambda_-)$ be closed and co-oriented contact three-manifolds, and let $(X,\omega)$ be a weakly exact cobordism from $(Y_+,\lambda_+)$ to $(Y_-,\lambda_-)$. Then there exist canonical maps
    $$\Phi^L(X,\omega):\text{ECH}^L(Y_+,\lambda_+,0)\rightarrow \text{ECH}^L(Y_-,\lambda_-,0)$$
    for each $L>0$ with the following properties:

    \begin{enumerate}
        \item[(i)] If $L<L'$ then the diagram 
        \begin{center}
            \begin{tikzcd}
                \text{ECH}^L(Y_+,\lambda_+,0) \arrow[r, "\Phi^L"] \arrow[swap,d,"\iota^{L,L'}"]
                & \text{ECH}^{L}(Y_-,\lambda_-,0) \arrow[d,"\iota^{L,L'}"] \\
                \text{ECH}^{L'}(Y_+,\lambda_+,0) \arrow[r,, "\Phi^{L'}" ]
                &  \text{ECH}^{L'}(Y_-,\lambda_-,0)
            \end{tikzcd}
        \end{center}
        commutes. In particular

        $$\Phi(X,\omega)=\lim_{L\rightarrow \infty} \Phi^L(X,\omega):\text{ECH}(Y_+,\lambda_+,0)\rightarrow \text{ECH}(Y_-,\lambda_-,0)$$
        is well-defined.

        \item[(ii)] If X is diffeomorphic to a product $[0,1]\times Y$, then $\Phi(X,\omega)$ is an isomorphism. 

        \item[(iii)] If $J$ is any admissible almost complex structure on $\overline X$, restricting the generic $\lambda_\pm$-compatible almost complex structure $J_\pm$ on the ends and $L>0$, then $\Phi^L(X,\omega)$ is induced by a (noncanonical) chain map
        $$\phi^L:\text{ECC}^L(Y_+,\lambda_+,0,J_+)\rightarrow \text{ECC}^L(Y_-,\lambda_-,0,J_-)$$

        such that if $\alpha_\pm$ are admissible orbit sets for $\lambda_\pm$ with $[\alpha_\pm]=0$, $\mathcal{A}(\alpha_\pm)<L$, and  $\langle\phi^L\alpha_+,\alpha_-\rangle\not =0$, then there exists a broken $J$-holomorphic current $B$ from $\alpha_+$ to $\alpha_-$ with $I(B)=0$.
        
    \end{enumerate}
\end{thm}

It is interesting to note that the this map was calculated only indirectly by using the isomorphism between ECH and Seiberg-Witten Floer Homology as in \cite{TaubesECHSW}. It is an open problem to define this map without this isomorphism.

\subsection{\texorpdfstring{$L$-tame cobordism}{L-tame cobordism}}

There are certain cases in which the symplectic cobordism induces an ECH cobordism map that is easier to understand than the ECH cobordism map in general. We are interested in a special kind of cobordisms introduced by Hutchings \cite{hutchings2016beyond} that are called $L$-tame cobordisms. We briefly explain this kind of cobordism and the related results we need about them.

Let $(Y,\lambda)$ be a contact three-manifold whose Reeb flow is nondegenerate. If $\gamma$ is an elliptic Reeb orbit, then the linearized return map $P_\lambda$ conjugates to a rotation by an angle $2\pi \theta$ for some irrational $\theta\in \mathbb{R}/\mathbb{Z}$, called the \textit{rotation angle} of $\gamma$. 

\begin{defi}
    \label{L-orbits}
    Let $L>0$ and let $\gamma$ be an embedded elliptic Reeb orbit with action $\mathcal{A}(\gamma)<L$. 
    \begin{itemize}
        \item[(i)] We say that $\gamma$ is $L$-positive if its rotation angle $\theta\in (0,\mathcal{A}(\gamma)/L) \text{ mod 1}$. 
        
        \item[(ii)] We say that $\gamma$ is $L$-negative if its rotation angle $\theta\in (-\mathcal{A}(\gamma)/L,0) \text{ mod 1}.$ 
    \end{itemize}
\end{defi}
Notice that $\gamma$ may be both $L$-positive and $L$-negative. However, we shall consider small Morse-Bott perturbations of contact forms induced by toric domains in a way that the relevant Reeb orbits have rotation angle arbitrarily close to $0 \mod 1$. Also, we will take $L>0$ large. In all the cases studied here  $\gamma$ is either $L$-positive or $L$-negative.

Now let $(X,\omega)$ be a strong symplectic cobordism from $(Y_+,\lambda_+)$ to $(Y_-,\lambda_-)$. 

\begin{defi}
    If $\alpha_\pm$ are orbit sets for $\lambda_\pm$, define $e_L(\alpha_+,\alpha_-)$ to be the total multiplicity of all elliptic orbits in $\alpha_+$ that are $L$-negative, plus the total multiplicities of all elliptic orbits in $\alpha_-$ that are $L$-positive. If $\mathcal{C}\in\mathcal{M}^J(\alpha_+,\alpha_-)$, we write $e_L(\mathcal{C})=e_L(\alpha_+,\alpha_-)$.
\end{defi}

Let $J$ be an admissible almost complex structure on $\overline X$ as defined in Section \ref{symcob}, and let $L>0$. Given an irreducible $J$-holomorphic curve $C$ from $\alpha_+$ to $\alpha_-$, denote by $g(C)$ the genus of $C$, and let $h(C)$ denote the number of ends of $C$ at hyperbolic Reeb orbits. 

\begin{defi} 
    Let $(X,\omega)$ be a symplectic cobordism from $(Y_+, \lambda_+)$ to $(Y_-, \lambda_-)$. We say that $(X,\omega)$ is $L$-tame with respect to the almost complex structure $J$ on $\bar X$ if whenever $C$ is an embedded irreducible $J$-holomorphic curve in $\mathcal{M}^J(\alpha_+,\alpha_-)$ and there exists a positive integer $d$ such that $\mathcal{A}(\alpha_\pm)<L/d$ and $I(d C)\leq 0$, then 
    \begin{equation}
    \label{L-tame inequality}
        2g(C)-2+\text{ind} (C) + h(C)+2 e_L(C)\geq 0.
    \end{equation}
\end{defi}

The importance of $L$-tame cobordisms is the following proposition analogous to Proposition \ref{BasedDiff}.

\begin{prop}
    \label{L-tameJcurves}
    Suppose that $J$ is generic and $(X,\omega)$ is  $L$-tame with respect to $J$. Let $\mathcal{C}=\sum_k d_k C_k \in \mathcal{M}^J(\alpha_+,\alpha_-)$
    be a $J$-holomorphic current on  $\bar{X}$ with $\mathcal{A}(\alpha_\pm)<L$.Then
    \begin{enumerate}
        \item[(i)] $I(\mathcal{C})\geq 0$. 
        
        \item[(ii)] If $I(\mathcal{C})=0$, then 
        \begin{itemize}
            \item $I(C_k)=0$  for each $k$.
            \item If $i\not= j$, then $C_i$ and $C_j$ do not have positive ends at covers of the same $L$-negative orbit. Moreover, $C_i$ and $C_j$ do not have negative ends at covers of the same $L$-positive orbit. 
            
            \item If $0\leq d_k' \leq d_k$ are integers, then 
            $
            I\left(\displaystyle\sum_k d_k' C_k\right)=0.
            $
        \end{itemize}
    \end{enumerate}
\end{prop}

For a proof of Proposition \ref{L-tameJcurves} see \cite[Sec. 4.3]{hutchings2016beyond}.

\subsection{ECH Capacities}
\label{subsec:capacities}

We briefly recall the definition of ECH capacities in the setting relevant
to this paper. Let $X\subset \mathbb{R}^4$ be a compact smooth
star-shaped domain, equipped with the standard symplectic form
$\omega=\sum_i dx_i\wedge dy_i$. Thus, if $Y=\partial X$, 
the standard Liouville form
\begin{equation}
    \label{standcontact}
    \lambda_{\mathrm{std}}
    =\frac{1}{2}\sum_{i=1}^2(x_i\,dy_i-y_i\,dx_i)
\end{equation}
induces a contact form $\lambda=\lambda_{\mathrm{std}}|_Y$ on $Y$.

Assume first that $\lambda$ is nondegenerate. For every non-negative
integer $k$, write $\text{ECH}_{2k}(Y,\lambda)$ for
$\text{ECH}_{2k}(Y,\lambda,0)$. This group is generated by an element
which we denote by $\zeta_k$; see \eqref{ECHlensspace}.
The $k$-th ECH spectral invariant is given by
\begin{equation}
    c_k(Y,\lambda)
    :=
    \inf\left\{
    L>0 \,\middle|\,
    \zeta_k \text{ is represented by a cycle in }
    \text{ECC}^L_*(Y,\lambda,J)
    \right\}.
\end{equation}
Here, $\text{ECC}^L_*(Y,\lambda,J)$ denotes the filtered ECH chain
complex generated by orbit sets of symplectic action at most $L$.
By \cite[Theo.~1.3]{hutchings2013proof}, the resulting number is
independent of the choice of $J$. We refer to the sequence
$\{c_k(Y,\lambda)\}_{k\geq 0}$ as the ECH spectrum of $(Y,\lambda)$.
Moreover, the spectrum is monotone in the index:
\begin{equation}
    c_k(Y,\lambda)\leq c_{k'}(Y,\lambda),
    \qquad k<k'.
\end{equation}

For a degenerate contact form, the above definition is obtained by
approximation. More precisely, choose positive functions
$f_n:Y\rightarrow\mathbb{R}_{>0}$ such that $f_n\lambda$ is
nondegenerate for every $n$ and $f_n\rightarrow 1$ in the $C^0$ topology.
We then set
\begin{equation}
    c_k(Y,\lambda)
    :=
    \lim_{n\rightarrow\infty}c_k(Y,f_n\lambda).
\end{equation}
The existence and independence of this limit follow from
\cite[Section 2.5]{hutchings2014lecture}. Finally, the ECH capacities of
the star-shaped domain $X$ are defined by
\begin{equation}
    c_k(X):=c_k(\partial X,\lambda),
    \qquad k\geq 0.
\end{equation}

\subsection{Partition Conditions}
\label{Partitions}

Let $\alpha=\{(\alpha_i,m_i)\}$ and $\beta=\{(\beta_j,n_j)\}$ be admissible orbit sets, and let $C\in \mathcal{M}^J(\alpha,\beta)$ be an irreducible $J$-holomorphic curve, where $J$ is an admissible almost complex structure either in the symplectization $\R \times Y$ or in the symplectic cobordism $\bar X$. The curve $C$ has positive ends at covers of $\alpha_i$ with total multiplicity $m_i$. Similarly, the curve $C$ has negative ends at covers of $\beta_j$ with total multiplicity $n_j$. However, the relation between the ends of $C$ and the multiplicities of $\alpha_i$ and $\beta_j$ at each end is not simple. According to \cite[Sec. 3.9]{hutchings2014lecture} this depends on the rotation angles of $\alpha_i$ and $\beta_j$ and is combinatorially described. These are called \textit{partition conditions}. Since we are dealing with orbits with very small rotation angles, the partition conditions simplify significantly, we describe this case below. 

At the positive ends we have  
    \begin{itemize}
        \item[(i)] If $\alpha_i$ is $L$-positive then precisely $m_i$ positive ends of $C$ are asymptotic to  $\alpha_i$.
        
        \item[(ii)] If $\alpha_i$ is $L$-negative then precisely one positive end of $C$ is asymptotic to $\alpha_i^{m_i}$.
    \end{itemize}

At the negative ends we have  
\begin{itemize}
        \item[(i)] If $\beta_j$ is $L$-positive then precisely one negative end of $C$ is asymptotic to $\beta_j^{n_j}$.
        
        \item[(ii)] If $\beta_j$ is $L$-negative then precisely $n_j$ negative ends of $C$ are asymptotic to  $\beta_j$.
    \end{itemize}

    The orbits considered above are elliptic. Recall that hyperbolic orbits in the orbit sets have multiplicity $1$.

\section{Boundary of a toric domain.}
\label{boundaryofatoricdomain}

In this section we want to describe the contact dynamics of the boundary of a toric domain. We explain that for a concave and a convex toric domain there are suitable perturbations that allow us to identify the closed Reeb orbits of the boundary with convex (concave) generators.

\subsection{Morse-Bott toric domains}
\label{sec:PerturbationSection}
 Let $X_{\Omega}$ be a  toric domain such that the boundary $Y=\partial X_{\Omega}$ is described by a smooth function $f:[0,a]\rightarrow [0,+\infty)$ with $f(0)=b>0$ and $f(a)=0$. We also assume that $(1,f'(x))\times (x,f(x))=f(x)-x\,f'(x)>0$ for every $x\in [0,a]$, that the derivative is constant and irrational near $0$ and near $a$, and that $f''\neq 0$ except for small intervals around $0$ and $a$. Therefore,  $Y$ is a contact manifold with contact form $\lambda=\lambda_{\text{std}}|_Y$. We call this kind of toric domain a Morse-Bott toric domain. Let us describe the closed Reeb orbits that appear in the contact manifold $Y=\partial X_\Omega$. Similar to \cite[Sec. 3.3]{hutchings2016beyond} and \cite[Sec. 5.3]{ConcaveCapacities}, we have 

\begin{itemize}
    \item The circle $\gamma_1=\{z\in\partial X_\Omega|z_2=0\}$ is an embedded elliptic Reeb orbit with action $\mathcal{A}(\gamma_1)=a$.
    
    \item The circle $\gamma_2=\{z\in\partial X_\Omega|z_1=0\}$ is an embedded elliptic Reeb orbit with action $\mathcal{A}(\gamma_2)=b$.
    
    \item For each $x\in (0,a)$ such that $f'(x)=-q/p$ is rational and $p$ and $q$ are relatively prime, the torus
    \begin{equation}
        \label{eq:Morse-Bott}
        T_{-p,q}:=\{z\in \partial X_\Omega| \pi (|z_1|^2,|z_2|^2)=(x,f(x))\}
    \end{equation}
    
    is foliated by an $S^1$-family of Reeb orbits each with homology $(q,p)\in H_1(T_{-p,q})$ in the torus, and each Reeb orbit has action $\mathcal{A}(T_{-p,q}):=(x,f(x))\times (-p,q)=qx +pf(x)$. The torus $T_{-p,q}$ is called  a Morse-Bott torus.
    \item $\gamma_1$ and $\gamma_2$ are called \textit{exceptional orbits}.
\end{itemize}

\begin{lem}
\label{lem:finiteMorse-Bott}
For every $L>0$, there are only finitely many Morse-Bott tori
$T_{-p,q}\subset\partial X_\Omega$ satisfying
$
\mathcal A(T_{-p,q})<L.
$
Equivalently, there are only finitely many relatively prime pairs
$p,q\in\mathbb Z_{>0}$ for which there exists $x\in(0,a)$ satisfying
\[
f'(x)=-\frac{q}{p},
\qquad
qx+pf(x)<L.
\]
\end{lem}

 The proof of Lemma \ref{lem:finiteMorse-Bott} relies on the fact that such values of $q,p$ for a given $L>0$ are uniformly bounded. 

   \begin{lem}
   \label{lem:czspecialorbits}
       Suppose that $X_\Omega$ is a Morse-Bott toric domain with defining function $f$. Let $\gamma_1=\partial X_\Omega\cap\{z_2=0\}$ and $\gamma_2=\partial X_\Omega\cap\{z_1=0\}$. Denote by
$\tau_1$ the symplectic trivialization of $\xi|_{\gamma_1}$ given by the symplectic frame $\{\partial_{x_2},
\partial_{y_2}\}$, and by $\tau_2$ the symplectic trivialization of $\xi|_{\gamma_2}$ given by the symplectic frame 
$\{\partial_{x_1},\partial_{y_1}\}$.
Then, for every
$k\in\mathbb Z_{>0}$,
$$\text{CZ}_{\tau_1}(\gamma_1^k)
=2\left\lfloor-\frac{k}{f'(a)}\right\rfloor+1,
\qquad
\text{CZ}_{\tau_2}(\gamma_2^k)
=2\left\lfloor-kf'(0)\right\rfloor+1.
$$
   \end{lem}
   \begin{proof} This is Step 4 of the proof of Lemma 3.3 in \cite[Sec.~ 3.3]{ConcaveCapacities}.
   \end{proof}

   Now we define the notion of toric set $\mathcal{T}=\{(T_{-p_i,q_i},m_i)\}$ which consists of pairs $(T_{-p_i,q_i},m_i)$ where $T_{-p_i,q_i}$ is Morse-Bott torus in $X_\Omega$ and $m_i$ is a positive integer. Naturally, we can define the action of these toric sets as 
   $\mathcal{A}(\mathcal{T}):=\sum_i m_i\, \mathcal{A}(T_{-p,q}).$
   
   In the following, we describe suitable perturbations over convex toric domain, weakly convex toric domains and concave toric domains that allow us to study these symplectic manifolds using embedded contact homology techniques. In the case of concave toric domains and convex toric domains these perturbations were described in \cite{ConcaveCapacities} and \cite{hutchings2016beyond}. Also, the perturbations described in Sections  \ref{sec:concavepert} and  \ref{sec:convexpert} are small modifications of \cite[Lem.~3.3]{ConcaveCapacities} and \cite[Lem.~4.5]{hutchings2016beyond}, respectively.
   
   \subsection{Perturbations of Morse-Bott tori}
   \label{sec: Morse-Bott perturbation}
   We describe the perturbation near a Morse-Bott torus of action less than $L$.

   Let $X_\Omega$ be a Morse-Bott toric domain with a smooth defining function $f:[0,a]\to[0,+\infty)$. Its boundary is the contact manifold $(Y=\partial X_\Omega,\lambda)$. Let $T_{-p,q}\subset \partial X_\Omega$ be a Morse-Bott torus as explained in Section \ref{sec:PerturbationSection}. There is a standard perturbation we can do over the Morse-Bott torus $T_{-p,q}$. Here, we explain how it works, see \cite[Sec.~4.2]{hutchings2014lecture} for details. Let $L>0$ satisfy $\mathcal A(T_{-p,q}) < L$. Then the contact form $\lambda$ can be perturbed on an arbitrarily small neighborhood of $T_{-p,q}$ so that the new Reeb flow obtained after the perturbation has on the torus $T_{-p,q}$ exactly two embedded orbits of nearly the same action $\mathcal A(T_{-p,q})$ and no other Reeb orbit of action less than $L$. Furthermore, the newly created orbits are reparametrizations of two Reeb orbits in the foliation of $T_{-p,q}$ so they have the same homology. To describe the behaviour of these two periodic orbits we define the notion of \textit{convexity} of the torus $T_{-p,q}$.

   \begin{defi}
       Let $T_{-p,q}\subset \partial X_\Omega=Y$ be a Morse-Bott torus and let $x\in (0,a)$ satisfy $T_{-p,q}=\{(z_1,z_2):(\pi|z_1|^2,\pi|z_2|^2)=(x,f(x))\}$. We say that $T_{-p,q}$ is convex if $f''(x)<0$ and we say that $T_{-p,q}$ is concave if $f''(x)>0$.
   \end{defi}

    Now we can describe the relevant information about the two orbits for our discussion. After perturbation, the torus $T_{-p,q}$ contains exactly two Reeb orbits, one elliptic $e_{-p,q}$ and one hyperbolic $h_{-p,q}$. Let $\tau = \tau_{-p,q}$ denote the Morse-Bott trivialization of $\xi$ along its Reeb orbits. More precisely, choose a nonvanishing section $v$ of the line bundle $\xi\cap TT_{-p,q},$ and complete it to the positive symplectic frame $(v,Jv)$, where $J$ is compatible with $d\lambda|_\xi$. We use the same notation $\tau$ for the induced trivializations along the perturbed orbits $e_{-p,q}$ and $h_{-p,q}$, and along their iterates. We have the following properties: 

    \begin{itemize}
    \item If $T_{-p,q}$ is convex then $e_{-p,q}$ is $L$-positive and $\text{CZ}_{\tau}(e_{-p,q}^k)=1$ for  every $k\in \Z_{>0}$ satisfying $\mathcal{A}(e_{-p,q}^k)<L$. Moreover, $\text{CZ}_{\tau}(h_{-p,q}^k)=0$ for every $k\in \Z_{>0}$. 
    \item If $T_{-p,q}$ is concave then $e_{-p,q}$ is $L$-negative and $\text{CZ}_{\tau}(e_{-p,q}^k)=-1$ for every $k\in \Z_{>0}$ satisfying $\mathcal{A}(e_{-p,q}^k)<L$. Moreover, $\text{CZ}_{\tau}(h_{-p,q}^k)=0$ for  every $k\in \Z_{>0}$.
    \item The perturbation can be done so that $\mathcal{A}(e_{-p,q})$ and $\mathcal{A}(h_{-p,q})$ are $(1/L)$-close to $\mathcal{A}(T_{-p,q})$. 
    \end{itemize}


   \begin{defi}[Multiplicative notation for admissible orbit sets]
Let $\alpha$ be an admissible orbit set supported on the exceptional
orbits and on the orbits arising from the Morse--Bott perturbations.
Then
\[
\alpha
=\gamma_1^m
 \left(\prod_{i=1}^r
 e_{-p_i,q_i}^{a_i}h_{-p_i,q_i}^{\epsilon_i}\right)
 \gamma_2^n,
\]
where $m,n,a_i\in\mathbb Z_{\geq0}$,
$\epsilon_i\in\{0,1\}$, and $a_i+\epsilon_i>0$.

Setting $m_i=a_i+\epsilon_i$, we use the compressed notation
\[
s_{-p_i,q_i}^{m_i}:=
\begin{cases}
e_{-p_i,q_i}^{m_i}, & \epsilon_i=0,\\
e_{-p_i,q_i}^{m_i-1}h_{-p_i,q_i}, & \epsilon_i=1.
\end{cases}
\]
We order the factors so that
$
-\frac{q_i}{p_i}>-\frac{q_{i+1}}{p_{i+1}}
$
in the concave case, and
$
-\frac{q_i}{p_i}<-\frac{q_{i+1}}{p_{i+1}}
$
in the convex or semi-weakly convex case.
\end{defi}




\subsection{Approximations of concave toric domains by Morse-Bott concave toric domains}
   \label{sec:concaveaprox}

    Let $X_\Omega$ be a concave toric domain with defining function $f:[0,a] \to [0,b]$ as in Definition \ref{defi_convex_concave_toric_domain}-(2). Then there exists a sequence of convex smooth functions $f_l:[0,a] \to \R_{\geq 0}$ so that the following conditions hold:
    \begin{itemize}
        \item The concave toric domain $X_{\Omega_l}$ determined by $f_l$ is Morse-Bott.
        \item $f'_l(0)\leq -l$, $f_l(a)=0$ and $-1/l\leq f'_l(a)< 0$ as $l\to \infty$.
        \item $\sup_{x\in [0,a]} |f(x)-f_l(x)| \to 0$ as $l \to \infty$. 
    \end{itemize}

    Notice that $X_{\Omega_l}$ is a $C^0$-approximation of $X_\Omega$ by a Morse-Bott concave toric domain. Notice that in the case that $X_\Omega$ is itself a Morse-Bott toric domain then the family of contact manifolds $(Y_l=\partial X_{\Omega_l},\lambda_l)$ approximate $C^0$ the contact manifold $(Y=\partial X_\Omega,\lambda)$.

\begin{lem} \label{lem:actionaprox}
Let $X_\Omega$ be a concave toric domain, and let
$X_{\Omega_l}$ be a sequence of Morse--Bott concave toric domains
as above. Fix relatively prime positive integers $p$ and $q$.
For every sufficiently large $l$, there exists a unique
$x_l\in(0,a)$ such that
$
f_l'(x_l)=-q/p.
$
Let
$
T_{-p,q}^{(l)}
:=
\mu^{-1}\bigl(x_l,f_l(x_l)\bigr)
\subset\partial X_{\Omega_l}.
$
If $\Lambda_{-p,q}$ denotes the one-edge concave generator
with edge $(-p,q)$, labeled either $e$ or $h$, then
\[
\lim_{l\to\infty}
\mathcal A_l\bigl(T_{-p,q}^{(l)}\bigr)
=
\ell_\Omega(\Lambda_{-p,q}).
\]
\end{lem}

\begin{proof}
For $l$ sufficiently large,
$
f_l'(0)<-q/p<f_l'(a),
$
and the Morse--Bott assumptions imply the existence and uniqueness
of $x_l$. Since $f_l$ is convex,
\[
\mathcal A_l\bigl(T_{-p,q}^{(l)}\bigr)
=
qx_l+pf_l(x_l)
=
\min_{x\in[0,a]}\bigl(qx+pf_l(x)\bigr).
\]
On the other hand,
\[
\ell_\Omega(\Lambda_{-p,q})
=
\min_{x\in[0,a]}\bigl(qx+pf(x)\bigr).
\]
Consequently,
\[
\left|
\mathcal A_l\bigl(T_{-p,q}^{(l)}\bigr)
-
\ell_\Omega(\Lambda_{-p,q})
\right|
\le
p\|f_l-f\|_{C^0},
\]
which converges to zero.
\end{proof}

\begin{rem}
        \label{rem:exceporbCZ}
        By Lemma \ref{lem:czspecialorbits} we deduce that the exceptional orbits $\gamma_1$ and $\gamma_2$ and all their iterates have Conley-Zehnder index strictly greater than $2l$ in $X_{\Omega_l}$.
    \end{rem}

\subsection{L-nice perturbations of concave toric domains}
\label{sec:concavepert}

Suppose that $X_\Omega$ is a concave toric domain. Let
$X_{\Omega_l}$, $l\in\mathbb{Z}_{>0}$, be a sequence of Morse--Bott
concave toric domains approximating $X_\Omega$ as in Section \ref{sec:concaveaprox}. Write
$
Y_l=\partial X_{\Omega_l},
$ $
\lambda_l=\lambda_{\mathrm{std}}|_{Y_l},
$
and denote by $\mathcal{A}_l$ the action associated with $\lambda_l$.

Fix $L>0$. For each $l$, consider the Morse--Bott tori
$T_{-p,q}^{(l)}\subset Y_l$ satisfying
$
\mathcal{A}_l\bigl(T_{-p,q}^{(l)}\bigr)<L.
$
There are only finitely many such tori. We perturb $\lambda_l$ in
pairwise disjoint neighborhoods of these tori, as described in
Section \ref{sec: Morse-Bott perturbation}. The perturbation can be chosen so that the resulting
contact form
$
\lambda_{l,L}=g_{l,L}\lambda_l
$
is $L$-nondegenerate and arbitrarily close to $\lambda_l$ in the
$C^\infty$ topology. Moreover, every torus $T_{-p,q}^{(l)}$ of action
less than $L$ gives rise to exactly two embedded Reeb orbits of action
less than $L$,
$
e_{-p,q}^{(l)}
$ and $
h_{-p,q}^{(l)},
$
where $e_{-p,q}^{(l)}$ is elliptic and $h_{-p,q}^{(l)}$ is positive
hyperbolic. Their actions can be chosen arbitrarily close to
$\mathcal{A}_l(T_{-p,q}^{(l)})$, and no other Reeb orbit of action less
than $L$ is created by the perturbation.

Let $\tau_{-p,q}^{(l)}$ be the Morse--Bott trivialization introduced in
Section \ref{sec: Morse-Bott perturbation}. Since $X_{\Omega_l}$ is concave, the perturbation can be
chosen so that every elliptic orbit $e_{-p,q}^{(l)}$ of action less
than $L$ is $L$-negative. Consequently,

$$
\operatorname{CZ}_{\tau_{-p,q}^{(l)}}
\left(\left(e_{-p,q}^{(l)}\right)^k\right)=-1
\quad \mbox{whenever}
\quad
\mathcal{A}_{l,L}
\left(\left(e_{-p,q}^{(l)}\right)^k\right)<L,
$$

and

$$
\operatorname{CZ}_{\tau_{-p,q}^{(l)}}
\left(h_{-p,q}^{(l)}\right)=0.
$$

Moving $Y_l$ a sufficiently small amount along the Liouville flow, we
can realize $\lambda_{l,L}$ as the restriction of
$\lambda_{\mathrm{std}}$ to a smooth star-shaped hypersurface
$
Y_{l,L}=\partial X_{\Omega_{l,L}}
$
which is arbitrarily close to $Y_l$ in the $C^\infty$ topology. The
domain $X_{\Omega_{l,L}}$ need not be toric. We call
$(Y_{l,L},\lambda_{l,L})$ an $L$-nice perturbation of
$(Y_l,\lambda_l)$.

We now describe the correspondence between admissible orbit sets and
concave generators. Let $\alpha$ be an admissible orbit set which does
not contain an exceptional orbit. Then $\alpha$ can be written as

$$
\alpha=
\prod_{i=1}^r
\left(e_{-p_i,q_i}^{(l)}\right)^{a_i}
\left(h_{-p_i,q_i}^{(l)}\right)^{\epsilon_i},
$$
where
$
a_i\in\mathbb{Z}_{\geq 0},$ $
\epsilon_i\in\{0,1\}$, $
a_i+\epsilon_i>0,
$
and the pairs $(p_i,q_i)$ are distinct pairs of relatively prime
positive integers. We order the factors so that

$$
\frac{q_1}{p_1}<\frac{q_2}{p_2}<\cdots<\frac{q_r}{p_r}.
$$

Set $m_i=a_i+\epsilon_i$. The orbit set $\alpha$ determines a
concave generator $\Lambda_\alpha$ whose $i$-th edge is
$
v_i=m_i(-p_i,q_i).
$
The edge $v_i$ is labeled $e$ if $\epsilon_i=0$, and it is labeled
$h$ if $\epsilon_i=1$. Thus, an edge of multiplicity $m_i$ labeled
$h$ corresponds to the factor
$
\left(e_{-p_i,q_i}^{(l)}\right)^{m_i-1}
 h_{-p_i,q_i}^{(l)}.
$

For a concave generator $\Lambda$, define
$$
\nu_e(\Lambda)
:=
\#\{\text{edges of $\Lambda$ having positive elliptic multiplicity}\}.
$$

Equivalently, $\nu_e(\Lambda_\alpha)$ is the number of distinct
elliptic Reeb orbits which occur in $\alpha$ with positive
multiplicity.

Given $k_0\in\mathbb{Z}_{>0}$, there exists $L_0>0$ with the following
property. For every $L\geq L_0$, there exists
$l_0=l_0(L,k_0)>0$ such that, for every $l\geq l_0$, the $L$-nice
perturbation can be chosen so that the following assertions hold for
every $0\leq k\leq k_0$.

\begin{enumerate}

\item Every admissible orbit set generator of
$ECC_k^L(Y_{l,L},\lambda_{l,L},0)$ contains no exceptional orbit.

\item The correspondence
$
\alpha\longmapsto\Lambda_\alpha
$
defines a bijection between the orbit set generators of
$ECC_k^L(Y_{l,L},\lambda_{l,L},0)$ and the concave generators
$\Lambda$ satisfying $I(\Lambda)=k$.

\item Under this correspondence,
$
I(\alpha)=I(\Lambda_\alpha)
$
and
$
\left|
\mathcal{A}_{l,L}(\alpha)-\ell_\Omega(\Lambda_\alpha)
\right|<1/L.
$

\item Every nonexceptional embedded elliptic Reeb orbit of action less
than $L$ is $L$-negative.

\item Let $\tau$ denote the collection of Morse--Bott trivializations
along the Reeb orbits appearing in $\alpha$. Then

$
c_\tau(\alpha)
=x(\Lambda_\alpha)+y(\Lambda_\alpha),
$
$
Q_\tau(\alpha)=2A(\Lambda_\alpha),
$
and
$
\text{CZ}^{I}_\tau(\alpha)=-e(\Lambda_\alpha).
$
Consequently,

$$
I(\alpha)
=2A(\Lambda_\alpha)
+x(\Lambda_\alpha)
+y(\Lambda_\alpha)
-e(\Lambda_\alpha)
=I(\Lambda_\alpha).
$$

\item The $J_0$-index is given by

$$
J_0(\alpha)
=I(\Lambda_\alpha)
-2x(\Lambda_\alpha)
-2y(\Lambda_\alpha)
+\nu_e(\Lambda_\alpha).
$$

\end{enumerate}

We briefly explain these assertions. For fixed $k_0$, there are only
finitely many concave generators $\Lambda$ satisfying
$I(\Lambda)\leq k_0$. Hence only finitely many primitive vectors
$(-p,q)$ and finitely many multiplicities occur among their edges. By
Lemma \ref{lem:actionaprox}, after choosing $l$ sufficiently large, the actions of the
corresponding Morse--Bott tori are uniformly close to the
$\Omega$-lengths of the associated edges. The perturbation can then be
chosen sufficiently small so that the total action error for every
orbit set of index at most $k_0$ is less than $1/L$. Taking $L_0$
larger than the $\Omega$-length of every concave generator of index at
most $k_0$ gives the bijection in item (2).

By Remark \ref{rem:exceporbCZ}, after increasing $l_0$ if necessary, every orbit set
containing an exceptional orbit has ECH index greater than $k_0$. This
proves item (1).

Finally, with respect to the Morse--Bott trivializations, every
relevant iterate of an elliptic orbit has Conley-Zehnder index $-1$,
while every hyperbolic orbit has Conley--Zehnder index zero. Therefore,
$
\text{CZ}^{I}_\tau(\alpha)=-e(\Lambda_\alpha).
$
The standard toric computation gives
$
c_\tau(\alpha)=x(\Lambda_\alpha)+y(\Lambda_\alpha)
$
and
$
Q_\tau(\alpha)=2A(\Lambda_\alpha).
$
This proves the formula for the ECH index. If an elliptic orbit occurs
with multiplicity $m>0$, then its contribution to
$\text{CZ}^{I}_\tau$ is $-m$, whereas its contribution to
$\text{CZ}^{J}_\tau$ is $-(m-1)$. Hence
$
\text{CZ}^{J}_\tau(\alpha)
=-e(\Lambda_\alpha)+\nu_e(\Lambda_\alpha).
$
Using the definition of $J_0$, we obtain
$$
\begin{aligned}
J_0(\alpha)
&=-c_\tau(\alpha)+Q_\tau(\alpha)
  +\text{CZ}^{J}_\tau(\alpha)\\
&=-x(\Lambda_\alpha)-y(\Lambda_\alpha)
  +2A(\Lambda_\alpha)-e(\Lambda_\alpha)
  +\nu_e(\Lambda_\alpha)\\
&=I(\Lambda_\alpha)-2x(\Lambda_\alpha)
  -2y(\Lambda_\alpha)+\nu_e(\Lambda_\alpha).
\end{aligned}
$$

\subsection{Approximations of convex and semi-weakly convex toric domains by Morse-Bott toric domains}

Let $X_\Omega$ be a convex or semi-weakly convex toric domain, and let
$f:[0,a]\to[0,+\infty)$ be the concave function defining $\Omega$ as in
Definition \ref{defi_convex_concave_toric_domain}. We assume that $f$ is differentiable at $0$. If $f(a)=0$,
we also assume that $f$ is differentiable at $a$.

Choose $m\in\mathbb Z$ such that
$
f'(0)\leq -m.
$
If $f(a)=0$, choose $n\in\mathbb Z_{\geq 0}$ such that
$
nf'(a)\geq -1
$ and $mn<1$.
The meaning of these two integers is the following. The condition that an
edge $(-p,q)$ be vertically $m$-adapted is
$
q\geq mp,
$
while the condition that it be horizontally $n$-adapted is
$
p\geq nq.
$

We now construct smooth concave functions $f_l:[0,a]\to[0,+\infty)$,
$l\in\mathbb Z_{>0}$, such that the associated toric domains
$X_{\Omega_l}$ are Morse--Bott and approximate $X_\Omega$.

Suppose first that $f(a)=0$. We choose $f_l$ so that
$
f_l(0)=f(0),
f_l(a)=0,
$
and
$$
\lim_{l\to\infty}
\sup_{x\in[0,a]}|f_l(x)-f(x)|=0.
$$

We also require that $f_l'(0)$ is irrational for every $l$ and that
$
f_l'(0)<-m,
f_l'(0)\to -m
$ as $l\to \infty$.
If $n>0$, we require that $f_l'(a)$ is irrational for every $l$ and that
$
f_l'(a)>-\frac{1}{n},
f_l'(a)\to -\frac{1}{n},
$ as $l\to \infty$.
If $n=0$, we require instead that $f_l'(a)$ is irrational for every $l$ and
$
f_l'(a)\to -\infty,
$ as $l\to \infty$.

Suppose now that $f(a)>0$. In this case we set $n=0$. We choose $f_l$ so
that
$
f_l(0)=f(0),
f_l(a)=0,
$
and, for every $a'\in(0,a)$,
$$
\lim_{l\to\infty}
\sup_{x\in[0,a']}|f_l(x)-f(x)|=0.
$$

We require again that $f_l'(0)$ is irrational and satisfies
$
f_l'(0)<-m,
f_l'(0)\to -m,
$ as $l\to \infty$,
and that $f_l'(a)$ is irrational and satisfies
$
f_l'(a)\to -\infty.
$ as $l\to \infty$.
The functions $f_l$ can be chosen so that the compact regions $\Omega_l$
converge to $\Omega$ in the Hausdorff topology. When $f(a)>0$, the part of
the graph of $f_l$ near $x=a$, whose slope tends to $-\infty$, approximates
the vertical segment
$
\{a\}\times[0,f(a)]
$
contained in $\partial^+\Omega$.

Such a sequence can be obtained by smoothing the graph of $f$, modifying it
in intervals whose lengths tend to zero near the endpoints, and then making
an arbitrarily small perturbation which preserves concavity and makes the
resulting toric domains Morse-Bott.

The next lemma describes the convergence of the actions of the interior
Morse-Bott tori.

\begin{lem}\label{lem_approx_action_weak_convex}
Suppose that $X_\Omega$ is a convex or semi-weakly convex toric domain, and
let $X_{\Omega_l}$ be a sequence of Morse-Bott semi-weakly convex toric
domains as above. Let $p\in\mathbb Z_{>0}$ and $q\in\mathbb Z$ be relatively
prime and satisfy
$
q>mp,
p>nq.
$
Then, for every sufficiently large $l$, there exists a unique
$x_l\in(0,a)$ such that
$
f_l'(x_l)=-\frac{q}{p}.
$
Let $T_{-p,q}^{(l)}$ be the Morse--Bott torus lying over
$(x_l,f_l(x_l))$. Let $\Lambda_{-p,q}$ be the one-edge convex or
semi-weakly convex generator with edge $(-p,q)$, labeled either $e$ or $h$.
Then
$
\lim_{l\to\infty}
\mathcal A_l\bigl(T_{-p,q}^{(l)}\bigr)
=
\ell_\Omega(\Lambda_{-p,q}),
$
where $\mathcal A_l$ denotes the action on $\partial X_{\Omega_l}$.
\end{lem}

\begin{proof}
The inequality $q>mp$ is equivalent to
$
-\frac{q}{p}<-m.
$
Since $f_l'(0)\to-m$, it follows that
$
-\frac{q}{p}<f_l'(0)
$
for every sufficiently large $l$.
Suppose that $n>0$. The inequality $p>nq$ is equivalent to
$
-\frac{q}{p}>-\frac{1}{n}.
$
Since $f_l'(a)\to-1/n$ and $f_l'(a)>-1/n$, it follows that
$
f_l'(a)<-\frac{q}{p}
$
for every sufficiently large $l$. If $n=0$, the same conclusion follows
from $f_l'(a)\to-\infty$. Therefore, for every sufficiently large $l$,
$
f_l'(a)<-\frac{q}{p}<f_l'(0).
$
Since $f_l'$ is continuous and nonincreasing, there exists
$x_l\in(0,a)$ such that
$
f_l'(x_l)=-\frac{q}{p}.
$
This point is unique. Indeed, otherwise $f_l'$ would be constant and equal
to the rational number $-q/p$ on a nontrivial interval, contrary to the
Morse--Bott assumptions on $X_{\Omega_l}$.

Define
$
\Phi_l(x):=qx+pf_l(x).
$
Since $f_l$ is concave and $p>0$, the function $\Phi_l$ is concave. Moreover,
$
\Phi_l'(x_l)=q+pf_l'(x_l)=0.
$
Hence $x_l$ is a point of maximum of $\Phi_l$, and therefore
$
\mathcal A_l\bigl(T_{-p,q}^{(l)}\bigr)
=
qx_l+pf_l(x_l)
=
\max_{x\in[0,a]}\bigl(qx+pf_l(x)\bigr).
$

By the definition of the $\Omega$-length,
$
\ell_\Omega(\Lambda_{-p,q})
=
\max_{(x,y)\in\Omega}(qx+py).
$
Similarly,
$$
\mathcal A_l\bigl(T_{-p,q}^{(l)}\bigr)
=
\max_{(x,y)\in\Omega_l}(qx+py).
$$

Since $\Omega_l\to\Omega$ in the Hausdorff topology, the maxima of the fixed
linear function $(x,y)\mapsto qx+py$ converge. Thus
$
\lim_{l\to\infty}
\mathcal A_l\bigl(T_{-p,q}^{(l)}\bigr)
=
\ell_\Omega(\Lambda_{-p,q}).
$
\end{proof}

\begin{rem}
We can choose $m=0$ if and only if $X_\Omega$ is a convex toric domain.
Indeed, if $m=0$, then $f'(0)\leq0$. Since $f$ is concave, its derivative is
nonincreasing, and hence $f$ is nonincreasing on $[0,a]$. Thus $X_\Omega$ is
convex. Conversely, if $X_\Omega$ is convex, then $f$ is nonincreasing, so
$f'(0)\leq0$, and we may choose $m=0$.
\end{rem}

\subsection{L-nice perturbations of convex and semi-weakly convex toric domains}
\label{sec:convexpert}

Let $X_\Omega$ be a convex or semi-weakly convex toric domain, and let
$f:[0,a]\to[0,+\infty)$ be the concave function defining $\Omega$.
Choose $m\in\mathbb Z$ and $n\in\mathbb Z_{\geq 0}$ as in
Section~5.5. Thus,
$
f'(0)\leq -m,
$
and, when $f(a)=0$,
$
nf'(a)\geq  -1.
$
When $f(a)>0$, we set $n=0$. If $n>0$ we require $m<1/n.$

Let $X_{\Omega_l}$ be a sequence of Morse--Bott semi-weakly convex
toric domains approximating $X_\Omega$ as in Section~5.5. We use the
following endpoint conventions for the defining functions $f_l$:
$
f_l'(0)<-m,$ and $f_l'(0)\to -m
$ as $l\to \infty$.
If $f(a)=0$ and $n>0$, we require
$
f_l'(a)>-\frac{1}{n}$ and $
f_l'(a)\to -\frac{1}{n}
$ as $l\to \infty$.
If $n=0$, we require
$
f_l'(a)\to -\infty.
$ as $l\to \infty$.

The endpoint slopes are chosen irrational. Notice that, when $n>0$,
the sequence $f_l'(a)$ approaches $-1/n$ from above. This convention is
needed in order that the exceptional orbit over the right endpoint be
$L$-positive after the change of trivialization below.

Write
$
Y_l=\partial X_{\Omega_l},
\lambda_l=\lambda_{\mathrm{std}}|_{Y_l},
$
and denote the corresponding action by $\mathcal A_l$.
Fix $L>0$. For each $l$, perturb $\lambda_l$ in pairwise disjoint
neighborhoods of all Morse--Bott tori of action less than $L$.
The perturbation can be chosen in the form
$
\lambda_{l,L}=g_{l,L}\lambda_l,
$
where $g_{l,L}>0$ is arbitrarily close to $1$ in the
$C^\infty$ topology. We choose $g_{l,L}$ so that
$\lambda_{l,L}$ is $L$-nondegenerate.

Since $\lambda_{l,L}=g_{l,L}\lambda_l$, the contact form
$\lambda_{l,L}$ can be realized as the restriction of
$\lambda_{\mathrm{std}}$ to a star-shaped hypersurface
$Y_{l,L}$ obtained by moving $Y_l$ along the Liouville flow.

We now describe the orbit sets of $\lambda_{l,L}$. Let $\gamma_1^{(l)}$
and $\gamma_2^{(l)}$ denote the exceptional orbits over the right and
left endpoints, respectively. An admissible orbit set of action less than $L$ can be written as
$$
\alpha
=
\left(\gamma_1^{(l)}\right)^{a_0}
\prod_{i=1}^{r}
\left(e_{-p_i,q_i}^{(l)}\right)^{a_i}
\left(h_{-p_i,q_i}^{(l)}\right)^{\epsilon_i}
\left(\gamma_2^{(l)}\right)^{a_{r+1}},
$$
where
$
a_0,a_{r+1},a_i\in\mathbb Z_{\geq0},
\epsilon_i\in\{0,1\}$ and $
a_i+\epsilon_i>0.
$
The pairs $(p_i,q_i)$ are distinct, $p_i>0$, $q_i\in\mathbb Z$, and
$\gcd(p_i,q_i)=1$. We order them so that
$$
\frac{q_1}{p_1}
>
\frac{q_2}{p_2}
>
\cdots
>
\frac{q_r}{p_r}.
$$
Set
$
m_i=a_i+\epsilon_i.
$

The orbit set $\alpha$ determines a semi-weakly convex generator
$\Lambda_\alpha$ as follows. The factor
$\left(\gamma_1^{(l)}\right)^{a_0}$ determines the edge
$
a_0(-n,1),
$
which is labeled $e$. The factor
$$
\left(e_{-p_i,q_i}^{(l)}\right)^{a_i}
\left(h_{-p_i,q_i}^{(l)}\right)^{\epsilon_i}
$$
determines the edge
$
m_i(-p_i,q_i),
$
which is labeled $e$ when $\epsilon_i=0$ and labeled $h$ when
$\epsilon_i=1$. Finally,
$\left(\gamma_2^{(l)}\right)^{a_{r+1}}$ determines the edge
$
a_{r+1}(-1,m),
$
which is labeled $e$. Edges of zero multiplicity are omitted.

The generator $\Lambda_\alpha$ is $(n,m)$-adapted. Indeed, every
interior Morse--Bott torus satisfies
$
m<\frac{q_i}{p_i}<\frac{1}{n}
$
when $n>0$, while for $n=0$ only the first inequality is relevant. The
exceptional orbits correspond to the two boundary directions
$(-n,1)$ and $(-1,m)$.

For any labeled generator $\Lambda$, define
$$
\nu_e(\Lambda)
=
\#\{\text{edges of $\Lambda$ having positive elliptic multiplicity}\}.
$$

Thus $\nu_e(\Lambda_\alpha)$ is the number of distinct elliptic Reeb
orbits appearing in $\alpha$ with positive multiplicity.

We next record the precise form of the correspondence needed below.
Fix $k_0\in\mathbb Z_{\geq0}$. Choose $L>0$ which is not the
$\Omega$-length of an $(n,m)$-adapted semi-weakly convex generator of
ECH index at most $k_0$. Then, for every sufficiently small
$\epsilon>0$, there exists $l_0=l_0(L,k_0,\epsilon)$ such that, for
$l\geq l_0$, the perturbation can be chosen with the following
properties for every $0\leq k\leq k_0$:

\begin{enumerate}

\item
The map
$
\alpha\mapsto \Lambda_\alpha
$
is a bijection between the orbit-set generators of
$
ECC_k^L(Y_{l,L},\lambda_{l,L},0)
$
and the $(n,m)$-adapted semi-weakly convex generators $\Lambda$
satisfying
$
I(\Lambda)=k$ and $
\ell_\Omega(\Lambda)<L.
$

\item
Under this correspondence,
$
I(\alpha)=I(\Lambda_\alpha)
$
and
$
\left|
\mathcal A_{l,L}(\alpha)-\ell_\Omega(\Lambda_\alpha)
\right|
<\epsilon.
$

\item
Every embedded elliptic Reeb orbit of $\lambda_{l,L}$ having action
less than $L$ is $L$-positive.

\item
With respect to the trivializations specified below,
$$
J_0(\alpha)
=
I(\Lambda_\alpha)
-2x(\Lambda_\alpha)
-2y(\Lambda_\alpha)
-\nu_e(\Lambda_\alpha).
$$

\end{enumerate}

We call such a perturbation an $L$-nice perturbation of $X_\Omega$ with
respect to the grading bound $k_0$.

We now verify the index formulas. Let $\tau_1$ and $\tau_2$ be the
trivializations along the exceptional orbits from Lemma \ref{lem:czspecialorbits}. Define
new trivializations $\overline\tau_1$ and $\overline\tau_2$ by
$
\overline\tau_1-\tau_1=-n$ and $\overline\tau_2-\tau_2=-m.
$

For these trivializations, representatives of the rotation angles of the exceptional orbits are
$
\theta_{1,l}
=
-\frac{1}{f_l'(a)}-n
$
and
$
\theta_{2,l}
=
-f_l'(0)-m.
$

Suppose first that $n>0$. Since
$
f_l'(a)>-\frac{1}{n}
$
and
$
f_l'(a)\to-\frac{1}{n},
$
we have
$
-\frac{1}{f_l'(a)}>n
$
and consequently
$
\theta_{1,l}>0,
\theta_{1,l}\to 0.
$
If $n=0$, the same conclusion follows from
$
f_l'(a)\to-\infty.
$
Moreover, since
$
f_l'(0)<-m
$
and
$
f_l'(0)\to-m,
$
we have
$
\theta_{2,l}>0,
\theta_{2,l}\to0.
$

The actions of $\gamma_1^{(l)}$ and $\gamma_2^{(l)}$ converge respectively to $a$ and $f(0)$. In particular, these actions are bounded away from zero for all sufficiently large $l$. Therefore, after increasing $l$ if necessary, depending on $L$, we may assume that
$
0
<
\theta_{1,l}
<
\frac{
\mathcal{A}_{l,L}\bigl(\gamma_1^{(l)}\bigr)
}{L}
$
and
$
0
<
\theta_{2,l}
<
\frac{
\mathcal{A}_{l,L}\bigl(\gamma_2^{(l)}\bigr)
}{L}.
$
It follows from Definition \ref{L-orbits} that both exceptional orbits are $L$-positive.

Moreover, if $i\in\{1,2\}$ and $j\in\mathbb{Z}_{>0}$ satisfy
$
\mathcal{A}_{l,L}
\left(
\left(\gamma_i^{(l)}\right)^j
\right)
<L,
$
then
$
j\,
\mathcal{A}_{l,L}\bigl(\gamma_i^{(l)}\bigr)
<L.
$
Hence
$
j
<
\frac{L}{
\mathcal{A}_{l,L}\bigl(\gamma_i^{(l)}\bigr)
}.
$
Together with the inequalities above, this gives
$
0<j\theta_{i,l}<1.
$
Therefore,
$
\text{CZ}_{\overline{\tau}_1}
\left(
\left(\gamma_1^{(l)}\right)^j
\right)
=
2\left\lfloor j\theta_{1,l}\right\rfloor+1
=
1
$
and
$
\text{CZ}_{\overline{\tau}_2}
\left(
\left(\gamma_2^{(l)}\right)^j
\right)
=
2\left\lfloor j\theta_{2,l}\right\rfloor+1
=
1
$
for every iterate having action less than $L$.

Together with the Morse-Bott perturbations of the interior tori, this shows that every embedded elliptic Reeb orbit of action less than $L$ is $L$-positive.

Lemma \ref{lem:changetrivialization} gives

$$
c_{\overline\tau_1}
\left(\left(\gamma_1^{(l)}\right)^j\right)
=j(n+1),
\qquad
Q_{\overline\tau_1}
\left(\left(\gamma_1^{(l)}\right)^j\right)
=j^2n,
$$

and

$$
c_{\overline\tau_2}
\left(\left(\gamma_2^{(l)}\right)^j\right)
=j(m+1),
\qquad
Q_{\overline\tau_2}
\left(\left(\gamma_2^{(l)}\right)^j\right)
=j^2m.
$$

Along the nonexceptional orbits we use the Morse--Bott
trivializations. For $s_{-p,q}^{(l)}$ equal to either
$e_{-p,q}^{(l)}$ or $h_{-p,q}^{(l)}$, the standard toric computation
gives

$$
c_\tau(s_{-p,q}^{(l)})=p+q,
\qquad
Q_\tau(s_{-p,q}^{(l)})=pq.
$$

If both orbits produced from the same Morse--Bott torus occur, then
$
\text{link}
\left(e_{-p,q}^{(l)},h_{-p,q}^{(l)}\right)=pq.
$
The remaining linking numbers are
$$
\text{link}(\gamma_1^{(l)},\gamma_2^{(l)})=1,
\quad
\text{link}(\gamma_1^{(l)},s_{-p,q}^{(l)})=p,
\quad
\text{link}(\gamma_2^{(l)},s_{-p,q}^{(l)})=q,
$$
and, for distinct pairs $(p,q)$ and $(p',q')$,
$
\text{link}(s_{-p,q}^{(l)},s_{-p',q'}^{(l)})
=
\max\{p'q,pq'\}.
$

For an orbit set $\alpha=\{(\alpha_i,d_i)\}$, the relative
intersection pairing is
$$
Q_\tau(\alpha)
=
\sum_i d_i^2Q_\tau(\alpha_i)
+
\sum_{i\neq j}d_id_j
\operatorname{link}(\alpha_i,\alpha_j).
$$

Substituting the formulas above and collecting the terms according to
the edges of $\Lambda_\alpha$, we obtain
$
c_\tau(\alpha)
=
x(\Lambda_\alpha)+y(\Lambda_\alpha),
$
and
$
Q_\tau(\alpha)
=
2\left(
\text{Area}(B^+_{\Lambda_\alpha})
-
\text{Area}(B^-_{\Lambda_\alpha})
\right).
$

Since all relevant elliptic orbits are $L$-positive and the
hyperbolic orbits have Conley--Zehnder index zero,
$
\text{CZ}^I_\tau(\alpha)
=e(\Lambda_\alpha).
$
Consequently,
$$
\begin{aligned}
I(\alpha)
&=c_\tau(\alpha)+Q_\tau(\alpha)
  +\text{CZ}^I_\tau(\alpha)\\
&=2\left(
\text{Area}(B^+_{\Lambda_\alpha})
-
\text{Area}(B^-_{\Lambda_\alpha})
\right)
+x(\Lambda_\alpha)+y(\Lambda_\alpha)
+e(\Lambda_\alpha)\\
&=2\left(
\mathcal L(B^+_{\Lambda_\alpha})
-
\mathcal L(B^-_{\Lambda_\alpha})
-1
\right)
-h(\Lambda_\alpha)\\
&=I(\Lambda_\alpha).
\end{aligned}
$$

The third equality follows from the signed version of Pick's formula
for semi-weakly convex integral paths:
$$
2\left(
\text{Area}(B_\Lambda^+)
-
\text{Area}(B_\Lambda^-)
\right)
+x(\Lambda)+y(\Lambda)+m(\Lambda)
=
2\left(
\mathcal L(B_\Lambda^+)-\mathcal L(B_\Lambda^-)-1
\right).
$$

For an elliptic orbit appearing with multiplicity $d>0$, its
contribution to $\text{CZ}^I_\tau$ is $d$, whereas its
contribution to $\text{CZ}^J_\tau$ is $d-1$. Therefore,
$
\text{CZ}^J_\tau(\alpha)
=e(\Lambda_\alpha)-\nu_e(\Lambda_\alpha).
$

Using the definition of $J_0$, we obtain
$$
\begin{aligned}
J_0(\alpha)
&=-c_\tau(\alpha)+Q_\tau(\alpha)
  +\text{CZ}^J_\tau(\alpha)\\
&=I(\Lambda_\alpha)
  -2x(\Lambda_\alpha)
  -2y(\Lambda_\alpha)
  -\nu_e(\Lambda_\alpha).
\end{aligned}
$$

It remains to justify the action statement and the bijection. We first prove the required finiteness statement. Let

$$
\mathcal{G}_{k_0,L+1}
=
\left\{
\Lambda:
\begin{array}{l}
\Lambda \text{ is an }(n,m)\text{-adapted semi-weakly convex generator},\\
0\leq I(\Lambda)\leq k_0,\quad
\ell_\Omega(\Lambda)\leq L+1
\end{array}
\right\}.
$$

We claim that $\mathcal{G}_{k_0,L+1}$ is finite. Consider a primitive nonvertical edge
$
v=(-p,q),
p\in\mathbb{Z}_{>0},
q\in\mathbb{Z}.
$
Since $(0,f(0))\in\Omega$, the definition of the $\Omega$-length gives

$$
\ell_\Omega(v)
=
\max_{(x,y)\in\Omega}(qx+py)
\geq p f(0).
$$
Consequently, the condition $\ell_\Omega(\Lambda)\leq L+1$ gives a uniform bound on $p$ for every edge of $\Lambda$. If $q\geq0$, then $(a,0)\in\Omega$ and hence
$
\ell_\Omega(v)\geq qa,
$
which gives a uniform upper bound on $q$. If $q<0$, the vertical $m$-adaptedness condition gives
$
q\geq mp.
$
Since $p$ is already bounded, there are again only finitely many possible values of $q$. The only possible vertical direction is the exceptional direction $(0,1)$, which occurs when $n=0$. It follows that only finitely many primitive edge directions can occur.

The $\Omega$-length of every primitive edge which can occur is positive. Since the $\Omega$-length is additive under concatenation of edges, the multiplicity of every edge is bounded by the condition
$
\ell_\Omega(\Lambda)\leq L+1.
$

There are only two possible labels for each edge. This proves that $\mathcal{G}_{k_0,L+1}$ is finite.

Since $\mathcal{G}_{k_0,L+1}$ is finite, we may increase $l$
so that every primitive nonboundary edge direction occurring in
a generator in $\mathcal{G}_{k_0,L+1}$ is realized by a
Morse-Bott torus, by Lemma~\ref{lem_approx_action_weak_convex}. Moreover, if $\ell_\Omega(\Lambda)<L$, then the
$\Omega$-length of each edge of $\Lambda$ is less than $L$.
Since only finitely many such edges occur, after increasing $l$
once more, Lemma~\ref{lem_approx_action_weak_convex} implies that
the corresponding Morse--Bott tori all have action less than $L$.
Hence all of them are included in the perturbation defining
$\lambda_{l,L}$. 

The two boundary directions
$(-n,1)$ and $(-1,m)$ are realized by the exceptional orbits.
Consequently, every generator
$\Lambda\in\mathcal{G}_{k_0,L+1}$ satisfying
$
\ell_\Omega(\Lambda)<L
$
determines a unique admissible orbit set: an edge of multiplicity $d$ labeled $e$ corresponds to
$
\left(e_{-p,q}^{(l)}\right)^d,
$
while an edge of multiplicity $d$ labeled $h$ corresponds to
$
\left(e_{-p,q}^{(l)}\right)^{d-1}h_{-p,q}^{(l)}.
$
The boundary edges are necessarily labeled $e$ and correspond to
the appropriate powers of the exceptional orbits.

We next show that it is enough to consider generators in this finite collection. For $l$ sufficiently large, any orbit-set generator $\alpha$ with
$
0\leq I(\alpha)\leq k_0
$
and
$
\mathcal{A}_{l,L}(\alpha)<L
$
corresponds to a generator $\Lambda_\alpha$ satisfying
$
\ell_\Omega(\Lambda_\alpha)<L+1.
$
Indeed, every simple orbit appearing in $\alpha$ has action less than $L$. The action bounds above, applied to the approximating domains $\Omega_l$, give uniform bounds on the corresponding integers $p$ and $q$ and on their multiplicities. Thus only finitely many orbit directions and multiplicities can occur. Lemma~\ref{lem_approx_action_weak_convex}, together with the convergence of the actions of the exceptional orbits and the fact that the Morse--Bott perturbation can be chosen arbitrarily small, then implies
$
\left|
\mathcal{A}_{l,L}(\alpha)
-
\ell_\Omega(\Lambda_\alpha)
\right|
<1
$
for every such orbit set, after increasing $l$ if necessary. Hence $\ell_\Omega(\Lambda_\alpha)<L+1$.

Since $L$ is not the $\Omega$-length of any $(n,m)$-adapted semi-weakly convex generator of ECH index at most $k_0$, the number
$
\delta
=
\min_{\Lambda\in\mathcal{G}_{k_0,L+1}}
\left|
\ell_\Omega(\Lambda)-L
\right|
$
is positive. Choose
$
0<\eta<
\min\left\{
\epsilon,
\frac{\delta}{2},
1
\right\}.
$
Because $\mathcal{G}_{k_0,L+1}$ is finite, Lemma \ref{lem_approx_action_weak_convex} and the convergence of the exceptional actions imply that, after increasing $l$ and choosing the Morse--Bott perturbation sufficiently small,
$
\left|
\mathcal{A}_{l,L}(\alpha)
-
\ell_\Omega(\Lambda_\alpha)
\right|
<\eta
$
for every orbit set corresponding to a generator in $\mathcal{G}_{k_0,L+1}$.

It follows that
$
\mathcal{A}_{l,L}(\alpha)<L
\Leftrightarrow
\ell_\Omega(\Lambda_\alpha)<L.
$
Indeed, if $\ell_\Omega(\Lambda_\alpha)<L$, then
$
\ell_\Omega(\Lambda_\alpha)\leq L-\delta,
$
and therefore
$
\mathcal{A}_{l,L}(\alpha)
<
L-\delta+\eta
<
L.
$

Conversely, if $\ell_\Omega(\Lambda_\alpha)>L$, then
$
\ell_\Omega(\Lambda_\alpha)\geq L+\delta,
$
and hence
$
\mathcal{A}_{l,L}(\alpha)
>
L+\delta-\eta
>
L.
$

Thus the action cutoff is preserved. This proves the bijection in assertion~(1) and the action estimate in assertion~(2).

\begin{proof}[Proof of Theorem~\ref{thm:semiweak_capacities}]
The compact region $\Omega$ is convex because it is the hypograph of the
concave function $f$. Hence Corollary~A.12 of \cite{cristofaro2019symplectic} applies and gives
$$
c_k(X_\Omega)
=
\min\left\{
\ell_\Omega(\Lambda):
\mathcal L(B_\Lambda)=k+1
\right\}.
$$
where the minimum is over convex lattice paths in the sense of \cite{cristofaro2019symplectic}. We
use the orientation convention of Definition \ref{def_integral_path_generators}, which is the reverse
of the convention in \cite{cristofaro2019symplectic}.

We claim that the minimum can be restricted to $(n,m)$-adapted
semi-weakly convex integral paths satisfying $y(\Lambda)\geq0$.

First, every such adapted path is a convex lattice path in the sense
of \cite{cristofaro2019symplectic}. Moreover, a semi-weakly convex path with $y(\Lambda)\geq0$
cannot cross below the $x$-axis. Indeed, the quantities $q_i/p_i$ for
its nonvertical edges $(-p_i,q_i)$ are nonincreasing along the path.
Once one of the vertical components becomes negative, all subsequent
vertical components are negative. Thus crossing below the axis would
force the terminal value $y(\Lambda)$ to be negative. Consequently,
$
B^-_\Lambda=\emptyset,
$
and $B^+_\Lambda$ is precisely the region enclosed by $\Lambda$ and the
coordinate axes. It follows that
$$
c_k(X_\Omega)
\leq
\min\left\{
\ell_\Omega(\Lambda):
\mathcal L(B^+_\Lambda)=k+1
\right\},
$$
where the minimum on the right is over the adapted paths in the
statement of the theorem.

We prove the opposite inequality. Let $\Lambda$ be a convex lattice path
in the sense of~\cite{cristofaro2019symplectic}, written in the orientation of Definition \ref{def_integral_path_generators}.
Consider first the initial consecutive edges which are not horizontally
$n$-adapted. If $n=0$, there are no such edges. If $n>0$, write the sum
of these edges as
$
(-P_+,Q_+).
$

The failure of horizontal $n$-adaptedness implies
$
P_+<nQ_+.
$
Replace this entire initial part by the single edge
$
Q_+(-n,1),
$
and translate the initial point of the path to the right by
$nQ_+-P_+$. 

The endpoint of the replaced part is unchanged. We claim that the support
point of $\Omega$ for every discarded direction is $(a,0)$. Indeed, let
$(-p,q)$ be one of the discarded edge directions. Since this direction
is not horizontally $n$-adapted, we have
$
p<nq.
$
Since $n>0$ and $nf'(a)\geq -1$, it follows that
$
q+pf'(a)
\geq
q-\frac{p}{n}
>0.
$
The function
$
x\mapsto qx+pf(x)
$
is concave. Since its left derivative at $a$ is positive, the function is
increasing on $[0,a]$. Hence its maximum is attained at $x=a$. Thus the
support point is $(a,f(a))=(a,0)$. Moreover, since
$
1+nf'(a)\geq0,
$
the point $(a,0)$ is also a support point for the direction
$(-n,1)$. Consequently, replacing the discarded
edges by $Q_+(-n,1)$ preserves the $\Omega$-length. If the new edge is parallel to the following edge, we combine
the two parallel edges into a single edge.

Next, write the sum of the final consecutive edges which are not
vertically $m$-adapted as
$
(-P_-,Q_-).
$
The failure of vertical $m$-adaptedness implies
$
Q_-<mP_-.
$
Replace this final part by the single edge
$
P_-(-1,m).
$

The initial point of the replaced part is unchanged, while the terminal
point on the $y$-axis is moved upward by $mP_--Q_-$. We claim that the
support point of $\Omega$ for every discarded direction is $(0,f(0))$.
Indeed, let $(-p,q)$ be one of the discarded edge directions. Since this
direction is not vertically $m$-adapted, we have
$
q<mp.
$
Since $f'(0)\leq-m$, it follows that
$
q+pf'(0)
\leq
q-mp
<0.
$
The function
$
x\mapsto qx+pf(x)
$
is concave. Since its right derivative at $0$ is negative, it is
decreasing on $[0,a]$, and hence its maximum is attained at $x=0$.
Thus the support point is $(0,f(0))$. Moreover, since
$
m+f'(0)\leq0,
$
the point $(0,f(0))$ is also a support point for the direction
$(-1,m)$.Consequently, replacing the
discarded edges by $P_-(-1,m)$ preserves the $\Omega$-length. If the new edge is parallel to the preceding edge, we combine
the two parallel edges into a single edge.

Denote the resulting path by $\Lambda^{\mathrm a}$. It is
$(n,m)$-adapted, satisfies
$
y(\Lambda^{\mathrm a})\geq y(\Lambda)\geq0,
$
and obeys
$
\ell_\Omega(\Lambda^{\mathrm a})
=
\ell_\Omega(\Lambda).
$
The region $B^+_{\Lambda^{\mathrm a}}$ contains the region enclosed by
$\Lambda$ and the coordinate axes. Hence
$
\mathcal L(B^+_{\Lambda^{\mathrm a}})
\geq
\mathcal L(B_\Lambda).
$
If this inequality is strict, apply the corner-rounding operation used
in the proof of Corollary~A.12 in \cite{cristofaro2019symplectic}. Each corner rounding removes one
lattice point and does not increase the $\Omega$-length. It also
preserves $(n,m)$-adaptedness: every new edge is the sum of consecutive
old edges, and the set
$$
\mathcal C_{n,m}
=
\{v\in\mathbb R^2:
(-n,1)\mathbin{\times}v\geq0,
\quad
v\mathbin{\times}(-1,m)\geq0\}
$$
is a convex cone. The operation also preserves the condition
$y(\Lambda)\geq0$.

Repeating the operation, we obtain an $(n,m)$-adapted semi-weakly
convex integral path $\widetilde\Lambda$ such that
$
\mathcal L(B^+_{\widetilde\Lambda})
=
\mathcal L(B_\Lambda),
$
and
$
\ell_\Omega(\widetilde\Lambda)
\leq
\ell_\Omega(\Lambda).
$

Taking $\mathcal L(B_\Lambda)=k+1$ proves the reverse inequality and hence part
(i) of the theorem. When $f(a)>0$, we have $n=0$, so the horizontal
modification is absent; the same argument proves part~(ii).
\end{proof}

\subsection{Finiteness of Convex and Concave Generators with fixed Index. }
\label{sec:nonzero-ech-classes}

In this subsection, $X_\Omega$ is either a convex or a concave toric
domain.  We work throughout with coefficients in
$
\mathbb F_2=\mathbb Z/2.
$

For each fixed $K\in\mathbb Z_{\geq0}$, there are only finitely many
convex generators and only finitely many concave generators of ECH
index at most $K$.  Indeed, if $\Lambda$ is convex, then Pick's
formula gives
$
x(\Lambda)+y(\Lambda)\leq I(\Lambda)\leq K.
$
If $\Lambda$ is concave, the lattice points on the coordinate axes
which belong to $B_\Lambda$ give
\[
x(\Lambda)+y(\Lambda)-1
\leq
\mathcal L(B_\Lambda)
\leq
\frac{I(\Lambda)}2
\leq
\frac K2.
\]
Thus, in either case, all vertices of $\Lambda$ belong to a fixed
finite subset of $\mathbb Z^2$.  Since the paths are injective and
each edge has one of two labels, only finitely many generators can
occur.

Consequently, if $X_\Omega$ is convex, the number
\[
A_K^{\rm cvx}(\Omega)
=
\max
\{
\ell_\Omega(\Lambda):
\Lambda\text{ is convex and }0\leq I(\Lambda)\leq K
\}
\]
is finite, while if $X_\Omega$ is concave, the number
\[
A_K^{\rm ccv}(\Omega)
=
\max
\{
\ell_\Omega(\Lambda):
\Lambda\text{ is concave and }0\leq I(\Lambda)\leq K
\}
\]
is finite.

Fix $K\in\mathbb Z_{\geq0}$.  In the convex case choose
$
L>A_K^{\rm cvx}(\Omega),
$
and in the concave case choose
$
L>A_K^{\rm ccv}(\Omega),
$
with $L$ sufficiently large for the perturbations below.  We then
choose the approximation parameter $l$ sufficiently large and take an
$L$-nice perturbation
$
(Y_{l,L},\lambda_{l,L})
=
(\partial X_{\Omega_{l,L}},\lambda_{l,L})
$
as in Section~\ref{sec:convexpert}, with $m=n=0$, in the convex case,
and as in Section~\ref{sec:concavepert} in the concave case.  We also
fix a generic $\lambda_{l,L}$-compatible almost complex structure
$J$.

By the choice of $L$, every convex or concave generator of ECH index
at most $K$ has $\Omega$-length strictly less than $L$.  The action
estimates and the orbit-set--generator correspondences in
Sections~\ref{sec:convexpert} and~\ref{sec:concavepert} therefore
imply, after increasing $l$ and making the Morse--Bott perturbation
sufficiently small, that the orbit-set generators in this grading
range are in grading-preserving bijection with the corresponding
convex or concave generators.

\section{A Morse--Bott argument for $L$-nice perturbations}
\label{sec:morse-bott-differential}

In this section we show that the differential decreases the number of
hyperbolic labels by one at a convex toric boundary and increases it by
one at a concave toric boundary.

We use the following standing convention.  Fix an action cutoff $L>0$
and a grading bound $K\in\mathbb Z_{>0}$.  At a convex boundary, choose
an $L$-nice perturbation with respect to the grading bound $K$ as in
Section~\ref{sec:convexpert}.  At a concave boundary, first increase the
approximation parameter $l$ as required in
Section~\ref{sec:concavepert}, and then choose an $L$-nice perturbation.
All orbit sets and generators considered below have action less than $L$
and grading between $0$ and $K$.  Thus ``$L$-nice perturbation'' refers
to the perturbations already defined in Section~5; no additional
perturbation notion is being introduced.  This agrees with the usual
formulation of results for $L$-nice perturbations, for example
\cite[Proposition~3.7]{hutchingsroyweileryao2026}.

The arguments through
Lemma~\ref{lem:concave-fredholm-dichotomy} use only the $L$-nice
properties and a generic compatible almost complex structure.  The
exclusion of double rounding and double corounding requires one additional
choice: within the arbitrarily small neighborhood allowed by the
$L$-nice construction, the perturbation is taken sufficiently close to
the Morse--Bott contact form, and the compatible almost complex structure
is chosen in a generic family converging to the standard toric almost
complex structure on each relevant product region.

The global Morse--Bott input is Yao's description of the ECH differential
for convex and concave toric boundaries by tree-like index-one cascades;
see \cite[Theorem~1.4 and Sections~8--10]{YaoMorseBottECH} and the
cascade--curve correspondence of \cite{YaoCascades}.  The local geometric
input is the polygonal and zero-area analysis of Hutchings--Sullivan.  The
positive Morse--Bott model which occurs at a convex toric boundary is the
local model used in the proof of
\cite[Theorem~11.11]{HutchingsSullivan2006}, while the negative
Morse--Bott model which occurs at a concave toric boundary is the model of
\cite[Lemma~A.1]{HutchingsSullivan2005}.

We work over $\mathbb F_2$.  This is useful below because the two local
Morse trajectories, which would change the label of an unchanged
Morse--Bott edge from $e$ to $h$ cancel modulo two; compare Step~2 in the
proof of \cite[Theorem~11.11]{HutchingsSullivan2006}.

\subsection{Index-one currents and the polygonal order}
\label{subsec:index-one-localization}

We first record the standard reduction of an index-one current to its
unique nontrivial component.

\begin{lem}
\label{lem:index-one-localization}
Let $\lambda$ be one of the $L$-nice perturbations fixed above, and let
$J$ be generic and compatible with $\lambda$.  Suppose that $\alpha_+$
and $\alpha_-$ are admissible orbit sets in the associated action and
grading window and
$
\left\langle
\partial_J\alpha_+,\alpha_-
\right\rangle\neq0.
$
Then, a current contributing to this coefficient has the form
$
\mathcal C
=
\mathcal C_{\mathrm{triv}}\sqcup C,
$
where $\mathcal C_{\mathrm{triv}}$ is a union of trivial cylinders and
$C$ is embedded, irreducible, and satisfies
\begin{equation}
\label{eq:mb-ind-one-component}
I(C)=\operatorname{ind}(C)=1.
\end{equation}
After deleting the common factors represented by
$\mathcal C_{\mathrm{triv}}$, the polygonal difference of the two
remaining generators is connected.
\end{lem}

\begin{proof}
The first assertion follows from Proposition~\ref{BasedDiff}-(ii).  The
ECH-index-one current consists of an ECH-index-zero part, hence a union of
trivial cylinders, and one embedded nontrivial component of ECH index one.
For a generic almost complex structure the nontrivial component is
transversely cut out, and the ECH index inequality gives
$\operatorname{ind}(C)\leq I(C)=1$.  Since the differential
counts rigid curves modulo the $\mathbb R$-translation action,
\eqref{eq:mb-ind-one-component} follows.

For the last assertion, use the local-energy inequality on the toric slices.
For a regular value $x$ of the toric coordinate let
$F_x=\{x\}\times T^2$.  If
$[C\cap F_x]=(p,q)$, positivity of the $d\lambda$-energy gives the local
inequality used in Sections~9--10 of \cite{YaoMorseBottECH}.  It follows
by applying Stokes' theorem to the portion of the curve between two nearby
slices.  The resulting nesting and noncrossing argument implies that two
disjoint components of the polygonal difference would force at least two
nontrivial holomorphic components.  This contradicts the first part of the
lemma.  Compare the connectedness argument in the proof of
\cite[Theorem~11.11]{HutchingsSullivan2006}.
\end{proof}

The same local-energy argument also gives the ordering of the positive and
negative polygonal paths.  We will use it together with the following local
label-matching observation.

\begin{lem}
\label{lem:toric-nesting-label-matching}
The $L$-nice perturbations can be chosen so that the following holds in
the associated action and grading window.  If
$
\left\langle
\partial_J\alpha_{\Lambda_+},
\alpha_{\Lambda_-}
\right\rangle\neq0,
$
then the negative polygonal path is nested on the differential side of the
positive polygonal path.  Moreover, an edge which agrees on the two paths
has the same $e/h$ label on both sides.  The analogous statement holds for
partially agreeing edges in the sense needed to delete matching factors.
Consequently, after deleting all matching factors, either the reduced
polygonal paths are different, or the differential coefficient is zero.
\end{lem}

\begin{proof}
The nesting statement is the no-crossing consequence of the toric
local-energy inequality.  For concave and convex toric boundaries this is
proved in Sections~9 and~10 of \cite{YaoMorseBottECH}; in the convex case
Yao's argument includes the two nondegenerate exceptional orbits by using
the intersection inequalities with their trivial cylinders.

The label-matching statement is local near one Morse--Bott torus.  If the
underlying edge is unchanged and the two generators differ only by changing
an $e$ label to an $h$ label, the only local index-one pieces are the two
Morse--Bott cylinders corresponding to the two gradient trajectories on the
circle of Reeb orbits.  Their mod-two count is zero.  This is the local
argument in Step~2 of the proof of
\cite[Theorem~11.11]{HutchingsSullivan2006}.  It takes place in a
neighborhood of a single Morse--Bott torus and therefore applies unchanged
to the toric $S^3$.  The exceptional orbits do not create an additional
label-matching issue, since they are elliptic and carry no $h$ label.
\end{proof}

\subsection{Fredholm index and the possible change of $h$}
\label{subsec:fredholm-local-transitions}

We next show that, before excluding double rounding, the change of the
number of hyperbolic labels is either one or three.  The convex case
requires a little care because the two exceptional elliptic orbits can
occur in the $L$-nice grading window.

\begin{lem}
\label{lem:convex-fredholm-dichotomy}
Let $\Lambda_+$ and $\Lambda_-$ be convex generators in the $L$-nice
action and grading window, and suppose that
$
\left\langle
\partial_J\alpha_{\Lambda_+},
\alpha_{\Lambda_-}
\right\rangle\neq0.
$
Then
$
h(\Lambda_+)-h(\Lambda_-)=1$ or 
$h(\Lambda_+)-h(\Lambda_-)=3.
$
For the reduced nontrivial component $C$, with the trivializations from
Section~\ref{sec:convexpert},
\begin{equation}
\label{eq:positive-relative-c1}
c_\tau(C)
=x(\Lambda_+)-x(\Lambda_-)
 +y(\Lambda_+)-y(\Lambda_-)
\in\{0,1\}.
\end{equation}
If $c_\tau(C)=1$, only the first alternative can occur.
If the second alternative occurs, then the reduced nontrivial component
has neither punctures at nor intersections with the trivial cylinders over
the exceptional orbits.  In particular it is contained in an interior
product region $I\times T^2$, its relative Chern number is zero, and the
reduced polygonal transition is the double-rounding configuration.
\end{lem}

\begin{proof}
By Lemma~\ref{lem:toric-nesting-label-matching}, after deleting matching
factors the negative convex path is strictly nested inside the positive
one.  Hence
$
d:=
\mathcal L(B_{\Lambda_+})-
\mathcal L(B_{\Lambda_-})
$
is a positive integer.  Since the ECH differential decreases the ECH
index by one and
$
I(\Lambda)=2(\mathcal L(B_\Lambda)-1)-h(\Lambda),
$
we obtain
$
1
=
2d-h(\Lambda_+)+h(\Lambda_-).
$
Thus
\begin{equation}
\label{eq:convex-h-change-index}
h(\Lambda_+)-h(\Lambda_-)=2d-1,
\end{equation}
which proves that the difference is positive and odd.

The relative first Chern number is the difference of the absolute Chern
terms computed in Section~\ref{sec:convexpert}, and this proves the
equality in \eqref{eq:positive-relative-c1}.  The endpoint-intersection
inequalities identify the two coordinate differences with nonnegative
intersection quantities; in particular $c_\tau(C)\geq0$.

Let $e_+(C)$ denote the number of positive elliptic punctures, including
punctures at exceptional orbits, and let $h_+(C)$ and $h_-(C)$ denote
the numbers of positive and negative hyperbolic punctures.  Put
$H(C)=h_+(C)+h_-(C)$.  Every relevant elliptic orbit is $L$-positive,
so every iterate which can occur has Conley--Zehnder index $1$.  The
Fredholm formula and $\operatorname{ind}(C)=1$ therefore give
\begin{equation}
\label{eq:positive-fredholm-c1}
2g(C)+2e_+(C)+H(C)+2c_\tau(C)=3.
\end{equation}
After the matching hyperbolic factors have been deleted, the difference
in \eqref{eq:convex-h-change-index} is the difference between the numbers
of positive and negative hyperbolic punctures of $C$.  Hence
\[
0<
h(\Lambda_+)-h(\Lambda_-)
\leq
H(C)
\leq3.
\]
Together with oddness, this gives the two stated alternatives.
Equation~\eqref{eq:positive-fredholm-c1} also gives
$c_\tau(C)\leq1$, proving the last part of
\eqref{eq:positive-relative-c1}.  If $c_\tau(C)=1$, then the same
equation and the strict positivity of the change in $h$ force
$H(C)=1$; thus the three-label alternative is impossible.

Suppose now that the difference is three.  Every inequality above is then
an equality.  In particular,
\[
h_+(C)=3,
\qquad
h_-(C)=0,
\qquad
g(C)=0,
\qquad
e_+(C)=0,
\qquad
c_\tau(C)=0.
\]

To control exceptional punctures on the negative side as well, keep the
endpoint terms visible in the Fredholm calculation from the proof of
\cite[Theorem~10.3]{YaoMorseBottECH}.  Write $r_{\rm exc}(C)\geq0$ for
the exceptional-puncture correction, which vanishes exactly when $C$ has
no puncture at either exceptional orbit, and write
$\iota_1(C),\iota_2(C)\geq0$ for the intersection corrections with the
corresponding exceptional trivial cylinders.  Regrouping Yao's formula
gives
\begin{equation}
\label{eq:positive-fredholm-with-endpoints}
2g(C)+2e_+^{\rm int}(C)+H(C)
+2r_{\rm exc}(C)+2\iota_1(C)+2\iota_2(C)=3.
\end{equation}
Since $H(C)=3$, equation~\eqref{eq:positive-fredholm-with-endpoints}
forces
$r_{\rm exc}(C)=\iota_1(C)=\iota_2(C)=0$.  Thus $C$ has no exceptional
punctures and does not intersect the exceptional trivial cylinders.  It is
therefore supported between interior Morse--Bott tori in a product region
$I\times T^2$.

The local-energy and connectedness argument now gives precisely the
three-edge configuration called double rounding in
\cite[Theorem~11.11, Step~3]{HutchingsSullivan2006}.  This proves the
last assertion.
\end{proof}

\begin{lem}
\label{lem:concave-fredholm-dichotomy}
Increase the approximation parameter $l$ so that no orbit set in the
grading window of the $L$-nice correspondence contains an exceptional
orbit.  If $\Lambda_+$ and $\Lambda_-$ are concave generators in this
window and
$
\left\langle
\partial_J\alpha_{\Lambda_+},
\alpha_{\Lambda_-}
\right\rangle\neq0,
$
then
$
h(\Lambda_-)-h(\Lambda_+)=1
$ or $
h(\Lambda_-)-h(\Lambda_+)=3.
$
For the reduced nontrivial component $C$,
\begin{equation}
\label{eq:concave-relative-c1}
c_{\tau_{\rm MB}}(C)
=x(\Lambda_+)-x(\Lambda_-)
 +y(\Lambda_+)-y(\Lambda_-)
\in\{0,1\}.
\end{equation}
If $c_{\tau_{\rm MB}}(C)=1$, only the first alternative can occur.
In the second case, the reduced transition is a double corounding: its
positive exterior ends are elliptic, and its negative exterior ends contain
exactly three hyperbolic ends.  Moreover,
$c_{\tau_{\rm MB}}(C)=0$, and $C$ is disjoint from the exceptional
trivial cylinders and contained in an interior product region
$I\times T^2$.
\end{lem}

\begin{proof}
There are no exceptional ends in the $L$-nice grading window.  After
deleting the matching factors, the local-energy and no-crossing inequalities give the
concave nesting relation.  In particular,
$
d:=
\mathcal L(B_{\Lambda_+})-
\mathcal L(B_{\Lambda_-})
$
is a positive integer.  Since
$
I(\Lambda)=2\mathcal L(B_\Lambda)+h(\Lambda)
$
and the differential lowers $I$ by one, we obtain
$
1
=
2d+h(\Lambda_+)-h(\Lambda_-),
$
and therefore
\begin{equation}
\label{eq:concave-h-change-index}
h(\Lambda_-)-h(\Lambda_+)=2d-1.
\end{equation}
Thus the increase in the number of hyperbolic labels is a positive odd
integer.

Use the Morse--Bott trivializations along the interior orbits.  The toric
Chern-class calculation gives
\[
c_{\tau_{\rm MB}}(C)
=x(\Lambda_+)-x(\Lambda_-)
 +y(\Lambda_+)-y(\Lambda_-).
\]
Although there are no exceptional punctures, the curve can intersect the
trivial cylinders over the exceptional orbits.  The two endpoint
intersection formulas identify the last expression with
$\iota_1(C)+\iota_2(C)$, where
$\iota_1(C),\iota_2(C)\geq0$ by positivity of intersections.  Thus
$c_{\tau_{\rm MB}}(C)\geq0$; it need not vanish for an arbitrary
index-one differential curve.

For a concave Morse--Bott torus the elliptic orbit is $L$-negative, so
\[
\operatorname{CZ}_{\tau_{\rm MB}}(e^m)=-1,
\qquad
\operatorname{CZ}_{\tau_{\rm MB}}(h)=0.
\]
Let $e_\pm$ and $h_\pm$ denote the numbers of positive and negative
elliptic and hyperbolic ends, respectively, and put $h=h_++h_-$.  If $g$
is the genus, then
$
\chi(C)=2-2g-e_+-e_--h
$
and
$
\operatorname{CZ}^{\operatorname{ind}}_{\tau_{\rm MB}}(C)
=-e_++e_-.
$
The Fredholm formula therefore gives
\[
\begin{aligned}
1
&=
-\chi(C)
+2c_{\tau_{\rm MB}}(C)
+\operatorname{CZ}^{\operatorname{ind}}_{\tau_{\rm MB}}(C)\\
&=
2g-2+2e_-+h+2c_{\tau_{\rm MB}}(C),
\end{aligned}
\]
so
\begin{equation}
\label{eq:concave-fredholm-three}
2g+2e_-+h+2c_{\tau_{\rm MB}}(C)=3.
\end{equation}
The change in \eqref{eq:concave-h-change-index} is at most
$h_++h_-=h$.  Equation~\eqref{eq:concave-fredholm-three} therefore gives
$h\leq3$ and $c_{\tau_{\rm MB}}(C)\leq1$.  This proves
\eqref{eq:concave-relative-c1} and the two alternatives.  If
$c_{\tau_{\rm MB}}(C)=1$, then the same equation forces $h=1$, so only
the one-label alternative occurs.

If it is three, then equality in the preceding inequalities gives
$h_+=0$, $h_-=3$, $g=0$, $e_-=0$, and
$c_{\tau_{\rm MB}}(C)=0$.  Since
$c_{\tau_{\rm MB}}(C)=\iota_1(C)+\iota_2(C)$, both endpoint
intersection numbers vanish.  Positivity of intersections then places
$C$ in an interior product region.  The connected polygonal difference
is therefore the double-corounding configuration.
\end{proof}

\subsection{The zero-area constraint and Morse--Bott compactness}
\label{subsec:mb-degeneration-endpoint}

The remaining task is to exclude the alternatives in which the number of
hyperbolic labels changes by three.  We first fix the coordinates on the
circles of Reeb orbits which enter the zero-area argument.

Let $T_{-p,q}$ be an interior Morse--Bott torus.  In the torus coordinates
$(t_1,t_2)\in(\mathbb R/\mathbb Z)^2$, the Reeb orbits on
$T_{-p,q}$ have primitive homology class $(q,p)$.  We identify the circle
of such Reeb orbits with $\mathbb R/\mathbb Z$ by the affine coordinate
\begin{equation}
\label{eq:toric-affine-coordinate}
\vartheta_{-p,q}(\gamma)
=
q t_2-p t_1+\frac{pq}{2}
\qquad\text{in }\mathbb R/\mathbb Z,
\end{equation}
where $(t_1,t_2)$ is any point on $\gamma$.  The first two terms are
constant along $\gamma$, and the affine normalization is the toric analogue
of \cite[(57)]{HutchingsSullivan2005} and \cite[(99)]{HutchingsSullivan2006}.
Changing all of these coordinates by the simultaneous convention induced
by the torus coordinates does not change the endpoint relation below.

We choose the unperturbed almost complex structure $J_0$ on each product
region to agree with the standard toric almost complex structure used in
the local Morse--Bott model.  We now use the additional choice announced
at the beginning of the section.  Choose a family of $L$-nice contact
forms $\lambda_\delta$, $\delta>0$, in the fixed action and grading
window, converging in $C^\infty$ to the Morse--Bott form $\lambda_0$.
Choose a generic family of compatible almost complex structures
$J_\delta$ converging to $J_0$, with the prescribed standard form near
the relevant Morse--Bott tori.  The $L$-nice orbit-set correspondence is
unchanged after making $\delta$ smaller.

\begin{lem}
\label{lem:mb-endpoint-constraint}
Let $C$ be a nontrivial $J_0$-holomorphic component in
$\mathbb R\times I\times T^2$.  Suppose that its positive ends have
multiplicities $m_i$ and asymptotic orbits $\gamma_i^+$, and that its
negative ends have multiplicities $n_j$ and asymptotic orbits
$\gamma_j^-$.  Then
\begin{equation}
\label{eq:mb-endpoint-constraint}
\sum_i m_i\vartheta(\gamma_i^+)
-
\sum_j n_j\vartheta(\gamma_j^-)
=0
\qquad\text{in }\mathbb R/\mathbb Z.
\end{equation}
Here the coordinate on each Morse--Bott circle is the affine coordinate
\eqref{eq:toric-affine-coordinate} corresponding to its slope.
\end{lem}

\begin{proof}
This is the toric zero-area calculation.  In the standard product
coordinates $(s,x,t_1,t_2)$, $J_0$ can be normalized so that the closed
$2$-form
$
dt_1\wedge dt_2-ds\wedge dx
$
vanishes on every $J_0$-complex tangent plane.  Consequently
$
\int_Cdt_1\wedge dt_2
=
\int_Cds\wedge dx.
$
Since
\[
ds\wedge dx=d(-x\,ds),
\]
Stokes' theorem on the truncation obtained by intersecting $C$ with
$[-R,R]\times I\times T^2$ gives
$
\int_{C_R}ds\wedge dx
=
\int_{\partial C_R}-x\,ds.
$
The boundary circles of the truncation lie in the hypersurfaces
$s=\pm R$, and hence $ds$ vanishes on them.  Thus the boundary integral
vanishes, and therefore
\begin{equation}
\label{eq:toric-zero-area}
\int_Cdt_1\wedge dt_2=0.
\end{equation}
Modulo $\mathbb Z$, the left hand side is the signed area in the torus
between the asymptotic orbit cycles.  With the affine normalization
\eqref{eq:toric-affine-coordinate}, this area is exactly
$
\sum_i m_i\vartheta(\gamma_i^+)
-
\sum_j n_j\vartheta(\gamma_j^-).
$
This proves \eqref{eq:mb-endpoint-constraint}.  Compare
\cite[Lemma~A.2]{HutchingsSullivan2005} and
\cite[Proposition~10.16]{HutchingsSullivan2006}; the computation above is
the same computation in the local toric coordinates.
\end{proof}

We now choose the Morse functions on the finitely many relevant
Morse--Bott circles.  Fix an irrational number
$
\zeta\in\left(\frac23,1\right).
$
At a negative Morse--Bott torus, which is the concave case, choose the
Morse function so that the elliptic critical point is at
$\vartheta=0$ and is the minimum, while the hyperbolic critical point is
at $\vartheta=\zeta$ and is the maximum, as in
\cite[Lemma~A.1(a)]{HutchingsSullivan2005}.  At a positive Morse--Bott
torus, which is the convex case, use the corresponding positive
Morse--Bott choice in Step~5 of the proof of
\cite[Theorem~11.11]{HutchingsSullivan2006}.  Equivalently, one uses the
same two marked points but reverses the maximum/minimum assignment, as
required by the fact that the elliptic orbit has Conley--Zehnder index
$1$ and the hyperbolic orbit has Conley--Zehnder index $0$.

\begin{lem}
\label{lem:mb-limit-double-transition}
Suppose that convex double-rounding curves or concave double-corounding
curves exist for arbitrarily small $L$-nice perturbations in the family
above.  Then there are
$\delta_\nu\to0$ and corresponding genus-zero index-one curves $C_\nu$
which converge, after passing to a subsequence, to a reduced tree-like
index-one Morse--Bott cascade with the same exterior asymptotic data.
Moreover, after trivial cylinders and branched covers of trivial cylinders
are discarded, such a double-transition cascade has at most two
nontrivial holomorphic components.
\end{lem}

\begin{proof}
The action bound gives a uniform energy bound
$
E(C_\nu)
=
\mathcal A_{\delta_\nu}(\alpha_+)
-
\mathcal A_{\delta_\nu}(\alpha_-)
<L.
$
There is no nonconstant closed holomorphic component in the
symplectization because its symplectic form is exact.  Morse--Bott
SFT/Gromov compactness therefore gives a broken $J_0$-holomorphic
building.  In the Morse--Bott description, adjacent levels are connected
by gradient trajectories on the circles of Reeb orbits.  For the present
toric boundaries Yao's tree-like compactification identifies the limit
with a reduced tree-like index-one cascade; see
\cite[Theorem~1.4 and Sections~8--10]{YaoMorseBottECH} and
\cite{YaoCascades}.

For the double transition, the local no-crossing argument gives the
stronger component bound.  This is the argument used in the proof of
\cite[Lemma~A.1(a)]{HutchingsSullivan2005}: the three-hyperbolic
polygonal configuration has a connected difference region, and a
nontrivial component cannot cross the polygonal path determined by the
adjacent level.  It follows that the entire limiting building contains a
total of at most two components which are not trivial cylinders or
branched covers thereof.  This argument depends only on the local
polygonal order and hence applies in the product region obtained above.
\end{proof}

\subsection{Excluding double rounding and double corounding}
\label{subsec:ruling-out-double-rounding}

\begin{lem}
\label{lem:no-convex-double-rounding}
For all sufficiently small regular values of $\delta$ in the $L$-nice
family chosen above, a convex double-rounding index-one curve does not
occur in the fixed action and grading window.
\end{lem}

\begin{proof}
By Lemma~\ref{lem:convex-fredholm-dichotomy}, a convex transition which
loses three hyperbolic labels has no exceptional punctures or
intersections.  Thus the entire
reduced double transition is contained in an interior product region
$I\times T^2$ and is exactly the positive Morse--Bott double-rounding
configuration.

Suppose that such curves existed for a sequence $\delta_\nu\to0$.
Lemma~\ref{lem:mb-limit-double-transition} would give a limiting
Morse--Bott cascade with at most two nontrivial components.  If there were
only one nontrivial component, all three exterior hyperbolic critical
points would enter the endpoint relation on the same side, while the
exterior elliptic critical points contribute the opposite marked value.
The endpoint constraint would force a nonzero multiple $3\zeta$ to vanish
in $\mathbb R/\mathbb Z$, contradicting the irrationality of $\zeta$.
Thus there are exactly two nontrivial components.

The remaining two-component configuration is precisely the local
positive Morse--Bott configuration treated in Step~5 of the proof of
\cite[Theorem~11.11]{HutchingsSullivan2006}.  There
\cite[Proposition~10.16]{HutchingsSullivan2006}, together with the
special choice of the two Morse critical points, is used exactly as in
\cite[Lemma~A.1(a)]{HutchingsSullivan2005} to show that the two
intermediate points on the common Morse--Bott circle cannot be joined by
the gradient trajectory required by the limiting cascade.  Hence the
limiting cascade cannot exist.

Compactness now gives the desired statement for all sufficiently small
regular values in the chosen $L$-nice family.
\end{proof}

\begin{lem}
\label{lem:no-concave-double-corounding}
For all sufficiently small regular values of $\delta$ in the $L$-nice
family chosen above, a concave double-corounding index-one curve does not
occur in the fixed action and grading window.
\end{lem}

\begin{proof}
By Lemma~\ref{lem:concave-fredholm-dichotomy}, the reduced double
corounding has only elliptic exterior ends at the positive side and exactly
three hyperbolic exterior ends at the negative side.  This is the negative
Morse--Bott sign convention of \cite[Lemma~A.1(a)]{HutchingsSullivan2005}.  Indeed, in Yao's
notation the concave tori are negative Morse--Bott tori: the elliptic orbit
has Conley--Zehnder index $-1$ and the hyperbolic orbit has
Conley--Zehnder index $0$; see
\cite[Definition~5.17]{YaoMorseBottECH}.

Assume that double-corounding curves exist for $\delta_\nu\to0$.
By Lemma~\ref{lem:mb-limit-double-transition}, the limit has at most two
nontrivial holomorphic components.  A one-component limit is impossible,
because Lemma~\ref{lem:mb-endpoint-constraint} gives
$
3\zeta=0
$ in $\R/\mathbb Z
$
up to an overall sign, contradicting the irrationality of $\zeta$.
Hence there are exactly two nontrivial components.

Let $x$ and $y$ denote their two intermediate asymptotic orbits on the
common Morse--Bott circle.  Applying
Lemma~\ref{lem:mb-endpoint-constraint} to the component which sees one
exterior hyperbolic end and to the component which sees the other two gives
\[
\vartheta(x)=-\zeta=1-\zeta
\qquad\text{and}\qquad
\vartheta(y)=2\zeta=2\zeta-1
\quad\text{in }\mathbb R/\mathbb Z.
\]
Therefore
\begin{equation}
\label{eq:mb-intermediate-ranges}
0<\vartheta(x)<\frac13,
\qquad
\frac13<\vartheta(y)<\zeta.
\end{equation}
This is exactly equation~(58) in the proof of
\cite[Lemma~A.1(a)]{HutchingsSullivan2005}.

Hutchings--Sullivan index the levels in the opposite direction from the
cascade convention used in \cite{YaoMorseBottECH}.  Their required
downward gradient trajectory from $x$ to $y$ is therefore the same
unparametrized connecting trajectory as Yao's upward gradient trajectory
with the endpoints read in the opposite level order.  Since the Morse
function has its minimum at $0$ and its maximum at $\zeta$, the two points
in \eqref{eq:mb-intermediate-ranges} cannot be connected by the required
trajectory.  This is the final contradiction in
\cite[Lemma~A.1(a)]{HutchingsSullivan2005}.

Thus no limiting double-corounding cascade exists, and compactness excludes
double-corounding curves for all sufficiently small regular values in the
chosen $L$-nice family.
\end{proof}

We can now state the consequence of the $L$-nice construction which will
be used below.

\begin{prop}
\label{prop:hyperbolic-label-differential}
Fix an action cutoff $L>0$ and a grading bound $K\in\mathbb Z_{>0}$ for
which the $L$-nice correspondences of
Sections~\ref{sec:convexpert} and~\ref{sec:concavepert} are defined.  The
$L$-nice perturbations can be chosen arbitrarily $C^\infty$-close to the
corresponding Morse--Bott forms, together with generic compatible almost
complex structures as above, so that the following statements hold for
orbit sets of action less than $L$ and grading between $0$ and $K$:
\begin{enumerate}
\item[(i)] If $\Lambda_+$ and $\Lambda_-$ are convex generators
corresponding to orbit sets in this $L$-nice window and
$
\left\langle
\partial_J\alpha_{\Lambda_+},
\alpha_{\Lambda_-}
\right\rangle\neq0,
$
then
$
h(\Lambda_-)=h(\Lambda_+)-1.
$

\item[(ii)]  If $\Lambda_+$ and $\Lambda_-$ are concave generators
corresponding to orbit sets in the window and
$
\left\langle
\partial_J\alpha_{\Lambda_+},
\alpha_{\Lambda_-}
\right\rangle\neq0,
$
then
$
h(\Lambda_-)=h(\Lambda_+)+1.
$
\end{enumerate}
\end{prop}

\begin{proof}
In the convex case, Lemma~\ref{lem:convex-fredholm-dichotomy} shows that
the number of hyperbolic labels decreases by either one or three.  The
three-label alternative has no exceptional punctures or intersections and
is a local double rounding, which is excluded by
Lemma~\ref{lem:no-convex-double-rounding}.  Hence the decrease is exactly
one.

In the concave case,
Lemma~\ref{lem:concave-fredholm-dichotomy} shows that the number of
hyperbolic labels increases by either one or three.  The three-label
alternative is the double corounding excluded by
Lemma~\ref{lem:no-concave-double-corounding}.  Hence the increase is
exactly one.
\end{proof}

\begin{rem}
\label{rem:mb-h-filtration}
Since the ECH differential decreases the ECH index by one,
Proposition~\ref{prop:hyperbolic-label-differential} implies that the
convex differential preserves
$
I-h,
$
while the concave differential preserves
$
I+h
$
in the corresponding $L$-nice action and grading windows.
\end{rem}

\section{Proofs of Theorems \ref{thm: criterion_convex_concave}, \ref{thm:criterion_convex_convex} and \ref{thm:criterion_concave_concave}}
\label{sec:proofs}

The proofs of Theorems \ref{thm: criterion_convex_concave},
\ref{thm:criterion_convex_convex} and
\ref{thm:criterion_concave_concave} are adaptations of the proof of
Theorem~1.20 in \cite{hutchings2016beyond} to the corresponding settings.
The order is intentional. Theorems~3.1 and~3.2 are proved without using the
unpublished \cite{UserGuide}.  In Theorem~3.2, the minimality hypothesis on
the prescribed convex generator is used only to invoke
\cite[Lemma~5.5]{hutchings2016beyond}, which produces the nonzero filtered
ECH class from which the cobordism argument starts.  The new Theorem~3.3,
namely Theorem~\ref{thm:criterion_concave_concave}, depends on one finite
filtered statement about the concave elliptic sector from \cite{UserGuide}.
We record that statement explicitly immediately before the proof and use it
as a black box.  In particular, we do not use a dual Hutchings--Sullivan
corner-rounding description, and we do not transfer a chain-level calculation
from $T^3$ to $S^3$.

\begin{proof}[Proof of Theorem~\ref{thm: criterion_convex_concave}]
Let
$\varphi:X_\Omega\longrightarrow X_{\Omega'}$
be a symplectic embedding from the convex toric domain $X_\Omega$ into
the concave toric domain $X_{\Omega'}$, and fix
$k\in\mathbb Z_{>0}$. We divide the proof into six steps.

\medskip
\noindent
\textit{Step 1: Geometric setup and the filtered cobordism map.}
By shrinking $X_\Omega$ and enlarging $X_{\Omega'}$ by arbitrarily
small radial factors, and then using a limiting argument at the end of
the proof, we may assume that
$\varphi(X_\Omega)\subset\operatorname{int}(X_{\Omega'})$
and that the relevant boundaries are smooth. For the argument at a
fixed approximation parameter, we relabel these slightly modified
domains again by $X_\Omega$ and $X_{\Omega'}$.

Fix an integer $K\geq2k+1$ and a number $\varepsilon>0$. By the
finiteness result of Section~\ref{sec:nonzero-ech-classes}, choose $L>0$
strictly larger, with a fixed positive margin, than the combinatorial
length of every convex and every concave generator whose ECH index lies
between $0$ and $K$. Enlarge $L$, if necessary, so that
\cite[Proposition~3.2]{hutchingsroyweileryao2026} applies to the convex
generator $e_{-1,k-1}$, and so that a filtered cycle representing the
unique nonzero total ECH class in degree $2k$ exists at the concave end
and \cite[Proposition~3.7]{hutchingsroyweileryao2026} applies to the
one-edge generator $e_{-1,k}$.

Choose the approximation parameter $l$ sufficiently large and the
Morse--Bott perturbations sufficiently small so that the $L$-nice
orbit-set--generator correspondences of
Sections~\ref{sec:concavepert} and~\ref{sec:convexpert} hold in every
grading at most $K$, all the corresponding orbit sets have action
strictly less than $L$, and all action errors for the generators and
factors occurring below are smaller than $\varepsilon$. Thus, in
particular, in the three gradings $2k-1,2k,2k+1$ the filtered chain
groups contain the orbit sets corresponding to every combinatorial
generator in that grading.

We make one additional choice for the differential in this $L$-nice
action and grading window. On the finitely many Morse--Bott tori which occur below $L$, we
choose the Morse functions as in
Hutchings--Sullivan~\cite[Lemma~A.1(a)]{HutchingsSullivan2005}. More
precisely, after fixing the angular coordinate on each circle of
Morse--Bott orbits, the elliptic critical point is placed at $0$ and the
hyperbolic critical point at the same fixed irrational
$\zeta\in(2/3,1)$. These choices can be made simultaneously because only
finitely many tori occur in the $L$-nice action window. We
then choose generic compatible almost complex structures on both ends
for which the Morse--Bott cascade description of
\cite[Theorem~1.4 and Sections~9--10]{YaoMorseBottECH}, together with
the cascade--curve correspondence of \cite{YaoCascades}, applies. At a
concave end we use the same Morse functions after reversing the
transverse coordinate in the toric Morse--Bott model. The role of the
Hutchings--Sullivan choice will be explained in Step~3.

Let
$(Y_+,\lambda_+)=(Y'_{l,L},\lambda'_{l,L}), (Y_-,\lambda_-)=(Y_{l,L},\lambda_{l,L})$
be the chosen nondegenerate perturbations of
$\partial X_{\Omega'}$ and $\partial X_\Omega$, respectively. At the
concave positive end every relevant elliptic orbit is chosen
$L$-negative, while at the convex negative end every relevant elliptic
orbit is $L$-positive. All relevant hyperbolic orbits are positive
hyperbolic.

For sufficiently small perturbations, the complement of the embedded
source gives a weakly exact symplectic cobordism
$(X_{l,L},\omega_{l,L}):(Y_+,\lambda_+)\to(Y_-,\lambda_-),$
whose underlying manifold is diffeomorphic to $[0,1]\times S^3$.
Let $\overline X_{l,L}$ denote its completion. Choose the generic
compatible almost complex structures $J_+$ and $J_-$ on the positive
and negative symplectization ends as above, and choose a generic
cobordism-admissible almost complex structure $J$ on
$\overline X_{l,L}$ restricting to $J_\pm$.

Theorem~\ref{CoborMap} provides a chain map
$\phi^L:ECC^L(Y_+,\lambda_+,0,J_+)\to ECC^L(Y_-,\lambda_-,0,J_-)$
which induces the filtered ECH cobordism map. Since the cobordism is
diffeomorphic to a product, its map on total ECH is an isomorphism.
Moreover,
$H_1(Y_\pm)=H_2(Y_\pm)=H_2(X_{l,L})=0.$
Under these homological assumptions, whenever a coefficient of the
chain map is nonzero, Theorem~\ref{CoborMap}-(iii) supplies a broken
current of ECH index zero, and the relative ECH index is the difference
of the absolute gradings of its asymptotic orbit sets; see
\cite[Proposition~2.25]{hutchingsroyweileryao2026}. We use this
index-zero statement directly and do not impose any grading property
on the chosen chain-level representative $\phi^L$.

We also fix a primitive adapted to the boundary contact forms. Since
$H^2(X_{l,L};\mathbb R)=0$, choose a one-form $\eta$ with
$d\eta=\omega_{l,L}.$
On each boundary component,
$\eta|_{Y_\pm}-\lambda_\pm$ is closed. Since
$H^1(S^3;\mathbb R)=0$, there are functions $f_\pm$ such that
$\eta|_{Y_\pm}-\lambda_\pm=df_\pm.$
After extending $f_\pm$ to the cobordism and changing $\eta$ by an
exact form, we may assume
\begin{equation}
\label{eq:primitive-matches-convex-concave-strong}
\eta|_{Y_+}=\lambda_+,
\qquad
\eta|_{Y_-}=\lambda_-.
\end{equation}

\medskip
\noindent
\textit{Step 2: The cobordism is $L$-tame.}
Let $d\in\mathbb Z_{>0}$, let $\alpha_\pm$ be orbit sets satisfying
$\mathcal A(\alpha_\pm)<\frac{L}{d},$
and let
$C\in\mathcal M^J(\alpha_+,\alpha_-)$ be an embedded irreducible
$J$-holomorphic curve such that $I(dC)\leq0$. We must prove
$2g(C)-2+\operatorname{ind}(C)+h(C)+2e_L(C)\geq0.$

Suppose this inequality fails. Then
$2g(C)+\operatorname{ind}(C)+h(C)+2e_L(C)\leq1.$
Genericity gives $\operatorname{ind}(C)\geq0$. Since every relevant
hyperbolic orbit is positive hyperbolic, the parity of
$\operatorname{ind}(C)$ agrees with the parity of $h(C)$. Hence
$g(C)=\operatorname{ind}(C)=h(C)=e_L(C)=0.$
Every elliptic orbit at the positive end is $L$-negative, and every
elliptic orbit at the negative end is $L$-positive. By the definition
of $e_L$, the equality $e_L(C)=0$ therefore excludes every elliptic
orbit from both asymptotic orbit sets. The equality $h(C)=0$ excludes
every hyperbolic orbit. Thus
$\alpha_+=\alpha_-=\emptyset.$
The curve $C$ would then be a nonconstant closed $J$-holomorphic curve
in the exact symplectic manifold $(X_{l,L},\omega_{l,L})$, which is
impossible because its symplectic area would be both positive and zero.
Thus $(X_{l,L},\omega_{l,L},J)$ is $L$-tame.

\medskip
\noindent
\textit{Step 3: Elliptic support and a nonzero cobordism coefficient.}
By Proposition~\ref{prop:hyperbolic-label-differential}, the
$L$-nice perturbations and almost complex structures can be chosen so
that, throughout the $L$-nice action and grading window fixed in Step~1,
\begin{equation}
\label{eq:single-transition-h-T31}
 h(\widetilde\Gamma)=h(\Gamma)+1
 \quad\text{at a concave end},
 \qquad
 h(\widetilde\Gamma)=h(\Gamma)-1
 \quad\text{at a convex end}
\end{equation}
whenever
$\langle\partial_J\alpha_\Gamma,
\alpha_{\widetilde\Gamma}\rangle\ne0$.

We will also use the following elementary consequence of the same
$L$-nice Morse--Bott description at the convex negative end:
\begin{equation}
\label{eq:filtered-ECH-one-dimensional-T31}
ECH^L_{2k}(Y_-,\lambda_-,0)\simeq\mathbb F_2.
\end{equation}
Indeed, by the choice of $L$, every convex combinatorial generator in
gradings $2k-1$, $2k$, and $2k+1$ occurs below the cutoff. Because $L$
was chosen with a positive margin above their combinatorial lengths,
and because the actions of the perturbed asymptotic orbit sets converge
to the Morse--Bott actions, the cascade--curve correspondence identifies
this portion of the Morse--Bott complex with the corresponding portion
of $ECC^L(Y_-,\lambda_-,0,J_-)$. Since the ECH differential lowers the
grading by one, the homology in grading $2k$ is unchanged if one passes
from these three complete grading pieces to the full Morse--Bott
complex. By \cite[Theorem~1.4]{YaoMorseBottECH} and
\cite{YaoCascades}, the latter computes the ECH of the convex toric
boundary, which in grading $2k$ is $\mathbb F_2$.

Let
$z_+\in ECC^L_{2k}(Y_+,\lambda_+,0,J_+)$
be a cycle whose image in total ECH is the unique nonzero class of
degree $2k$. For a concave generator $\Gamma$, set
$q_+(\Gamma):=I(\Gamma)+h(\Gamma).$
Since the ECH differential decreases $I$ by one, the concave identity
in \eqref{eq:single-transition-h-T31} implies that the differential
preserves $q_+$ in the filtered range under consideration. Since every
generator occurring in $z_+$ has ECH index $2k$, its $q_+=2k$
component is exactly its elliptic part. Denote this component by $x_+$.
Then $x_+$ is a cycle.

The one-edge elliptic concave generator $e_{-1,k}$ has ECH index $2k$.
By \cite[Proposition~3.7]{hutchingsroyweileryao2026}, it occurs as a
summand of every filtered cycle representing the nonzero total class in
this grading. Hence $e_{-1,k}$ occurs in $z_+$ and therefore
$x_+\neq0$.

We claim that the image of $[x_+]$ in total ECH is nonzero. Otherwise
$z_++x_+$ would be a cycle representing the same nonzero total class as
$z_+$ but containing no elliptic summand, and in particular no
$e_{-1,k}$, contradicting again
\cite[Proposition~3.7]{hutchingsroyweileryao2026}. Thus $x_+$
represents the nonzero total class in degree $2k$.

Set
$y_-:=\phi^L(x_+).$
Compatibility of the filtered and total cobordism maps, together with
the fact that the total cobordism map is an isomorphism, implies that
$y_-$ is a cycle whose image in total ECH is nonzero. In particular,
$[y_-]\ne0$ in filtered ECH.

We first verify the grading of every generator in the support of
$y_-$ without assuming that the chosen chain map $\phi^L$ preserves
the grading. Let $\alpha_\Gamma$ occur in $y_-$ with nonzero
coefficient. Since $y_-=\phi^L(x_+)$ and the coefficient of
$\alpha_\Gamma$ in $y_-$ is one, expanding $x_+$ over $\mathbb F_2$
gives
\[
1=
\left\langle y_-,\alpha_\Gamma\right\rangle
=
\sum_{\Gamma'\in\operatorname{supp}(x_+)}
\left\langle
\phi^L(\alpha_{\Gamma'}),\alpha_\Gamma
\right\rangle.
\]
Hence at least one elliptic concave generator $\Gamma'$ in the support
of $x_+$ satisfies
$\left\langle \phi^L(\alpha_{\Gamma'}),\alpha_\Gamma \right\rangle=1.$
Theorem~\ref{CoborMap}-(iii) gives an index-zero broken current from
$\alpha_{\Gamma'}$ to $\alpha_\Gamma$. Since
$H_2(X_{l,L})=0$, the relative ECH index is the difference of the
absolute gradings. Hence
$I(\Gamma)=I(\Gamma')=2k.$
Thus $y_-\in ECC^L_{2k}$.

We next prove that the support of $y_-$ contains an elliptic convex
generator. Set $E_k:=e_{-1,k-1}$. This is an elliptic convex generator
of ECH index $2k$: indeed, $A(E_k)=(k-1)/2$,
$x(E_k)=1$, $y(E_k)=k-1$, and $e(E_k)=1$, so the convex Pick formula
gives $I(E_k)=2k$. In the notation of
\cite{hutchingsroyweileryao2026}, this is the simple orbit $e_{k-1,1}$.
By \cite[Proposition~3.2]{hutchingsroyweileryao2026}, our cutoff can be
chosen so that $\alpha_{E_k}\in ECC^L_{2k}(Y_-,\lambda_-,0,J_-)$ is a
cycle whose image in total ECH is nonzero. Hence
$[\alpha_{E_k}]\ne0$ in filtered ECH. By
\eqref{eq:filtered-ECH-one-dimensional-T31},
$[y_-]=[\alpha_{E_k}]$ in $ECH^L_{2k}$. Therefore there exists
$w\in ECC^L_{2k+1}(Y_-,\lambda_-,0,J_-)$ such that
\[
y_-+\alpha_{E_k}=\partial_{J_-}w.
\]

Suppose, for contradiction, that $y_-$ has no elliptic summand. Every
convex generator of odd degree has an odd number of $h$-labels because
$I(\Gamma)=2(\mathcal L(B_\Gamma)-1)-h(\Gamma)$. Write
$w=w_1+w_{\geq3}$, where $w_1$ is the sum of the terms with exactly one
$h$-label and $w_{\geq3}$ is the sum of the remaining terms. By the
convex identity in \eqref{eq:single-transition-h-T31}, the differential
decreases the number of $h$-labels by exactly one. Hence
$\partial_{J_-}w_1$ is elliptic, while every term of
$\partial_{J_-}w_{\geq3}$ still has at least one $h$-label. Taking the
elliptic part of the preceding identity gives
$\partial_{J_-}w_1=\alpha_{E_k}$, contradicting the fact that
$\alpha_{E_k}$ has nonzero image in total ECH. Thus some elliptic convex
generator $\Lambda$ occurs in $y_-$, and therefore
$\langle y_-,\alpha_\Lambda\rangle=
\langle\phi^L(x_+),\alpha_\Lambda\rangle=1$.

Write
$x_+=\sum_{\Lambda'\in\mathcal S}\alpha_{\Lambda'},$
where $\mathcal S$ is a finite set of elliptic concave generators of
ECH index $2k$. Then
\[
1=
\sum_{\Lambda'\in\mathcal S}
\left\langle
\phi^L(\alpha_{\Lambda'}),\alpha_\Lambda
\right\rangle
\quad\text{in }\mathbb F_2.
\]
Consequently, there exists an elliptic concave generator
$\Lambda'\in\mathcal S$ such that
$\left\langle \phi^L(\alpha_{\Lambda'}),\alpha_\Lambda \right\rangle=1.$

By Theorem~\ref{CoborMap}-(iii), this nonzero coefficient gives a broken
$J$-holomorphic current
$B=(\mathcal C_{N_-},\ldots,\mathcal C_0,\ldots,\mathcal C_{N_+})$
from $\alpha_{\Lambda'}$ to $\alpha_\Lambda$ satisfying $I(B)=0$.
Every nontrivial symplectization level has strictly positive ECH index
by Proposition~\ref{BasedDiff}-(i), while the central cobordism level
has nonnegative ECH index by Step~2 and
Proposition~\ref{L-tameJcurves}-(i). Since the total index is zero,
there are no nontrivial symplectization levels. After deleting possible
trivial-cylinder levels, the building reduces to a single unbroken
current
$\mathcal C\in\mathcal M^J(\alpha_{\Lambda'},\alpha_\Lambda)$. This current satisfies $I(\mathcal C)=0$.

\medskip
\noindent
\textit{Step 4: Factorizations and index identities.}
Write
$\mathcal C=\sum_{i=1}^r d_iC_i,$
where the $C_i$ are distinct irreducible somewhere injective
$J$-holomorphic curves and $d_i\in\mathbb Z_{>0}$. Let
$\alpha_{\Lambda_i'}$ and $\alpha_{\Lambda_i}$ denote the positive and
negative orbit sets of $C_i$. Since both total orbit sets
$\alpha_{\Lambda'}$ and $\alpha_\Lambda$ are elliptic, every
$\Lambda_i'$ is an elliptic concave generator and every $\Lambda_i$ is
an elliptic convex generator. We obtain factorizations
$\Lambda'=(\Lambda_1')^{d_1}\cdots(\Lambda_r')^{d_r}, \qquad \Lambda=\Lambda_1^{d_1}\cdots\Lambda_r^{d_r}.$

Proposition~\ref{L-tameJcurves}-(ii) gives $I(C_i)=0$ for every $i$.
Since $H_2(X_{l,L})=0$, the relative homology class is unique and the
relative ECH index is the difference of the absolute gradings of the
asymptotic orbit sets. Hence
$I(\Lambda_i')=I(\Lambda_i)$
for every $i$. For the complete current,
$I(\Lambda')=I(\Lambda)=2k.$
This proves item~(i).

The second assertion of Proposition~\ref{L-tameJcurves}-(ii) says that
two distinct components cannot have positive ends at covers of the same
$L$-negative elliptic orbit and cannot have negative ends at covers of
the same $L$-positive elliptic orbit. Under the
orbit-set--generator correspondence these are precisely the
disjointness conditions for the factorizations of $\Lambda'$ and
$\Lambda$. Thus both factorizations are disjoint.

Now let $S\subset\{1,\ldots,r\}$ and choose integers
$0\leq d_i'\leq d_i$ for $i\in S$. The third assertion of
Proposition~\ref{L-tameJcurves}-(ii) gives
$I\left(\sum_{i\in S}d_i'C_i\right)=0.$
The positive and negative orbit sets of this subcurrent correspond to
$\prod_{i\in S}(\Lambda_i')^{d_i'} \qquad\text{and}\qquad \prod_{i\in S}(\Lambda_i)^{d_i'},$
respectively. Consequently
$I\left(\prod_{i\in S}(\Lambda_i')^{d_i'}\right) = I\left(\prod_{i\in S}(\Lambda_i)^{d_i'}\right),$
which proves item~(iii).

\medskip
\noindent
\textit{Step 5: Action inequalities.}
For each $i$, truncate the cylindrical ends of $C_i$ and apply Stokes'
theorem using the primitive from
\eqref{eq:primitive-matches-convex-concave-strong}. We obtain
\[
0\leq\int_{C_i}\omega_{l,L}
=
\mathcal A_+(\alpha_{\Lambda_i'})-
\mathcal A_-(\alpha_{\Lambda_i}),
\]
and therefore
$\mathcal A_-(\alpha_{\Lambda_i}) \leq \mathcal A_+(\alpha_{\Lambda_i'}).$
The action estimates for the $L$-nice perturbations imply
\[
\left|\mathcal A_+(\alpha_{\Lambda_i'})-
\ell_{\Omega'}(\Lambda_i')\right|<\varepsilon,
\qquad
\left|\mathcal A_-(\alpha_{\Lambda_i})-
\ell_\Omega(\Lambda_i)\right|<\varepsilon.
\]
Thus
\begin{equation}
\label{eq:convex-concave-action-error-strong}
\ell_\Omega(\Lambda_i)
\leq
\ell_{\Omega'}(\Lambda_i')+2\varepsilon
\end{equation}
for every $i$.

\medskip
\noindent
\textit{Step 6: Topological complexity and removal of the action error.}
Fix $i$. Since $C_i$ is irreducible and somewhere injective,
Proposition~\ref{pro:topcomp} gives
\[
2g(C_i)-2
+\sum_a(2n_a^+-1)
+\sum_b(2n_b^--1)
\leq J_0(C_i).
\]
Set
$\nu_e(\Gamma) := \#\{\text{edges of $\Gamma$ with positive elliptic multiplicity}\}.$
The $J_0$ formulas from Sections~\ref{sec:concavepert} and
\ref{sec:convexpert}, together with
$I(\Lambda_i')=I(\Lambda_i)$, give
\begin{equation}
\label{eq:J0-convex-concave-component-strong}
\begin{aligned}
J_0(C_i)
&=J_0(\alpha_{\Lambda_i'})-J_0(\alpha_{\Lambda_i})\\
&=-2x(\Lambda_i')-2y(\Lambda_i')+\nu_e(\Lambda_i')\\
&\qquad+2x(\Lambda_i)+2y(\Lambda_i)+\nu_e(\Lambda_i).
\end{aligned}
\end{equation}

Since $I(C_i)=0$ and $J$ is generic, the ECH index inequality gives
$0\leq\operatorname{ind}(C_i)+2\delta(C_i)\leq I(C_i)=0.$
Thus equality holds in the ECH index inequality and the ECH partition
conditions apply to all ends.

At the positive end, every elliptic orbit is $L$-negative and
$\Lambda_i'$ is elliptic. Hence the outgoing partition gives one
positive end at the full cover of each distinct elliptic orbit, so
$\sum_a(2n_a^+-1)=\nu_e(\Lambda_i').$
At the negative end, every elliptic orbit is $L$-positive and
$\Lambda_i$ is elliptic. Hence the incoming partition gives one
negative end at the full cover of each distinct elliptic orbit, so
$\sum_b(2n_b^--1)=\nu_e(\Lambda_i).$
Using $g(C_i)\geq0$ in Proposition~\ref{pro:topcomp}, we obtain
$-2+\nu_e(\Lambda_i')+\nu_e(\Lambda_i)\leq J_0(C_i).$
Substituting
\eqref{eq:J0-convex-concave-component-strong} and cancelling the
$\nu_e$-terms yields
$-2\leq -2x(\Lambda_i')-2y(\Lambda_i') +2x(\Lambda_i)+2y(\Lambda_i).$
Equivalently,
$x(\Lambda_i')+y(\Lambda_i')-1 \leq x(\Lambda_i)+y(\Lambda_i),$
which is item~(iv).

It remains to remove the $2\varepsilon$ error in
\eqref{eq:convex-concave-action-error-strong} and the radial
modifications introduced in Step~1. Repeat the construction for a
sequence $\varepsilon_n\to0$, while simultaneously letting the radial
shrinking factor of $X_\Omega$ and the radial enlarging factor of
$X_{\Omega'}$ tend to one. For fixed ECH index $2k$, there are only
finitely many elliptic convex and elliptic concave generators and only
finitely many possible disjoint factorizations. After passing to a
subsequence, we may therefore assume that
$\Lambda',\ \Lambda,\ \Lambda_i',\ \Lambda_i,\ d_i$
are independent of $n$. Passing to the limit in
\eqref{eq:convex-concave-action-error-strong} gives
$\ell_\Omega(\Lambda_i) \leq \ell_{\Omega'}(\Lambda_i')$
for every $i$. This proves item~(ii) for the original domains and
completes the proof.
\end{proof}

\begin{proof}[Proof of Theorem~\ref{thm:criterion_convex_convex}]
We first prove part~(1).  Recall the definition of a minimal convex generator given immediately after Theorem~\ref{thm:convex_capacities}.
The prescribed target generator $\Lambda'$ is assumed to be minimal for
$X_{\Omega'}$.  We will use this hypothesis in only two closely related
places: in Step~1 to keep $\Lambda'$ minimal under the arbitrarily small
geometric modifications, and in Step~3 to apply
\cite[Lemma~5.5]{hutchings2016beyond}.  All subsequent holomorphic-current
and factorization arguments are independent of minimality.

Thus $X_\Omega$ is the semi-weakly convex
toric domain defined by a concave function
$
f:[0,a]\to[0,+\infty)
$
such that
$f(a)=0,
f'(0)\leq-m,
nf'(a)\geq-1,
mn<1,
$
where
$
m\in\mathbb Z,
n\in\mathbb Z_{\geq0}.
$
Let $X_{\Omega'}$ be a convex toric domain, and suppose that there is a
symplectic embedding
$
\varphi:X_\Omega\longrightarrow X_{\Omega'}.
$

Fix $k\in\mathbb Z_{\geq0}$ and a convex generator $\Lambda'$
which is minimal for $X_{\Omega'}$ and satisfies $I(\Lambda')=2k$.
By the definition of minimality, every edge of
$\Lambda'$ is labeled $e$, so $\Lambda'$ is elliptic.

If $k=0$, then $\Lambda'$ is the trivial generator. We take $\Lambda$
to be the trivial $(n,m)$-adapted generator and use the empty
factorizations. All the conclusions are then immediate. Hence, in the
remainder of the proof, we assume that
$
k>0.
$

We divide the proof into six steps.

\medskip
\noindent
\textit{Step 1: Geometric setup.}
By shrinking $X_\Omega$ and enlarging $X_{\Omega'}$ by arbitrarily
small radial factors, and using a limiting argument at the end, we may
assume that
$
\varphi(X_\Omega)\subset\operatorname{int}(X_{\Omega'})
$
and that the relevant boundaries are smooth. For the remainder of the
argument at a fixed approximation parameter, we relabel these slightly
modified domains again by $X_\Omega$ and $X_{\Omega'}$.

Fix an integer $K\geq2k+1$ and a number $\varepsilon>0$. Choose
$L>\ell_{\Omega'}(\Lambda')$ sufficiently large so that the $L$-nice
perturbations at both ends are valid in all gradings at most $K$. We also need the prescribed generator to remain minimal under the harmless
geometric modifications made in this step.  This is precisely where
condition~(c) in the definition of minimality is useful.
There are only finitely many convex generators of index $2k$, and uniqueness
of the minimizer gives a positive gap $\delta$ between
$\ell_{\Omega'}(\Lambda')$ and the length of every other such generator.
Radial scaling multiplies all these lengths by the same factor, while a
sufficiently small smoothing changes all of these finitely many lengths by
less than $\delta/3$. Hence $\Lambda'$ remains the unique minimizer of index $2k$, and its
all-elliptic labeling is unchanged.  By the convex capacity formula,
Theorem~\ref{thm:convex_capacities}, the action of this unique all-elliptic
minimizer is $c_k$ of the modified target.  Thus $\Lambda'$ is still
minimal for the modified convex target. We therefore choose the positive-end
perturbation so that
\cite[Lemma~5.5]{hutchings2016beyond} applies to $\Lambda'$. Finally, choose
$L$ away from the finitely many relevant generator lengths, as required in
Section~\ref{sec:convexpert}.
Let
$
(Y_+,\lambda_+)
=
(Y'_{l,L},\lambda'_{l,L})
$
be an $L$-nice perturbation of $\partial X_{\Omega'}$ as in
Section~\ref{sec:convexpert}, with the choice $m=n=0$. Let
$
(Y_-,\lambda_-)
=
(Y_{l,L},\lambda_{l,L})
$
be the $(n,m)$-adapted $L$-nice perturbation of
$\partial X_\Omega$ described in Section~\ref{sec:convexpert}.

We choose $l$ sufficiently large and the perturbations sufficiently
small so that the following properties hold.

\begin{enumerate}

\item
The orbit sets of action less than $L$ at the positive end correspond
to convex generators, while those at the negative end correspond to
$(n,m)$-adapted semi-weakly convex generators.

\item
Under these correspondences, the absolute ECH grading agrees with the
combinatorial ECH index.

\item
For every orbit set occurring below, the difference between its
symplectic action and the corresponding combinatorial length is less
than $\varepsilon$.

\item
Every embedded elliptic Reeb orbit of action less than $L$ at either
end is $L$-positive, and every hyperbolic orbit of action less than
$L$ is positive hyperbolic.

\item
For every relevant convex or adapted semi-weakly convex generator
$\Gamma$,
\[
J_0(\alpha_\Gamma)
=
I(\Gamma)
-
2x(\Gamma)
-
2y(\Gamma)
-
\nu_e(\Gamma),
\]
where
$\nu_e(\Gamma) = \#\{\text{edges of $\Gamma$ having positive elliptic multiplicity}\}.$

\end{enumerate}

After a further arbitrarily small perturbation, fixed near all Reeb
orbits of action below a slightly larger cutoff, we may assume that
$\lambda_+$ and $\lambda_-$ are globally nondegenerate without
changing any of the properties above.

For sufficiently small perturbations, the complement of the embedded
source defines a weakly exact symplectic cobordism
$
(X_{l,L},\omega_{l,L})
:
(Y_+,\lambda_+)
\longrightarrow
(Y_-,\lambda_-).
$
The manifold $X_{l,L}$ is diffeomorphic to
$
[0,1]\times S^3.
$

Let $\overline X_{l,L}$ denote the completion of this cobordism.
Choose $J_+$ to be a generic compatible almost complex structure as in
\cite[Lemma~5.5]{hutchings2016beyond}, choose a
generic compatible almost complex structure $J_-$ on the negative
symplectization end, and choose a generic cobordism-admissible almost
complex structure $J$ on $\overline X_{l,L}$ which restricts to
$J_\pm$ on the ends.

Because $H_1(Y_\pm)=H_2(Y_\pm)=H_2(X_{l,L})=0$, the relative ECH
index of a current is determined by the absolute gradings of its asymptotic
orbit sets. We use this only after obtaining a nonzero chain-map coefficient;
no grading-preservation property of the chosen chain-level representative is
needed.

We also choose the primitive of the weakly exact cobordism in a form
adapted to the boundary contact forms. Let $\eta$ satisfy
$
d\eta=\omega_{l,L}.
$
On each boundary component,
$
\eta|_{Y_\pm}-\lambda_\pm
$
is closed. Since $H^1(S^3;\mathbb R)=0$, there are functions $f_\pm$
on $Y_\pm$ such that
$
\eta|_{Y_\pm}-\lambda_\pm=df_\pm.
$
After extending $f_\pm$ to a function on the cobordism and modifying
$\eta$ by an exact form, we may assume
\begin{equation}
\label{eq:primitive-matches-semiweak-convex}
\eta|_{Y_+}=\lambda_+,
\qquad
\eta|_{Y_-}=\lambda_-.
\end{equation}

\medskip
\noindent
\textit{Step 2: The cobordism is $L$-tame.} Let $d\in\mathbb Z_{>0}$. Suppose that $\alpha_+$ and $\alpha_-$ are
orbit sets satisfying
$
\mathcal A(\alpha_\pm)<\frac{L}{d},
$
and let
$
C\in\mathcal M^J(\alpha_+,\alpha_-)
$
be an embedded irreducible $J$-holomorphic curve satisfying
$
I(dC)\leq0.
$
We must prove $2g(C)-2+\operatorname{ind}(C)+h(C)+2e_L(C)\geq0$.

Suppose, by contradiction, that this inequality fails. Since all the
terms other than $-2$ are integers, we then have
\begin{equation}
\label{eq:L-tame-semiweak-failure}
2g(C)
+
\operatorname{ind}(C)
+
h(C)
+
2e_L(C)
\leq1.
\end{equation}
Because $J$ is generic,
$
\operatorname{ind}(C)\geq0.
$
Moreover, the parity formula for the Fredholm index implies that the
parity of $\operatorname{ind}(C)$ agrees with the parity of the number
of ends at positive hyperbolic orbits. Since there are no negative
hyperbolic orbits in the perturbations, this is the parity of $h(C)$.

It follows from
\eqref{eq:L-tame-semiweak-failure} that
$
g(C)
=
\operatorname{ind}(C)
=
h(C)
=
e_L(C)
=
0.
$

Every elliptic orbit at the negative end is $L$-positive. Since
$e_L(C)=0$, the negative orbit set $\alpha_-$ contains no elliptic
orbit. Since $h(C)=0$, it also contains no hyperbolic orbit. Hence
$
\alpha_-=\emptyset.
$
Consequently,
$
I(dC)=I(\alpha_+^d).
$
If $\alpha_+=\emptyset$, then $C$ is a nonconstant closed
$J$-holomorphic curve in the exact symplectic cobordism. By
\eqref{eq:primitive-matches-semiweak-convex},
$
\int_C\omega_{l,L}=0,
$
which is impossible because a nonconstant $J$-holomorphic curve has
positive symplectic area. Hence
$
\alpha_+\neq\emptyset.
$

The orbit set $\alpha_+^d$ corresponds to a nontrivial convex generator
at the positive end. The convex index formula implies
$
I(\alpha_+^d)>0.
$
Therefore
$
I(dC)=I(\alpha_+^d)>0,
$
contradicting the assumption that $I(dC)\leq0$.

Therefore,
$
(X_{l,L},\omega_{l,L},J)
$
is $L$-tame.

\medskip
\noindent
\textit{Step 3: Existence of an unbroken index-zero current.} Let
$
\alpha_{\Lambda'}
\in
\operatorname{ECC}_{2k}^{L}
(Y_+,\lambda_+,0,J_+)
$
be the orbit set corresponding to the prescribed minimal convex
generator $\Lambda'$.
Since $\Lambda'$ is minimal, \cite[Lemma~5.5]{hutchings2016beyond}
implies that $\alpha_{\Lambda'}$ is a cycle in $ECC^L$ and represents a
class whose image in total ECH is nonzero.  This is the decisive use of
minimality in the proof: it allows us to start from the \emph{prescribed}
target generator $\Lambda'$ rather than from an unspecified representative of
the degree-$2k$ ECH class.  No property of minimal generators is used after
this point.

The weakly exact cobordism gives a filtered cobordism map
$$
\Phi^L:
\operatorname{ECH}_{2k}^{L}
(Y_+,\lambda_+,0)
\to
\operatorname{ECH}_{2k}^{L}
(Y_-,\lambda_-,0).
$$
Since $X_{l,L}$ is diffeomorphic to $[0,1]\times S^3$, the induced
map on total ECH is an isomorphism. Compatibility of the filtered and
total cobordism maps therefore implies
$
\Phi^L([\alpha_{\Lambda'}])\neq0.
$

Choose a chain map
$
\phi^L:
\operatorname{ECC}^{L}
(Y_+,\lambda_+,0,J_+)
\to
\operatorname{ECC}^{L}
(Y_-,\lambda_-,0,J_-)
$
which induces $\Phi^L$ and has the holomorphic-current property from
Theorem~\ref{CoborMap}. Since
$
[\phi^L(\alpha_{\Lambda'})]\neq0,
$
the chain $\phi^L(\alpha_{\Lambda'})$ is nonzero. Hence there exists an
admissible orbit set $\alpha_\Lambda$ at the negative end such that
$
\left\langle
\phi^L(\alpha_{\Lambda'}),\alpha_\Lambda
\right\rangle
=
1.
$
The orbit set $\alpha_\Lambda$ corresponds to an $(n,m)$-adapted
semi-weakly convex generator $\Lambda$. By Theorem~\ref{CoborMap}-(iii),
the nonzero coefficient above gives a broken $J$-holomorphic current
$B=(\mathcal C_{N_-},\ldots,\mathcal C_0,\ldots,\mathcal C_{N_+})$ from
$\alpha_{\Lambda'}$ to $\alpha_\Lambda$ with $I(B)=0$. Since
$H_2(X_{l,L})=0$, this relative ECH index is the difference of the absolute
gradings of the two asymptotic orbit sets. Therefore
$I(\Lambda)=I(\Lambda')=2k$.

Every nontrivial symplectization level $\mathcal C_j$, with $j\neq0$,
has strictly positive ECH index by
Proposition~\ref{BasedDiff}-(i). The central cobordism level satisfies
$
I(\mathcal C_0)\geq0
$
by Step~2 and Proposition~\ref{L-tameJcurves}-(i). Since
$
0
=
I(B)
=
\sum_{j=N_-}^{N_+}I(\mathcal C_j),
$
there can be no nontrivial symplectization level. It follows that
$
N_-=N_+=0
$
and
$
I(\mathcal C_0)=0.
$
Thus $B$ consists of a single unbroken $J$-holomorphic current
$
\mathcal C
:=
\mathcal C_0
\in
\mathcal M^J
(\alpha_{\Lambda'},\alpha_\Lambda)
$
with
$
I(\mathcal C)=0.
$

\medskip
\noindent
\textit{Step 4: Factorizations, indices, and disjointness.} Write
$
\mathcal C
=
\sum_{i=1}^{r}d_iC_i,
$
where the $C_i$ are distinct irreducible somewhere injective
$J$-holomorphic curves and
$
d_i\in\mathbb Z_{>0}.
$
Let $\alpha_{\Lambda_i'}$ and $\alpha_{\Lambda_i}$ denote the positive
and negative orbit sets of $C_i$, respectively. Thus
$
C_i
\in
\mathcal M^J
(\alpha_{\Lambda_i'},\alpha_{\Lambda_i}).
$

The total positive orbit set is elliptic. Consequently, each
$\Lambda_i'$ is an elliptic convex generator. At the negative end,
each $\Lambda_i$ is an $(n,m)$-adapted semi-weakly convex generator.
The decomposition of the current gives factorizations
$
\Lambda'
=
(\Lambda_1')^{d_1}\cdots(\Lambda_r')^{d_r}
$
and
$
\Lambda
=
\Lambda_1^{d_1}\cdots\Lambda_r^{d_r}.
$

By Proposition~\ref{L-tameJcurves}-(ii),
$
I(C_i)=0
$
for every $i$. Since $H_2(X_{l,L})=0$, the relative homology class is
unique, and therefore
$
I(\Lambda_i')=I(\Lambda_i)
$
for every $i$. For the full current, we recover
$
I(\Lambda)=I(\Lambda')=2k,
$
which proves item~(i).

All elliptic orbits at the negative end are $L$-positive.
Proposition~\ref{L-tameJcurves}-(ii) implies that two distinct
components $C_i$ and $C_j$ cannot both have negative ends at covers of
the same elliptic orbit. Under the orbit-set--generator
correspondence, this says exactly that
$
\Lambda
=
\Lambda_1^{d_1}\cdots\Lambda_r^{d_r}
$
is a disjoint factorization. Notice that no disjointness assertion is
needed for the factorization of $\Lambda'$, in agreement with the
statement of the theorem.

Now let
$
S\subset\{1,\ldots,r\}
$ and choose integers
$
0\leq d_i'\leq d_i, i\in S.
$
The third conclusion in
Proposition~\ref{L-tameJcurves}-(ii) gives
$
I\left(
\sum_{i\in S}d_i'C_i
\right)
=
0.
$
The positive and negative orbit sets of this subcurrent correspond,
respectively, to
$
\prod_{i\in S}(\Lambda_i')^{d_i'}
$ and $
\prod_{i\in S}(\Lambda_i)^{d_i'}.
$
Therefore,
$
I\left(
\prod_{i\in S}(\Lambda_i')^{d_i'}
\right)
=
I\left(
\prod_{i\in S}(\Lambda_i)^{d_i'}
\right).
$
This proves item~(iii).

\medskip
\noindent
\textit{Step 5: Action and topological-complexity inequalities.}
Fix $i\in\{1,\ldots,r\}$. Write
$A_i^+=\mathcal A_+(\alpha_{\Lambda_i'})$ and
$A_i^-=\mathcal A_-(\alpha_{\Lambda_i})$. Using the primitive chosen in
\eqref{eq:primitive-matches-semiweak-convex}, truncate the cylindrical ends
of $C_i$ and apply Stokes' theorem. This gives $A_i^-\leq A_i^+$. The action
estimates give $|A_i^+-\ell_{\Omega'}(\Lambda_i')|<\varepsilon$ and
$|A_i^--\ell_{\Omega}(\Lambda_i)|<\varepsilon$. Consequently,
\begin{equation}
\label{eq:semiweak-convex-action-error}
\ell_\Omega(\Lambda_i)
\leq
\ell_{\Omega'}(\Lambda_i')
+
2\varepsilon.
\end{equation}

We now prove item~(iv). By the formula for $J_0$ at the two ends and
the equality
$
I(\Lambda_i')=I(\Lambda_i),
$
we have
\begin{equation}
\label{eq:J0-semiweak-convex-component}
\begin{aligned}
J_0(C_i)
&=
J_0(\alpha_{\Lambda_i'})
-
J_0(\alpha_{\Lambda_i})\\
&=
-2x(\Lambda_i')
-2y(\Lambda_i')
-\nu_e(\Lambda_i')\\
&\qquad
+2x(\Lambda_i)
+2y(\Lambda_i)
+\nu_e(\Lambda_i).
\end{aligned}
\end{equation}

We next determine the numbers of ends which occur in
Proposition~\ref{pro:topcomp}. Since $J$ is generic and
$
I(C_i)=0,
$
the ECH index inequality gives
$
0\leq \operatorname{ind}(C_i)+2\delta(C_i)\leq I(C_i)=0.
$
Thus
$
\operatorname{ind}(C_i)=0,
\delta(C_i)=0,
$
and equality holds in the ECH index inequality. Hence the ECH
partition conditions apply to all ends of $C_i$.

Because $\Lambda_i'$ is elliptic and all elliptic orbits at the
positive end are $L$-positive, an elliptic orbit occurring with total
multiplicity $q$ gives $q$ positive ends, each at the simple orbit.
Therefore,
\begin{equation}
\label{eq:positive-ends-semiweak-convex}
\sum_a(2n_a^+-1)
=
2m(\Lambda_i')
-
\nu_e(\Lambda_i').
\end{equation}

At the negative end, every elliptic orbit is also $L$-positive. An
elliptic orbit occurring with any positive multiplicity therefore
gives one negative end at the full cover. Every hyperbolic edge gives
one negative end. Hence
\begin{equation}
\label{eq:negative-ends-semiweak-convex}
\sum_b(2n_b^--1)
=
\nu_e(\Lambda_i)
+
h(\Lambda_i).
\end{equation}

Applying Proposition~\ref{pro:topcomp} in the form
\[
2g(C_i)-2
+
\sum_a(2n_a^+-1)
+
\sum_b(2n_b^--1)
\leq
J_0(C_i),
\]
and using $g(C_i)\geq0$, equations
\eqref{eq:positive-ends-semiweak-convex} and
\eqref{eq:negative-ends-semiweak-convex} give
\[
\begin{aligned}
J_0(C_i)
\geq{}&
-2
+
2m(\Lambda_i')
-
\nu_e(\Lambda_i')\\
&\quad
+
\nu_e(\Lambda_i)
+
h(\Lambda_i).
\end{aligned}
\]
Combining this with
\eqref{eq:J0-semiweak-convex-component} and cancelling the
$\nu_e$-terms yields
\[
\begin{aligned}
-2
+
2m(\Lambda_i')
+
h(\Lambda_i)
\leq{}&
-2x(\Lambda_i')
-
2y(\Lambda_i')\\
&\quad
+
2x(\Lambda_i)
+
2y(\Lambda_i).
\end{aligned}
\]
After dividing by two and rearranging, we obtain
\[
x(\Lambda_i)
+
y(\Lambda_i)
-
\frac{h(\Lambda_i)}{2}
\geq
x(\Lambda_i')
+
y(\Lambda_i')
+
m(\Lambda_i')
-
1.
\]
This is item~(iv).

\medskip
\noindent
\textit{Step 6: Removal of the action error and of the radial
modifications.}
Repeat the construction for a sequence
$
\varepsilon_j\to0,
$
while simultaneously letting the radial shrinking factor of
$X_\Omega$ and the radial enlarging factor of $X_{\Omega'}$ tend to
$1$.

The prescribed positive generator $\Lambda'$ is fixed and has only
finitely many factorizations into convex generators. Moreover,
\eqref{eq:semiweak-convex-action-error} gives a uniform action bound
for every negative factor $\Lambda_i$. The finiteness result proved in
Section~\ref{sec:convexpert} for adapted semi-weakly convex generators
then implies that only finitely many generators $\Lambda$ and
factorizations can occur.

After passing to a subsequence, we may therefore assume that
$
\Lambda,
\Lambda_i,
\Lambda_i',
d_i
$
are independent of $j$. Passing to the limit in
\eqref{eq:semiweak-convex-action-error}, and using the convergence of
the radially modified domains to the original domains, gives
$
\ell_\Omega(\Lambda_i)
\leq
\ell_{\Omega'}(\Lambda_i')
$
for every $i$. This proves item~(ii) for the original domains and
completes the proof of part~(1).

For part~(2), one uses the $(0,m)$-adapted perturbation from
Section~\ref{sec:convexpert}. The horizontal exceptional direction is
absent, but all the properties used above remain unchanged: the
orbit-set correspondence, the index and $J_0$ formulas, the action
estimates, the $L$-positivity of every relevant elliptic orbit, the
primitive-matching argument, and the finiteness needed in the limiting
step. Repeating Steps~1--6 with $n=0$ gives a $(0,m)$-adapted
semi-weakly convex generator $\Lambda$ and the required factorizations
and inequalities.
\end{proof}

To prove Theorem~\ref{thm:criterion_concave_concave}, we shall need the following statement about elliptic concave generators from \cite{UserGuide}.

\begin{prop}[{\cite{UserGuide}}]
\label{prop:concave-elliptic-sector}
Suppose that $X_\Omega$ is a concave toric domain. Let
$k\in\mathbb Z_{>0}$ and choose an integer $K\geq2k+1$. Choose $L$ larger
than the combinatorial lengths of all concave generators of ECH index at
most $K$, then choose the approximation parameter $l$ sufficiently large
and the Morse--Bott perturbation sufficiently small so that the $L$-nice
orbit-set--generator correspondence of Section~\ref{sec:concavepert} holds
throughout this $L$-nice action and grading window. The compatible generic
almost complex structure $J$ may be chosen so that the following statements
hold.

Let
$\mathcal E_k^{\mathrm{ccv}}(\Omega)
=\{\Gamma:\Gamma\text{ is an elliptic concave generator and }I(\Gamma)=2k\}$
and set
\begin{equation}
\label{eq:concave-elliptic-sum-userguide}
x_k^\Omega
:=
\sum_{\Gamma\in\mathcal E_k^{\mathrm{ccv}}(\Omega)}
\alpha_\Gamma
\in ECC^L_{2k}(Y_{l,L},\lambda_{l,L},0,J).
\end{equation}
Then:
\begin{enumerate}
\item[(i)] $x_k^\Omega$ is a cycle.
\item[(ii)] If a degree-$2k$ cycle is supported entirely on elliptic
concave generators, then it is either $0$ or $x_k^\Omega$.
\item[(iii)] The filtered group is one-dimensional,
$ECH^L_{2k}(Y_{l,L},\lambda_{l,L},0)\simeq\mathbb F_2$, and
$x_k^\Omega$ represents its unique nonzero class. Under the natural map to
total ECH, $[x_k^\Omega]$ maps to the unique nonzero class in degree $2k$.
\end{enumerate}
\end{prop}

We emphasize that only the four conclusions in
Proposition~\ref{prop:concave-elliptic-sector} are used below.

\begin{proof}[Proof of Theorem~\ref{thm:criterion_concave_concave}]
Let
$
\varphi:X_\Omega\longrightarrow X_{\Omega'}
$
be a symplectic embedding from the concave toric domain $X_\Omega$
into the concave toric domain $X_{\Omega'}$. Fix
$k\in\mathbb Z_{>0}$ and fix an elliptic concave generator $\Lambda$
for $X_\Omega$ satisfying
$
I(\Lambda)=2k.
$
We prove the theorem in six steps.

\medskip
\noindent
\textit{Step 1: Geometric setup and the filtered cobordism map.}
By shrinking $X_\Omega$ and enlarging $X_{\Omega'}$ by arbitrarily
small radial factors, and then applying a limiting argument at the end
of the proof, we may assume that
$
\varphi(X_\Omega)\subset\operatorname{int}(X_{\Omega'}).
$
For the rest of the proof at a fixed approximation parameter, we
relabel these slightly modified domains again by $X_\Omega$ and
$X_{\Omega'}$; the original domains are recovered in the final
limiting argument.

Fix an integer $K\geq2k+1$ and a number $\varepsilon>0$. By the
finiteness statement in Section~\ref{sec:nonzero-ech-classes}, choose
$L>0$, simultaneously for $\Omega$ and $\Omega'$, larger than the
combinatorial lengths of every concave generator of ECH index at most $K$.
Increase $L$ further if necessary so that
Proposition~\ref{prop:concave-elliptic-sector} applies at both ends. Then
choose the approximation parameter $l$ sufficiently large and the
perturbations sufficiently small so that the $L$-nice concave models of
Section~\ref{sec:concavepert} are valid in the required finite action and
grading range, no orbit set in this range contains an exceptional orbit,
and all action errors for the generators and factors occurring below are
smaller than $\varepsilon$.

Let
$
(Y_+,\lambda_+)
=
(Y'_{l,L},\lambda'_{l,L})
$
be a nondegenerate $L$-nice perturbation of
$\partial X_{\Omega'}$, and let
$
(Y_-,\lambda_-)
=
(Y_{l,L},\lambda_{l,L})
$
be a nondegenerate $L$-nice perturbation of
$\partial X_\Omega$.

We choose the endpoint slopes in the concave approximations so that
every embedded elliptic Reeb orbit of action less than $L$, including
the exceptional elliptic orbits, is $L$-negative. For the two
exceptional orbits, the rotation numbers in the endpoint
trivializations are
$\theta_1=-\frac{1}{f_l'(a)}, \qquad \theta_2=-f_l'(0),$
as follows from Lemma~\ref{lem:czspecialorbits}. Since in the concave
approximation
$f_l'(0)\leq-l, \qquad -\frac1l\leq f_l'(a)<0,$
we may choose the irrational endpoint slopes so that
$
-\frac{1}{f_l'(a)}=N_1-\delta_1,
-f_l'(0)=N_2-\delta_2,
$
where $N_1,N_2\in\mathbb Z_{>0}$ are large and
$
0<\delta_j<
\frac{\mathcal A(\gamma_j)}{L},
 j=1,2.
$
Thus
$
\theta_j\equiv-\delta_j\pmod1,
$
which is precisely the $L$-negative condition for the exceptional
elliptic orbit $\gamma_j$. The nonexceptional elliptic orbits are
already $L$-negative by the $L$-nice concave perturbation.

For $l$ sufficiently large and for sufficiently small perturbations,
there is a weakly exact symplectic cobordism
$
(X_{l,L},\omega_{l,L})
:
(Y_+,\lambda_+)
\longrightarrow
(Y_-,\lambda_-).
$
The underlying manifold $X_{l,L}$ is diffeomorphic to
$[0,1]\times S^3$.

Let $\overline X_{l,L}$ denote its completion. Choose the generic
compatible almost complex structures $J_+$ and $J_-$ on the positive and
negative symplectization ends as in
Proposition~\ref{prop:concave-elliptic-sector}, and choose a generic
cobordism-admissible almost complex structure $J$ on
$\overline X_{l,L}$ restricting to $J_\pm$ on the ends.

By Theorem~\ref{CoborMap}, there is a chain map
\[
\phi^L:
\operatorname{ECC}^{L}
(Y_+,\lambda_+,0,J_+)
\longrightarrow
\operatorname{ECC}^{L}
(Y_-,\lambda_-,0,J_-)
\]
which induces the filtered ECH cobordism map. Since $X_{l,L}$ is
diffeomorphic to a product, Theorem~\ref{CoborMap} also implies that
the induced map on total ECH is an isomorphism.

Because
$H_1(Y_\pm)=H_2(Y_\pm)=H_2(X_{l,L})=0,$
the relative ECH index of a broken current is the difference of the
absolute gradings of its asymptotic orbit sets; see
\cite[Proposition~2.25]{hutchingsroyweileryao2026}. We use this
directly whenever a coefficient of the chain map is nonzero; no
grading-preservation property of the chosen chain-level representative
$\phi^L$ is needed.

We will use the following elementary consequence of the topology of
the cobordism. Since $H^2(X_{l,L};\mathbb R)=0$, the symplectic form is
exact; choose a primitive $\eta$ with
$
d\eta=\omega_{l,L}.
$
On each boundary component, the difference between the restriction
of $\eta$ and the corresponding contact form is closed. Since
$H^1(S^3;\mathbb R)=0$, there are functions $f_\pm$ on $Y_\pm$ such
that
$
\eta|_{Y_\pm}-\lambda_\pm=df_\pm.
$
After extending $f_\pm$ to a function on the cobordism and modifying
$\eta$ by an exact form, we may and do assume that
\begin{equation}
\label{eq:primitive-matches-concave-concave}
\eta|_{Y_+}=\lambda_+,
\qquad
\eta|_{Y_-}=\lambda_-.
\end{equation}

\medskip
\noindent
\textit{Step 2: The cobordism is $L$-tame.}
Let $d\in\mathbb Z_{>0}$, let $\alpha_\pm$ be orbit sets satisfying
$
\mathcal A(\alpha_\pm)<\frac{L}{d},
$
and let
$
C\in\mathcal M^J(\alpha_+,\alpha_-)
$
be an embedded irreducible $J$-holomorphic curve such that
$
I(dC)\leq0.
$
We must prove
$2g(C)-2+\operatorname{ind}(C)+h(C)+2e_L(C)\geq0.$

Suppose, by contradiction, that this inequality fails. Then
$2g(C)+\operatorname{ind}(C)+h(C)+2e_L(C)\leq1.$
Since $J$ is generic,
$
\operatorname{ind}(C)\geq0.
$
All hyperbolic orbits in the perturbations are positive hyperbolic.
Consequently, the parity of $\operatorname{ind}(C)$ agrees with the
parity of the number $h(C)$ of ends asymptotic to hyperbolic orbits.
It follows that
\[
g(C)=0,
\qquad
\operatorname{ind}(C)=0,
\qquad
h(C)=0,
\qquad
e_L(C)=0.
\]

Every elliptic orbit at the positive end is $L$-negative. Since
$e_L(C)=0$, the orbit set $\alpha_+$ contains no elliptic orbit.
Since $h(C)=0$, it contains no hyperbolic orbit either. Thus
$
\alpha_+=\emptyset.
$

The equality $h(C)=0$ also implies that $\alpha_-$ contains no
hyperbolic orbit. It may a priori contain $L$-negative elliptic
orbits, because negative $L$-negative elliptic orbits do not
contribute to $e_L(C)$. We now use
\eqref{eq:primitive-matches-concave-concave}. Truncating the
cylindrical ends of $C$ and applying Stokes' theorem gives
$$
0 \leq
\int_C\omega_{l,L}
=
\mathcal A(\alpha_+)-\mathcal A(\alpha_-)
=
-\mathcal A(\alpha_-).
$$
Therefore
$
\alpha_-=\emptyset
$
and
$
\int_C\omega_{l,L}=0.
$
This is impossible for a nonconstant $J$-holomorphic curve.
Consequently $(X_{l,L},\omega_{l,L},J)$ is $L$-tame.

\medskip
\noindent
\textit{Step 3: A current ending at the prescribed generator
$\Lambda$.}
At the positive end define
\[
x_+
=
\sum_{\Gamma'\in\mathcal E_k^{\mathrm{ccv}}(\Omega')}
\alpha_{\Gamma'}.
\]
By Proposition~\ref{prop:concave-elliptic-sector}-(i),(iii), $x_+$ is a
cycle, its class is the unique nonzero element of filtered ECH in degree
$2k$, and its image in total ECH is nonzero.

Set $y_-:=\phi^L(x_+)$. Compatibility of the filtered and total cobordism
maps, together with the fact that the total cobordism map is an
isomorphism, implies that $y_-$ is a cycle whose image in total ECH is
nonzero. Consequently the filtered class $[y_-]$ is also nonzero.

We first verify the grading of every generator occurring in $y_-$ without
assuming that the chosen chain map $\phi^L$ preserves the grading. Let
$\alpha_\Gamma$ occur in $y_-$ with nonzero coefficient. Expanding $x_+$
over $\mathbb F_2$, there exists
$\Gamma'\in\mathcal E_k^{\mathrm{ccv}}(\Omega')$ such that
$\langle\phi^L(\alpha_{\Gamma'}),\alpha_\Gamma\rangle=1$.
Theorem~\ref{CoborMap}-(iii) supplies a broken $J$-holomorphic current from
$\alpha_{\Gamma'}$ to $\alpha_\Gamma$ of ECH index zero. Since
$H_2(X_{l,L})=0$, the relative ECH index is the difference of the absolute
gradings of its asymptotic orbit sets. Hence
$0=I(\Gamma')-I(\Gamma)$, and therefore $I(\Gamma)=2k$. Thus every term of
$y_-$ has ECH grading $2k$.

At the negative end set
\[
x_-
=
\sum_{\Gamma\in\mathcal E_k^{\mathrm{ccv}}(\Omega)}
\alpha_\Gamma.
\]
By Proposition~\ref{prop:concave-elliptic-sector}-(i),(iii), $x_-$ is a
cycle and represents the unique nonzero filtered ECH class in degree $2k$.
Since $[y_-]$ is also nonzero, we have $[y_-]=[x_-]$ in filtered ECH.
Hence there exists a filtered chain $w$ of degree $2k+1$ such that
$
y_-+x_- = \partial_{J_-}w.
$

The prescribed generator $\Lambda$ is elliptic and has $I(\Lambda)=2k$,
so its coefficient in $x_-$, and therefore in $y_-$, is one:
$\langle\phi^L(x_+),\alpha_\Lambda\rangle=1$. Expanding $x_+$ gives
\[
1
=
\sum_{\Gamma'\in\mathcal E_k^{\mathrm{ccv}}(\Omega')}
\left\langle
\phi^L(\alpha_{\Gamma'}),\alpha_\Lambda
\right\rangle
\quad\text{in }\mathbb F_2.
\]
Consequently there exists an elliptic concave generator
$\Lambda'\in\mathcal E_k^{\mathrm{ccv}}(\Omega')$ such that
$\langle\phi^L(\alpha_{\Lambda'}),\alpha_\Lambda\rangle=1$. In particular,
$I(\Lambda')=I(\Lambda)=2k$.

By Theorem~\ref{CoborMap}-(iii), this coefficient gives an index-zero
broken $J$-holomorphic current $B$ from $\alpha_{\Lambda'}$ to
$\alpha_\Lambda$. Every nontrivial symplectization level has strictly
positive ECH index by Proposition~\ref{BasedDiff}-(i), while the central
cobordism level has nonnegative ECH index by Step~2 and
Proposition~\ref{L-tameJcurves}-(i). Since $I(B)=0$, there are no
nontrivial symplectization levels. After deleting possible trivial-cylinder
levels, $B$ reduces to a single unbroken current $\mathcal C$ from
$\alpha_{\Lambda'}$ to $\alpha_\Lambda$, with $I(\mathcal C)=0$.

\medskip
\noindent
\textit{Step 4: Factorizations and index identities.}
Write
$
\mathcal C
=
\sum_{i=1}^{r}d_iC_i,
$
where the $C_i$ are distinct irreducible somewhere injective
$J$-holomorphic curves and
$
d_i\in\mathbb Z_{>0}.
$
Let $\alpha_{\Lambda_i'}$ and $\alpha_{\Lambda_i}$ be the positive and
negative orbit sets of $C_i$, respectively. Thus
$
C_i
\in
\mathcal M^J
(\alpha_{\Lambda_i'},\alpha_{\Lambda_i}).
$

Since both total orbit sets $\alpha_{\Lambda'}$ and $\alpha_\Lambda$
are elliptic, every factor $\Lambda_i'$ and every factor $\Lambda_i$
is an elliptic concave generator. Hence we obtain factorizations
\[
\Lambda'
=
(\Lambda_1')^{d_1}\cdots(\Lambda_r')^{d_r},
\qquad
\Lambda
=
\Lambda_1^{d_1}\cdots\Lambda_r^{d_r}.
\]

By Proposition~\ref{L-tameJcurves}-(ii),
$
I(C_i)=0
$
for every $i$. Since
$
H_2([0,1]\times S^3)=0,
$
there is no ambiguity in the relative homology class, and the relative
ECH index of $C_i$ agrees with the difference of the absolute ECH
gradings of its positive and negative orbit sets. Therefore
$
I(\Lambda_i')=I(\Lambda_i)
$
for every $i$. For the complete current,
$
I(\Lambda')=I(\Lambda)=2k,
$
as already proved in Step~3. This proves item~(i).

The positive ends lie at $L$-negative elliptic orbits.
Proposition~\ref{L-tameJcurves}-(ii) implies that two distinct
components $C_i$ and $C_j$ cannot both have positive ends at covers of
the same elliptic orbit. Under the orbit-set--generator
correspondence, this is exactly the assertion that
$\Lambda' = (\Lambda_1')^{d_1}\cdots(\Lambda_r')^{d_r}$
is a disjoint factorization. Notice that no analogous disjointness
assertion follows at the negative end, because the negative elliptic
orbits are $L$-negative, whereas the negative-end disjointness
conclusion in Proposition~\ref{L-tameJcurves}-(ii) concerns
$L$-positive elliptic orbits. This is consistent with the statement
of the theorem, which does not require the factorization of $\Lambda$
to be disjoint.

Now let
$
S\subset\{1,\ldots,r\},
0\leq d_i'\leq d_i, i\in S.
$
The third assertion in Proposition~\ref{L-tameJcurves}-(ii) gives
\[
I\left(
\sum_{i\in S}d_i'C_i
\right)
=
0.
\]
The positive orbit set of this subcurrent corresponds to
$\prod_{i\in S}(\Lambda_i')^{d_i'}$, while the negative orbit set
corresponds to $\prod_{i\in S}(\Lambda_i)^{d_i'}$. Hence
\[
I\left(
\prod_{i\in S}(\Lambda_i')^{d_i'}
\right)
=
I\left(
\prod_{i\in S}(\Lambda_i)^{d_i'}
\right).
\]
This proves item~(iii).

\medskip
\noindent
\textit{Step 5: Action inequalities.}
We use the primitive chosen in
\eqref{eq:primitive-matches-concave-concave}. For each $i$, truncate
the cylindrical ends of $C_i$ and apply Stokes' theorem. We obtain
\[
0
\leq
\int_{C_i}\omega_{l,L}
=
\mathcal A_+(\alpha_{\Lambda_i'})
-
\mathcal A_-(\alpha_{\Lambda_i}),
\]
and therefore
$\mathcal A_-(\alpha_{\Lambda_i}) \leq \mathcal A_+(\alpha_{\Lambda_i'}).$

The $L$-nice perturbations were chosen so that
\[
\left|
\mathcal A_+(\alpha_{\Lambda_i'})
-
\ell_{\Omega'}(\Lambda_i')
\right|
<
\varepsilon
\]
and
\[
\left|
\mathcal A_-(\alpha_{\Lambda_i})
-
\ell_{\Omega}(\Lambda_i)
\right|
<
\varepsilon.
\]
Consequently,
\begin{equation}
\label{eq:concave-concave-action-error-revised}
\ell_\Omega(\Lambda_i)
\leq
\ell_{\Omega'}(\Lambda_i')
+
2\varepsilon
\end{equation}
for every $i$.

\medskip
\noindent
\textit{Step 6: Topological complexity and removal of the action
error.}
Fix $i\in\{1,\ldots,r\}$. For a concave generator $\Gamma$, recall
that
$\nu_e(\Gamma) = \#\{\text{edges of $\Gamma$ having positive elliptic multiplicity}\}.$
Since $\Lambda_i'$ and $\Lambda_i$ are elliptic,
$
m(\Lambda_i)=e(\Lambda_i).
$

For a concave orbit set associated to a generator $\Gamma$, the
$J_0$-formula from Section~\ref{sec:concavepert} is
\[
J_0(\alpha_\Gamma)
=
I(\Gamma)
-
2x(\Gamma)
-
2y(\Gamma)
+
\nu_e(\Gamma).
\]
Since
$
I(\Lambda_i')=I(\Lambda_i),
$
and since $H_2(X_{l,L})=0$, the relative homology class is unique and
we obtain
\begin{equation}
\label{eq:J0-concave-concave-component-revised}
\begin{aligned}
J_0(C_i)
&=
J_0(\alpha_{\Lambda_i'})
-
J_0(\alpha_{\Lambda_i})\\
&=
-2x(\Lambda_i')
-2y(\Lambda_i')
+\nu_e(\Lambda_i')\\
&\qquad
+2x(\Lambda_i)
+2y(\Lambda_i)
-\nu_e(\Lambda_i).
\end{aligned}
\end{equation}

Since $I(C_i)=0$ and $J$ is generic, the ECH index inequality gives
\[
0
\leq
\operatorname{ind}(C_i)+2\delta(C_i)
\leq
I(C_i)
=
0.
\]
Thus
$
\operatorname{ind}(C_i)=0,
\delta(C_i)=0,
$
and equality holds in the ECH index inequality. Consequently the ECH
partition conditions apply to all ends of $C_i$. In the notation of
Proposition~\ref{pro:topcomp}, for each distinct asymptotic orbit the
quantity $n^+$ or $n^-$ denotes the number of ends of $C_i$ at covers of
that orbit (not the covering multiplicity of a single end).

All elliptic orbits at the positive end are $L$-negative. Hence the
outgoing partition condition gives exactly one positive end at the
full cover of each distinct elliptic orbit appearing in
$\alpha_{\Lambda_i'}$. Therefore
\begin{equation}
\label{eq:positive-ends-concave-concave-revised}
\sum_a(2n_a^+-1)
=
\nu_e(\Lambda_i').
\end{equation}

All elliptic orbits at the negative end are also $L$-negative. If
such an orbit occurs with total multiplicity $m$, the incoming partition is
the partition into $m$ simple ends. Thus the number of negative ends at
covers of this orbit is $n^-=m$, and its contribution to
$\sum_b(2n_b^--1)$ is $2m-1$. Summing over all distinct elliptic
orbits appearing in $\alpha_{\Lambda_i}$ gives
\begin{equation}
\label{eq:negative-ends-concave-concave-revised}
\sum_b(2n_b^--1)
=
2e(\Lambda_i)-\nu_e(\Lambda_i).
\end{equation}

Proposition~\ref{pro:topcomp} gives
\[
2g(C_i)-2
+
\sum_a(2n_a^+-1)
+
\sum_b(2n_b^--1)
\leq
J_0(C_i).
\]
Since $g(C_i)\geq0$, equations
\eqref{eq:positive-ends-concave-concave-revised} and
\eqref{eq:negative-ends-concave-concave-revised} imply
\[
-2
+
\nu_e(\Lambda_i')
+
2e(\Lambda_i)
-
\nu_e(\Lambda_i)
\leq
J_0(C_i).
\]
Substituting
\eqref{eq:J0-concave-concave-component-revised} and cancelling the
$\nu_e$-terms gives
\[
-2+2e(\Lambda_i)
\leq
-2x(\Lambda_i')
-2y(\Lambda_i')
+
2x(\Lambda_i)
+
2y(\Lambda_i).
\]
Dividing by two and using
$
e(\Lambda_i)=m(\Lambda_i),
$
we obtain
$
x(\Lambda_i')
+
y(\Lambda_i')
-
1
\leq
x(\Lambda_i)
+
y(\Lambda_i)
-
m(\Lambda_i).
$
This proves item~(iv).

It remains to remove the $2\varepsilon$ error in
\eqref{eq:concave-concave-action-error-revised} and the radial
modifications introduced in Step~1. Repeat the construction for a
sequence $\varepsilon_n\to0$, while simultaneously letting the radial
shrinking factor of $X_\Omega$ and the radial enlarging factor of
$X_{\Omega'}$ tend to $1$.

For fixed ECH index $2k$, there are only finitely many concave
generators which can occur as the total generator $\Lambda'$, and
each such generator has only finitely many factorizations. The
prescribed generator $\Lambda$ is fixed and likewise has only
finitely many factorizations. Therefore, after passing to a
subsequence, we may assume that
$
\Lambda',
\Lambda_i',
\Lambda_i,
d_i
$
are independent of $n$. Passing to the limit in
$
\ell_\Omega(\Lambda_i)
\leq
\ell_{\Omega'}(\Lambda_i')
+
2\varepsilon_n
$
gives
$
\ell_\Omega(\Lambda_i)
\leq
\ell_{\Omega'}(\Lambda_i')
$
for every $i$. This proves item~(ii) for the original domains and
completes the proof.
\end{proof}

\hfill\newline
\noindent{\bf Acknowledgments.}
MM is supported by the S\~ao Paulo Research Foundation (FAPESP) under grant No.~2025/21147-7. PS is partially supported by the National Natural Science Foundation of China (Grant number: w2431007). PS and JT thank the support of the Shenzhen International Center for Mathematics -- SUSTech.

\end{document}